\documentclass{article}
\usepackage[T1]{fontenc}
\usepackage{amsthm, fullpage}
\usepackage{amsmath}
\usepackage{amssymb}
\usepackage{amsfonts}
\usepackage{eucal}
\usepackage[all]{xy}
\usepackage[frenchb, english]{babel}
\usepackage{pslatex}
\usepackage[pagebackref]{hyperref}

\newtheorem{prop}{Proposition}[subsection]
\newtheorem{theo}[prop]{Théor\`eme}
\newtheorem*{theo**}{Théorème}
\newtheorem{coro}[prop]{Corollaire}
\newtheorem*{conj*}{Conjecture}
\newtheorem{lemm}[prop]{Lemme}
\newtheorem{lemm*}{Lemme}[prop]

\theoremstyle{definition}

\newtheorem{vide}[prop]{}
\newtheorem{defi}[prop]{Définition}
\newtheorem*{defi*}{Définition}

\theoremstyle{remark}
\newtheorem{rema}[prop]{Remarques}

\newtheorem{nota}[prop]{Notations}

\numberwithin{equation}{prop}

\newcommand{\riso}{ \overset{\sim}{\longrightarrow}\, }
\newcommand{\liso}{ \overset{\sim}{\longleftarrow}\, }

\renewcommand{\AA}{{\mathcal{A}}}

\newcommand{\FF}{{\mathcal{F}}}
\newcommand{\B}{{\mathcal{B}}}

\newcommand{\E}{{\mathcal{E}}}
\newcommand{\G}{{\mathcal{G}}}

\newcommand{\M}{{\mathcal{M}}}

\newcommand{\D}{{\mathcal{D}}}

\newcommand{\PP}{{\mathcal{P}}}

\renewcommand{\O}{{\mathcal{O}}}

\newcommand{\V}{\mathcal{V}}

\newcommand{\ZZ}{\mathcal{Z}}
\newcommand{\X}{\mathfrak{X}}

\newcommand{\U}{\mathfrak{U}}

\renewcommand{\L}{\mathbb{L}}
\newcommand{\R}{\mathbb{R}}
\newcommand{\Q}{\mathbb{Q}}
\newcommand{\Z}{\mathbb{Z}}
\newcommand{\N}{\mathbb{N}}

\newcommand{\hdag}{  \phantom{}{^{\dag} }    }

\begin{document}

\title{Sur la préservation de la cohérence par image inverse extraordinaire par une immersion fermée}
\author{Daniel Caro} 
\date{}
\maketitle

\begin{abstract}
 
Let $\mathcal{V}$ be a complete discrete valuation ring of unequal characteristic with perfect residue field, 
$u\colon \mathcal{Z} \hookrightarrow \mathfrak{X}$ be a closed immersion of smooth, quasi-compact, separated formal schemes over $\mathcal{V}$, 
$T$ be a divisor of $X$ such that $U:= T \cap Z$ is a divisor of $Z$, 
$\mathfrak{D}$ a strict normal crossing divisor of $\mathfrak{X}$ such that $u ^{-1} (\mathfrak{D})$ is a strict normal crossing divisor of 
$\mathcal{Z}$. We pose $\mathfrak{X} ^{\sharp}:= (\mathfrak{X}, \mathfrak{D})$,
$\mathcal{Z} ^{\sharp}:= (\mathcal{Z}, u ^{-1}\mathfrak{D})$
and $u ^{\sharp}\colon \mathcal{Z} ^{\sharp} \hookrightarrow \mathfrak{X} ^{\sharp}$
the exact closed immersion of smooth logarithmic formal schemes over $\V$.
In Berthelot's theory of arithmetic $\mathcal{D}$-modules,
we work with the inductive system of sheaves of rings 
$\smash{\widehat{\mathcal{D}}} _{\mathfrak{X} ^{\sharp}} ^{(\bullet)} (T): =
(\smash{\widehat{\mathcal{D}}} _{\mathfrak{X} ^{\sharp}} ^{(m)}(T))_{m\in \N}$,
where $\smash{\widehat{\mathcal{D}}} _{\mathfrak{X} ^{\sharp}} ^{(m)}$ is the $p$-adic completion of the ring
of differential operators of level $m$ over $\mathfrak{X} ^{\sharp}$ 
and where $T$ means that we add overconvergent singularities along the divisor $T$.
Moreover, Berthelot introduced the sheaf 
$\mathcal{D} ^{\dag} _{\mathfrak{X} ^{\sharp}} (\hdag T) _{\mathbb{Q}}:=\underset{\underset{m}{\longrightarrow}}{\lim}\,
\smash{\widehat{\mathcal{D}}} _{ \mathfrak{X} ^{\sharp}} ^{(m)} (T) \otimes _{\mathbb{Z} }\mathbb{Q}$ 
of differential operators over $\mathfrak{X} ^{\sharp}$ of finite level with overconvergent singularities along $T$.
Let $\mathcal{E} ^{(\bullet)} \in \smash{\underrightarrow{LD}} ^{\mathrm{b}} _{\mathbb{Q}, \mathrm{coh}} 
( \smash{\widehat{\mathcal{D}}} _{\mathfrak{X} ^{\sharp}} ^{(\bullet)} (T))$
and 
$\mathcal{E} := \underrightarrow{\lim}  ~ (\mathcal{E} ^{(\bullet)}) $
the corresponding objet of $D ^{\mathrm{b}} _{\mathrm{coh}}( \smash{\mathcal{D}} ^\dag _{\mathfrak{X} ^{\sharp}}(\hdag T) _{\mathbb{Q}} )$. 
In this paper, we study sufficient conditions on $\mathcal{E}$
so that if $u ^{\sharp !} (\mathcal{E}) \in D ^{\mathrm{b}} _{\mathrm{coh}}( \smash{\mathcal{D}} ^\dag _{\mathcal{Z} ^{\sharp}}(\hdag U) _{\mathbb{Q}} )$
then 
$u ^{\sharp (\bullet) !} (\mathcal{E} ^{(\bullet)}) \in \smash{\underrightarrow{LD}} ^{\mathrm{b}} _{\mathbb{Q}, \mathrm{coh}} ( \smash{\widehat{\mathcal{D}}} _{\mathcal{Z} ^{\sharp}} ^{(\bullet)} (U))$.
For instance, we check that this is the case when 
$\mathcal{E}$ is a coherent $ \smash{\mathcal{D}} ^\dag _{\mathfrak{X}^{\sharp}}(\hdag T) _{\mathbb{Q}}$-module
such that the cohomological spaces of 
$u ^{\sharp !} (\mathcal{E})$ are isocrysals on $\mathcal{Z} ^{\sharp}$ overconvergent along $U$. 

\end{abstract}

\selectlanguage{frenchb}

\tableofcontents
\date

\section*{Introduction}

Soit $\V$ un anneau de valuation discrète complet d'inégales caractéristiques $(0,p)$, 
de corps résiduel parfait et de corps des fractions $K$.
Soient $u\colon \ZZ  \hookrightarrow \mathfrak{X}$ une immersion fermée de $\V$-schémas formels séparés, quasi-compacts et lisses, 
$T$ un diviseur de $X$ tel que $U:= T \cap Z$ soit un diviseur de $Z$.
Soit $\mathfrak{D}$ un diviseur à croisements normaux strict de $\mathfrak{X}$ tel que 
$u ^{-1} (\mathfrak{D})$ soit un diviseur à croisements normaux strict de
$\mathcal{Z}$. On pose $\mathfrak{X} ^{\sharp}:= (\mathfrak{X}, \mathfrak{D})$,
$\mathcal{Z} ^{\sharp}:= (\mathcal{Z}, u ^{-1}\mathfrak{D})$
et $u ^{\sharp}\colon \mathcal{Z} ^{\sharp} \hookrightarrow \mathfrak{X} ^{\sharp}$
l'immersion fermée exacte de schémas formels logarithmiques lisses sur $\V$.
Pour simplifier la présentation de cette introduction, 
supposons que $u$ soit de codimension pure égale à $1$.
Pour tout entier $m\in \N$, 
on note
$\smash{\widehat{\D}} _{\X ^{\sharp}} ^{(m)} (T):=
\widehat{\B} ^{(m)} _{\X} ( T)  \smash{\widehat{\otimes}} _{\O _{\X}} \smash{\widehat{\D}} _{\X ^{\sharp}} ^{(m)}$, 
où $\widehat{\B} ^{(m)} _{\X} ( T) $ désigne les faisceaux d'anneaux construits par Berthelot dans
\cite[4.2.3]{Be1} et 
$\smash{\D} _{\X ^{\sharp}} ^{(m)}$ est le faisceau des opérateurs différentiels de niveau $m$ sur $\X$
(voir \cite[2.2]{Be1} puis sa version logarithmique dans \cite{these_montagnon}),
le chapeau signifiant la complétion $p$-adique.
On dispose de plus des morphismes canoniques de changement de niveaux
$\smash{\widehat{\D}} _{\X ^{\sharp}} ^{(m)} (T)
\to 
\smash{\widehat{\D}} _{\X ^{\sharp}} ^{(m+1)} (T)$ (voir \cite{these_montagnon} ou pour sa version non logarithmique \cite{Be1}), ce qui donne le 
système inductif d'anneaux 
$\smash{\widehat{\D}} _{\X ^{\sharp}} ^{(\bullet)} (T): =
(\smash{\widehat{\D}} _{\X ^{\sharp}} ^{(m)} (T))_{m\in \N}$. 
Berthelot construit le faisceau des opérateurs différentiels de niveau fini
en posant 
$\D ^{\dag} _{ \X ^{\sharp}}(\hdag T) _{\Q}:=\underset{\underset{m}{\longrightarrow}}{\lim}\,
\smash{\widehat{\D}} _{ \X ^{\sharp}} ^{(m)} (T) \otimes _{\Z }\Q$.
Par tensorisation par $\Q$ et passage à la limite sur le niveau, on obtient le foncteur noté
$\underrightarrow{\lim}\colon 
D ^{\mathrm{b}} ( \smash{\widehat{\D}} _{\X ^{\sharp}} ^{(\bullet)} (T))
\to 
D ^{\mathrm{b}} (\D ^{\dag} _{ \X ^{\sharp}}(\hdag T) _{\Q})$.
Afin d'obtenir un foncteur pleinement fidèle 
qui factorise ce foncteur $\underrightarrow{\lim}$, 
Berthelot a introduit la catégorie 
$\smash{\underrightarrow{LD}} ^{\mathrm{b}} _{\Q}
 ( \smash{\widehat{\D}} _{\X ^{\sharp}} ^{(\bullet)} (T))$ qui est une localisation 
 de 
$D ^{\mathrm{b}} ( \smash{\widehat{\D}} _{\X ^{\sharp}} ^{(\bullet)} (T))$.
Il a défini la sous-catégorie pleine des complexes cohérents de 
$\smash{\underrightarrow{LD}} ^{\mathrm{b}} _{\Q}
 ( \smash{\widehat{\D}} _{\X ^{\sharp}} ^{(\bullet)} (T))$
qu'il note
$\smash{\underrightarrow{LD}} ^{\mathrm{b}} _{\Q, \mathrm{coh}} ( \smash{\widehat{\D}} _{\X ^{\sharp}} ^{(\bullet)} (T))$.
Il a alors établi que le foncteur 
$\underrightarrow{\lim}$
induit l'équivalence de catégories
\begin{equation}
\notag
(*)\hspace{1cm}
\underrightarrow{\lim} 
\colon 
\smash{\underrightarrow{LD}} ^{\mathrm{b}} _{\Q, \mathrm{coh}} ( \smash{\widehat{\D}} _{\X ^{\sharp}} ^{(\bullet)} (T))
\cong
D ^{\mathrm{b}} _{\mathrm{coh}}( \smash{\D} ^\dag _{\X ^\sharp} (\hdag T) _{\Q} ).
\end{equation}

Soit
$\E ^{(\bullet)}$ est un objet 
de $\smash{\underrightarrow{LD}} ^{\mathrm{b}} _{\Q, \mathrm{coh}} ( \smash{\widehat{\D}} _{\X ^{\sharp}} ^{(\bullet)} (T))$
et 
$\E := \underrightarrow{\lim}  ~ (\E ^{(\bullet)}) $
l'objet de $D ^{\mathrm{b}} _{\mathrm{coh}}( \smash{\D} ^\dag _{\X ^{\sharp}}(\hdag T) _{\Q} )$ correspondant. 
On dispose de $u ^{\sharp (\bullet)!} (\E ^{(\bullet)}) $ l'image inverse extraordinaire de 
$\E ^{(\bullet)}$ par $u$ et  
de
$\R \underline{\Gamma} ^{\dag} _Z (\E ^{(\bullet)}) $ le foncteur cohomologique local à support strict
dans $Z$ de $\E ^{(\bullet)}$.
Ces foncteurs s'étendent naturellement à 
$D ^{\mathrm{b}} _{\mathrm{coh}}(\smash{\D} ^\dag _{\X ^{\sharp}}(\hdag T) _{\Q})$ 
et sont compatibles à l'équivalence de catégories (*) ci-dessus, i.e., 
on bénéficie des isomorphismes canoniques fonctoriels en $\E^{(\bullet)}$ de la forme:
$\R \underline{\Gamma} ^{\dag} _Z (\E)
\riso 
\underrightarrow{\lim}  ~\circ \R \underline{\Gamma} ^{\dag} _Z (\E ^{(\bullet)})$
et
$u ^{!\sharp} (\E)
\riso 
\underrightarrow{\lim}  ~ \circ u ^{\sharp (\bullet)!} (\E ^{(\bullet)})$.
Comme conséquence immédiate de \cite{caro-stab-sys-ind-surcoh},
on vérifie que les trois propriétés $u ^{\sharp (\bullet)!} (\E ^{(\bullet)}) \in 
\smash{\underrightarrow{LD}} ^{\mathrm{b}} _{\Q, \mathrm{coh}} ( \smash{\widehat{\D}} _{\ZZ ^{\sharp}}  ^{(\bullet)} (U))$, 
$\R \underline{\Gamma} ^{\dag} _Z (\E ^{(\bullet)}) \in 
\smash{\underrightarrow{LD}} ^{\mathrm{b}} _{\Q, \mathrm{coh}} ( \smash{\widehat{\D}} _{\X ^{\sharp}} ^{(\bullet)} (T))$
et 
$\R \underline{\Gamma} ^{\dag} _Z (\E) \in D ^{\mathrm{b}} _{\mathrm{coh}}( \smash{\D} ^\dag _{\X ^{\sharp}} ( T) _{\Q} )$
sont équivalentes (voir la preuve de \ref{5assert-eq}).
De plus, 
il est évident que si 
$u ^{\sharp (\bullet)!} (\E ^{(\bullet)}) \in 
\smash{\underrightarrow{LD}} ^{\mathrm{b}} _{\Q, \mathrm{coh}} ( \smash{\widehat{\D}} _{\ZZ ^{\sharp}}  ^{(\bullet)} (U))$
alors 
$u ^{!\sharp} (\E)\in D ^{\mathrm{b}} _{\mathrm{coh}}( \smash{\D} ^\dag _{\ZZ ^\sharp}(\hdag U) _{\Q} )$. 
La réciproque est loin d'être claire. 
La raison est que 
pour tout objet $\FF ^{(\bullet)} $
de 
$\smash{\underrightarrow{LD}} ^{\mathrm{b}} _{\Q} ( \smash{\widehat{\D}} _{\ZZ ^{\sharp}}  ^{(\bullet)} (U))$,
la propriété que
$\underrightarrow{\lim}  ~ (\FF ^{(\bullet)})\in D ^{\mathrm{b}} _{\mathrm{coh}}( \smash{\D} ^\dag _{\ZZ ^\sharp}(\hdag U) _{\Q} )$ 
n'implique pas en général que $\FF ^{(\bullet)}\in 
\smash{\underrightarrow{LD}} ^{\mathrm{b}} _{\Q, \mathrm{coh}} ( \smash{\widehat{\D}} _{\ZZ ^{\sharp}}  ^{(\bullet)} (U))$.
Lorsque $\E$ est un $ \smash{\D} ^\dag _{\X ^\sharp} (\hdag T) _{\Q} $-module cohérent,
nous nous intéressons dans ce papier à cette réciproque.
Nous prouvons en particulier que si  
les espaces de cohomologies de $u ^{!\sharp} (\E)$ sont des isocristaux sur $\ZZ$
surconvergent le long d'un diviseur de $U$, alors 
$u ^{\sharp (\bullet)!} (\E ^{(\bullet)}) \in \smash{\underrightarrow{LD}} ^{\mathrm{b}} _{\Q, \mathrm{coh}} ( \smash{\widehat{\D}} _{\ZZ ^{\sharp}}  ^{(\bullet)}(U))$.
\bigskip

Précisons à présent le contenu de ce papier. 
Dans le premier chapitre, nous donnons quelques préliminaires topologiques concernant
les $K$-espaces topologiques localement convexes. Nous rappelons notamment la définition des espaces de type LB
et nous reprenons quelques points sur les produits tensoriels complétés
de modules localement convexes dans le contexte qui nous sera utile dans la suite de ce travail. 
Dans le deuxième chapitre, nous munissons naturellement les isocristaux surconvergents et le faisceau
des opérateurs différentiels de niveau fini à singularités surconvergentes d'une structure canonique d'espace de type LB.
Après quelques propriétés topologiques sur les foncteurs images directes et images inverses extraordinaire par une immersion fermée,
nous établissons dans le dernier chapitre le résultat principal décrit en début d'introduction de ce papier. 
Nous finissons par des applications du théorème principale aux log-isocristaux surconvergents 
satisfaisant certaines propriétés de type non-Liouville.
Les résultats que l'on déduit généralisent les propositions \cite[1.3.13, 2.2.9, 2.3.4]{caro-Tsuzuki}
et nous obtenons en fait des preuves plus simples (on utilise néanmoins \cite[1.3.13]{caro-Tsuzuki} et l'on démontre le reste).
L'une de ces généralisations est, grâce à une remarque de \cite{AC-weil2} (un travail en commun avec Tomoyuki Abe, 
plus précisément, voir la remarque \ref{AC}), 
de ne pas supposer que l'on dispose d'une rétraction lisse $\X \to \ZZ$ de $u$ 
(hypothèse qui apparaît par exemple dans le théorème
\cite[1.3.13]{caro-Tsuzuki}).

\subsection*{Remerciement}
Je remercie Tomoyuki Abe pour ses commentaires concernant 
le fait que les $\D$-modules arithmétiques cohérents sont de type LB
et 
notre travail en commun 
qui a conduit à la remarque  \ref{AC}.

\section*{Notations}
Dans ce papier, 
on désigne par $\V$ un anneau de valuation discrète complet d'inégales caractéristiques $(0,p)$, 
$k$ son corps résiduel supposé parfait, $K$ son corps des fractions et $\pi$ une uniformisante. 
Les faisceaux seront notés par des lettres calligraphiques, 
leurs sections globales par la lettre droite associée. 
Les modules sont par défaut à gauche. 
On notera avec des chapeaux les complétions $p$-adiques et si 
$\E$ est un faisceau en groupes abéliens alors on posera $\E _{\Q}:= \E \otimes _{\Z} \Q$. 
Soient $\AA$ un faisceau d'anneaux sur un espace topologique $X$.
Si $*$ est l'un des symboles $+$, $-$, ou $\mathrm{b}$, $D ^* ( \AA )$ désigne
la catégorie dérivée des complexes de $\AA$-modules (à gauche) vérifiant les conditions correspondantes d'annulation
des faisceaux de cohomologie. Lorsque l'on souhaite préciser entre droite et gauche, on précise alors comme suit
$D ^* ( \overset{ ^\mathrm{g}}{}\AA )$ ou $D ^* ( \AA \overset{ ^\mathrm{d}}{})$.
On note $D ^{\mathrm{b}} _{\mathrm{coh}} ( \AA )$
la sous-catégorie pleine de $D  ( \AA )$
des complexes à cohomologie cohérente et bornée.
On suppose (sans nuire à la généralité) que tous les
$k$-schémas sont réduits et on pourra confondre les diviseurs avec leur support.
Les $\V$-schémas formels seront indiqués par des lettres calligraphiques ou gothiques 
et leur fibre spéciale par la lettre droite correspondante.

\section{Préliminaires topologiques}
Notons $\mathfrak{C}$ la catégorie des $K$-espaces vectoriels topologiques localement convexes.
Notons $\mathfrak{D}$ la sous-catégorie pleine de $\mathfrak{C}$ des $K$-espaces séparés et complets. 
On remarque qu'un morphisme surjectif $ V \to V''$ de $\mathfrak{C}$ est le conoyau de son noyau dans $\mathfrak{C}$
si et seulement si $V''$ est muni de la topologie quotient. 

\subsection{Espaces de type $LB$.}
Nous agglomérons ce dont nous aurons besoin sur les $K$-espaces de type $LB$, 
surtout du lemme \ref{qsepLBestLB} mais aussi de sa preuve (voir l'étape $2$ de la preuve de \ref{theo-u^*coh}).
\begin{vide}
Soit $(V _i ) _{i\in I}$ un système inductif filtrant de $\mathfrak{C}$.
Posons $V := \underrightarrow{\lim} _i \, V _i$ la limite inductive calculée dans $\mathfrak{C}$.
En tant que $K$-espace vectoriel, $V$ est la limite inductive de $(V _i ) _{i\in I}$ calculée 
dans la catégorie des $K$-espaces vectoriels. La topologie localement convexe sur $V$ est la plus fine rendant continue
tous les morphismes canoniques $V _i \to V$.

\end{vide}

\begin{rema}
\label{im-dense}
Soient 
$(V _i ) _{i\in I}$ et $(W _i ) _{i\in I}$ deux systèmes inductifs filtrants de $\mathfrak{C}$,
$f _i \colon V _i \to W _i$ une famille compatible de morphismes de $\mathfrak{C}$
et 
$f \colon 
\underrightarrow{\lim} _i \, V _i
\to 
\underrightarrow{\lim} _i  \, W _i
$
le morphisme de $\mathfrak{C}$ induit par passage à la limite inductive. 
Si pour tout $i$ l'image de $f _i$ est dense dans $W _i$ alors l'image de
$f $
est dense. 

En effet, si $F$ est un fermé de $\underrightarrow{\lim} _i  \, W _i$
contenant l'image de $f$, alors l'image inverse de $F$ sur $W _i$ est un fermé contenant
l'image de $f _i$ qui est dense dans $W _i$. La flèche canonique 
$W _i \to \underrightarrow{\lim} _i  \, W _i$ se factorise donc toujours
via $W _i \to F$. D'où $F = \underrightarrow{\lim} _i  \, W _i$.
\end{rema}

\begin{lemm}
\label{lim-surj-strict}
Soient 
$(V _i ) _{i\in I}$ et $(W _i ) _{i\in I}$ deux systèmes inductifs filtrants de $\mathfrak{C}$,
$f _i \colon V _i \to W _i$ une famille compatible de morphismes surjectifs, stricts de $\mathfrak{C}$
et 
$f \colon 
\underrightarrow{\lim} _i \, V _i
\to 
\underrightarrow{\lim} _i  \, W _i
$
le morphisme de $\mathfrak{C}$ induit par passage à la limite inductive. 
Alors $f$ est un morphisme surjectif strict.
\end{lemm}

\begin{proof}
La surjectivité de $f$ est déjà connue. 
Posons $W:= \underrightarrow{\lim} _i  \, W _i$.
Comme le morphisme $f$ est continu et surjectif, 
la propriété que $f$ soit strict est alors équivalente à la propriété
suivante: tout morphisme $g \colon W \to W'$ tel que $g \circ f $ soit continu 
est lui-même continu. Soit $g \colon W \to W'$ tel que $ g \circ f $ soit continu.
Notons $g _{i}\colon W _i \to W'$ le composé du morphisme canonique 
$W _i \to W$ avec $g$. 
Comme $f _i$ est surjectif et strict, 
comme $ g _i \circ f _i \colon V _{i} \to W'$ 
est continu (car composé de $V _i \to V $ avec $g \circ f$), 
les morphismes $g _i$ sont alors continus. 
D'après la propriété universelle de la limite inductive, 
$g$ est donc aussi continu.
\end{proof}

\begin{lemm}
\label{LB-mNm}
Soient $(V _i ) _{i\in I}$ un système inductif filtrant de $\mathfrak{C}$
et $J$ une partie cofinale de $I$.
L'isomorphisme $K$-linaire canonique
$\underrightarrow{\lim} _{j\in J}  V _j
\riso 
\underrightarrow{\lim} _{i\in I}  V _i$ 
est alors un homéomorphisme.
\end{lemm}

\begin{proof}
Les morphismes canoniques $K$-linéaires et continues
$\underrightarrow{\lim} _{j\in J}  V _j
\to 
\underrightarrow{\lim} _{i\in I}  V _i$ 
et
$\underrightarrow{\lim} _{i\in I}  V _i
\to 
\underrightarrow{\lim} _{j\in J}  V _j$ 
commutent au foncteur {\og oubli de la topologie\fg}.
Ce sont donc des bijections. 
\end{proof}

\begin{defi}
\label{defi-LB}
Un $K$-espace de type $LB$ est un $K$-espace localement convexe séparé $V$ 
tel qu'il existe, pour tout entier $m \in \N$, des morphismes continues de $K$-espaces de Banach 
$V _{m} \to V _{m+1}$ et un homéomorphisme
de la forme 
$\underrightarrow{\lim} _m  \, V _m \riso V$.
\end{defi}

\begin{rema}
Dans la définition de $K$-espace de type $LB$ de \ref{defi-LB} et avec ses notations, il n'est pas restrictif 
de supposer que les morphismes $V _m \to V _{m+1}$ soit injectifs. 
En effet, si on note $j _m \colon V _m \to V$, $W _m := V _m / \ker j _m$ muni de la topologie quotient, 
topologie qui 
en fait un $K$-espace de Banach (car $V$ est séparé donc $W _m$ est un quotient séparé d'un $K$-espace de Banach), 
on vérifie par propriété universelle que les morphismes $K$-linéaires canoniques réciproques 
$\underrightarrow{\lim} _m  V _m \to 
\underrightarrow{\lim} _m  W _m$
et 
$\underrightarrow{\lim} _m  W_m\to 
\underrightarrow{\lim} _m  V _m$ 
sont continus. 
\end{rema}

\begin{lemm}
\label{qsepLBestLB}
Un quotient séparé d'un espace de type $LB$ est un espace de type $LB$.
\end{lemm}

\begin{proof}
Pour tout entier $m \in \N$, donnons-nous des monomorphismes continues de $K$-espaces de Banach 
$V _{m} \to V _{m+1}$.
On note
$V := \underrightarrow{\lim} _m  \, V _m $
et $i _m \colon V _m \hookrightarrow V$ les monomorphismes continues canoniques.
Soit $G:=V /W$ un quotient de $V$ qui soit séparé.
Notons $G ^{(m)}:=  V _m / i _m ^{-1}(W) $ muni de la topologie quotient, 
i.e. telle que la surjection canonique 
$\pi _m \colon  V _m
\twoheadrightarrow
G ^{(m)}$
soit stricte. 
Notons
$j _m \colon G ^{(m)} \to G$ l'injection canonique. 
Comme $j _m \circ \pi _m$ est continu, comme 
$\pi _m$ est strict, $j _m$ est donc continu.  
Comme $G$ est séparé, il en est alors de même de $G ^{(m)}$.
Ainsi $G ^{(m)}$ est un $K$-espace de Banach. 
Il découle de \ref{lim-surj-strict} que le morphisme de gauche du diagramme canonique
\begin{equation}
\xymatrix @R=0,3cm{
{\underrightarrow{\lim} _m  \, G ^{(m)}} 
\ar[r] ^-{\sim} 
& 
{G } 
\\ 
{\underrightarrow{\lim} _m  \, V _m} 
\ar[r] ^-{\sim}
\ar@{->>}[u] ^-{}
& 
{V ,} 
\ar@{->>}[u] ^-{}
}
\end{equation}
est un épimorphisme strict. 
Comme il en est de même du morphisme de droite, 
comme l'isomorphisme du bas est un homéomorphisme,
il en est donc de même de l'isomorphisme du haut. 
\end{proof}

\subsection{Topologie projective d'un produit tensoriel sur une $K$-algèbre}
\label{section1.2}
Soit $D$ une $K$-algèbre (sans topologie et non nécessaire commutative) telle que $K$ soit dans le centre de $D$.
Notons $\mathfrak{C} _{D, g}$ (resp. $\mathfrak{C} _{D, d}$) 
la sous-catégorie pleine de 
$\mathfrak{C}$ des objets de $\mathfrak{C}$ tels que la structure de $K$-espace vectoriel se prolonge en 
une structure de $D$-module à gauche (resp. à droite).

\begin{vide}
[Application $D$-balancée continue et complétion]
\label{constr-hatbeta-carrecomm}
Soient $V$ un objet de $\mathfrak{C} _{D, d}$,
$W$ un objet de $\mathfrak{C} _{D, g}$
et
$U$ un objet de $\mathfrak{C}$.
On munit $V \times W$ de la topologie produit, i.e. $V \times W$ est le produit calculé dans
$\mathfrak{C}$.
Soit
$\beta \colon V \times W \to U$ une application $D$-balancée.
Par $K$-bilinéarité de $\beta$, si $L$ est un sous-$\V$-module (resp. un réseau de $U$ au sens de \cite[2.1]{Schneider-NonarchFuncAn}) de $U$,
alors $\phi ^{-1} (L)$ est un sous-$\V$-module (resp. un réseau) de $V \times W$.
De plus, d'après \cite[17.1]{Schneider-NonarchFuncAn}, comme $\beta$ est $K$-bilinéaire, 
l'application $\beta$ est continue si et seulement si elle l'est en zéro, i.e., 
pour tout sous-$\V$-module ouvert de $L$ de $U$, il existe des sous-$\V$-modules ouverts respectivement $M$ de $V$ et $N$ de $W$
tels que $\beta ( M \times N) \subset L$.

Supposons $\beta$ continue. 
On dispose alors du morphisme $K$-bilinéaire
$\underset{M, N}{\underleftarrow{\lim}}\,  V\times W / M\times N \to \underset{L}{\underleftarrow{\lim}} \, V /L$,
où 
$L$ (resp. $M$, resp. $N$) parcourt les réseaux ouverts de $U$ (resp. $V$, resp. $W$).
On note $\widehat{\beta}\colon \widehat{V} \times  \widehat{W} \to  \widehat{U}$, 
cette application. Comme les réseaux ouverts de $\widehat{U}$ sont de la forme
$\underset{L}{\underleftarrow{\lim}} \, L _0 /L$ où
$L _0$ est un réseau ouvert de $U$ et $L$ parcourt les réseaux ouverts de $U$ inclus dans $L _0$
(et de même pour $ \widehat{V} \times  \widehat{W}$), 
on vérifie alors que $\widehat{\beta}$ est continue.
Comme l'image de $V \times W$ dans  
$\widehat{V} \times  \widehat{W}$ est dense, on vérifie que $\widehat{\beta}$ 
est l'unique application $K$-bilinéaire continue induisant le carré commutatif: 
\begin{equation}
\label{hatbeta-carrecomm}
\xymatrix @R=0,3cm {
{\widehat{V} \times  \widehat{W}} 
\ar@{.>}[r] ^-{\widehat{\beta}}
& 
{\widehat{U} } 
\\ 
{V \times W} 
\ar[r] ^-{\beta}
\ar[u] ^-{}
& 
{U.} 
\ar[u] ^-{}
}
\end{equation}
Par contre, pour que $\widehat{\beta}$ soit $D$-balancée, il faut des hypothèses topologiques sur $D$
(voir \ref{Dtopo-hatbeta-carrecomm}). 
\end{vide}

\begin{vide}
\label{def-prod-tens}

Soient $V$ un objet de $\mathfrak{C} _{D, d}$
et $W$ un objet de $\mathfrak{C} _{D, g}$.
En munissant $V \times W$ de la topologie produit, 
la topologie projective sur le produit tensoriel
$V \otimes _{D} W$ est la topologie $K$-localement convexe la plus fine 
telle que le morphisme $K$-bilinéaire canonique
$\rho _{V,W}\colon V \times W \to V \otimes _{D} W$ 
soit continue.
Ainsi, un réseau $L \subset V \otimes _{D} W$ est ouvert si et seulement si 
$\rho _{V,W} ^{-1} (L)$ est ouvert. 
Comme nous ne considérerons que des topologies de type projectif sur les produits tensoriels,
nous pourrons omettre d'indiquer le qualificatif {\og projectif \fg}.
 
L'objet $V \otimes _{D} W$ vérifie 
la propriété universelle: 
pour toute application $D$-balancée et 
continue 
de la forme
$\phi \colon V \times W \to U$, 
il existe un unique 
morphisme dans $\mathfrak{C}$
de la forme 
$\theta \colon V \otimes _{D} W \to U$ 
tel que $\theta \circ \rho _{V,W} = \phi$.

On en déduit que l'on obtient en fait le bifoncteur canonique 
$$- \otimes _{D} -\colon \mathfrak{C} _{D, d} \times \mathfrak{C} _{D, g} \to \mathfrak{C}.$$

\end{vide}

\begin{lemm}
\label{lemm-topo-tens-prod}
Soient $V$ un objet de $\mathfrak{C} _{D, d}$
et $W$ un objet de $\mathfrak{C} _{D, g}$.
On suppose qu'il existe un sous-$\V$-module $V _0$ de $V$ (resp. $W _0$ de $W$) tel qu'une base de voisinages de zéro de $V$ (resp. $W$)
soit donnée par la famille $(p ^{n} V _0) _{n\in\N}$ (resp. $(p ^{n} W _0) _{n\in\N}$).
Notons $U _0:= <\rho _{V,W}(V _0 \times W _0)>$, 
où $<?>$ désigne le {\og sous $\V$-module de 
$V \otimes _{D} W$ engendré par $?$\fg}.
Alors une base de voisinages sur 
$V \otimes _{D} W$ muni de sa topologie canonique (voir \ref{def-prod-tens})
est donnée par 
$(p ^n U _0) _{n\in \N}$.
\end{lemm}

\begin{proof}
Comme $V _0$ et $W _0$ sont des réseaux respectifs de $V$ et $W$, 
$U _0$ est  un réseau de  $U$.
Pour tout entier $n\in \N$, comme $\rho _{V,W} ^{-1} (p ^n U _0)\supset p ^{n} (V _0 \times W _0)$, 
les $p ^n U _0$ sont donc des ouverts de $V \otimes _{D} W$.
Réciproquement, soit $L$ un sous-$\V$-module ouvert de $V \otimes _{D} W$. 
Il existe alors un entier $n$ assez grand tel que 
$\rho _{V,W} ^{-1} (L) \supset p ^n (V _0 \times W _0)$.
Comme $L$ est un $\V$-module, 
on a alors
$L \supset <\rho _{V,W} (p ^n (V _0 \times W _0))>
=
p ^{2n} <\rho _{V,W} (V _0 \times W _0)>
=
p ^{2n} U _0$.

\end{proof}

\begin{vide}
Soit $D ' \to D$ est un homomorphisme de $K$-algèbres tel que $K$ soit aussi dans le centre de $D'$. 
Soient $V$ un objet de $\mathfrak{C} _{D, d}$
et $W$ un objet de $\mathfrak{C} _{D, g}$.
Comme le composé  
$V \times W \to V \otimes _{D'} W \to V \otimes _{D} W$
est le morphisme canonique, par définition des topologies définies sur 
$V \otimes _{D'} W$ et $ V \otimes _{D} W$,
l'épimorphisme 
$V \otimes _{D'} W \to V \otimes _{D} W$ est donc strict.  
\end{vide}

\begin{lemm}
\label{epi-otimes-strict}
Soient $V \to V''$ (resp. $W \to W''$)
un épimorphisme strict de 
$\mathfrak{C} _{D, d}$
(resp. $W$ un objet de $\mathfrak{C} _{D, g}$).
Les épimorphismes 
$V \otimes _{D} W \to V \otimes _{D} W''$
et
$V \otimes _{D} W \to V'' \otimes _{D} W$
sont alors stricts. 
\end{lemm}

\begin{proof}
Par symétrie, 
vérifions-le seulement pour le premier. 
Soit $L$ un réseau de $V \otimes _{D} W''$.
Comme $V \times W \to V \times W''$ est strict, 
par définition de la topologie sur 
$V \otimes _{D} W''$,
$L$ est ouvert si est seulement si son image inverse 
sur $V \times W $ est un ouvert. Par définition de la topologie sur 
$V \otimes _{D} W$, cela équivaut au fait que son image inverse sur 
$V \otimes _{D} W $ soit un ouvert.
D'où le résultat.

\end{proof}

Nous ne devrions pas avoir besoin des deux lemmes qui suivent mais cela vaut sans-doute la peine de l'écrire:

\begin{lemm}
\label{limD-balance}
Soient $(V _i ) _{i\in I}$ un système inductif filtrant de 
$\mathfrak{C} _{D, d}$,
$(W _i ) _{i\in I}$ un système inductif filtrant de 
$\mathfrak{C} _{D, g}$,
$U$ un objet de $\mathfrak{C} $,
$\beta _i \colon V _i \times W _i \to U$
une famille compatible d'applications $D$-balancées continues.
Alors l'application $D$-balancée 
$\beta:= \underrightarrow{\lim} _i V _i \times W _i \to U$
est continue. 
\end{lemm}

\begin{proof}
Notons $g _i \colon 
 V _i \times W _i 
\to \underrightarrow{\lim} _i V _i \times W _i $
les applications $D$-balancées continues canoniques.
Soit $L$ un réseau ouvert de $U$.
Le lemme découle alors de l'égalité
$\beta ^{-1} (L) = 
\sum _{i\in I} g _i  (\beta _i ^{-1} (L)).$
\end{proof}

\begin{lemm}
Soient $(V _i ) _{i\in I}$ un système inductif filtrant de 
$\mathfrak{C} _{D, d}$
et
$W$ un objet de $\mathfrak{C} _{D, g}$.
\begin{enumerate}
\item Le morphisme canonique 
$\underrightarrow{\lim} _i \, (V _i \times W)
\to
 (\underrightarrow{\lim} _i \, V _i )\times W$
 est un homéomorphisme.
\item On dispose alors de l'isomorphisme canonique dans 
$\mathfrak{C}$ de la forme
$$ \underrightarrow{\lim} _i \, (V _i \otimes _{D} W)
\riso
 (\underrightarrow{\lim} _i \, V _i )\otimes _{D} W.$$
\end{enumerate}
\end{lemm}

\begin{proof}
Notons $V := \underrightarrow{\lim} _i \, V _i $
et
$f _i \colon V _i \to V$ les morphismes canoniques.
Traitons d'abord la première assertion. 
Par définition de la topologie sur la limite inductive, 
la bijection canonique
$\underrightarrow{\lim} _i \, (V _i \times W)
\to
V\times W$
est continue. 
Soient $(L _i ) _{i\in I}$ une famille de réseaux ouverts de $(V _i ) _{i\in I}$
et 
$(M _i ) _{i\in I}$ une famille de réseaux ouverts de $W$.
Il s'agit de vérifier que 
$\sum _{i \in I} f _i (L _i) \times M _i$
est un ouvert de  
$V\times W$.
Cela résulte immédiatement de l'inclusion
$\sum _{i \in I} f _i (L _i) \times M _i
\supset 
(\sum _{i \in I} f _i (L _i) ) \times M _{i _0}$
valable quelque que soit $i _0 \in I$
(en effet, on utilise la relation $(x, y) = (x, 0) + (0,y)$ pour 
$x \in \sum _{i \in I} f _i (L _i)$ et $y \in M _{i _0}$). 

Vérifions à présent la seconde assertion, i.e. le $K$-espace localement convexe
$V \otimes _{D} W$ vérifie la propriété universelle de la limite inductive dans $\mathfrak{C}$ 
du système 
$(V _i \otimes _{D} W ) _{i\in I}$: 
si on se donne une famille compatible de morphismes de 
$\mathfrak{C}$ de la forme 
$V _i \otimes _{D} W  \to U$, 
alors ils se factorisent de manière unique en un morphisme 
$K$-linéaire de la forme 
$V \otimes _{D} W \to U$.
Pour vérifier que celui-ci est continue, 
il faut et il suffit que le composé 
$V \times W \to V \otimes _{D} W \to U$ l'est. 
Or, comme toutes les applications $D$-balancées
$V _i \times W  \to U$ sont continus, 
d'après \ref{limD-balance},
il en est de même de l'application $D$-balancée
$\underrightarrow{\lim} _i \, (V _i \times W)\to U$.
On déduit alors de la première assertion du lemme que 
le morphisme canonique $V \times W \to U$ est continu. 
D'où le résultat. 
\end{proof}

\subsection{Complétions de produit tensoriel de modules localement convexes}

\begin{vide}
\label{defi-widehat-otimes}
Avec les notations de \ref{section1.2}, 
soient $V$ un objet de $\mathfrak{C} _{D, d}$
et $W$ un objet de $\mathfrak{C} _{D, g}$.
On note $V \widehat{\otimes} _{D} W$ le séparé complété de $V \otimes _{D} W$
et $i _{V,W}\colon V\otimes _{D} W \to V \widehat{\otimes} _{D} W$ le morphisme canonique. 
Il résulte de la propriété universelle du produit tensoriel et de celle du séparé complété la propriété universelle suivante: 
pour tout application $D$-balancée et continue de la forme
$\phi \colon V \times W \to U$ avec $U \in \mathfrak{D}$,
il existe alors un unique 
morphisme dans $\mathfrak{D}$
de la forme 
$\theta \colon V \widehat{\otimes} _{D} W \to U$ 
tel que $\theta \circ i _{V,W}\circ \rho _{V,W} = \phi$.

On obtient ainsi le bifoncteur canonique 
$$- \widehat{\otimes} _{D} -\colon \mathfrak{C} _{D, d} \times \mathfrak{C} _{D, g} \to \mathfrak{D}.$$

\end{vide}

\begin{defi}
\label{moduleloccvx}
Soit $D$ une $K$-algèbre telle que $K$ soit dans le centre de $D$. 
\begin{enumerate}
\item On dit que $D$ est une $K$-algèbre localement convexe, si $D$ est muni d'une topologie $K$-localement convexe
telle que la multiplication soit une application $K$-bilinéaire continue. 
Un morphisme de $K$-algèbres localement convexes est un morphisme de $K$-algèbres qui est continue pour les topologies respectives.
On dit que $D$ est un $K$-algèbre de Banach si $D$ est une $K$-algèbre localement convexe 
dont la topologie sous-jacente en fait un $K$-espace de Banach. 

\item Soit $D$ une $K$-algèbre localement convexe. Un $D$-module à gauche localement convexe est un $D$-module à gauche $M$ muni d'une topologie $K$-localement convexe telle
 que la loi externe structurale $D \times M \to M$ soit une application $K$-bilinéaire continue. 
Un $D$-module à gauche de Banach est un $D$-module à gauche localement convexe
dont la topologie sous-jacente en fait une $K$-espace de Banach. 
 Un morphisme de $D$-modules à gauche localement convexes (resp. de Banach) est un morphisme de 
 $D$-modules à gauche qui est aussi un morphisme de $K$-espaces localement convexe (resp. de Banach).
 
De même en remplaçant à gauche par à droite. 
\end{enumerate}

\end{defi}

\begin{lemm}
Soient $D$ une $K$-algèbre localement convexe, 
$\phi \colon M\to N$ un morphisme de $D$-modules à gauche (resp. à droite) localement convexes.
\begin{enumerate}
\item La structure de $K$-espace localement convexe séparé complet sur $\widehat{D}$ se prolonge en une structure canonique de 
 $K$-algèbre localement convexe. 
 De plus, le morphisme canonique 
$D \to \widehat{D}$ est un morphisme de $K$-algèbres localement convexes.
\item Le morphisme induit par séparée complétion $\widehat{\phi}\colon \widehat{M} \to  \widehat{N}$ est 
un morphisme de 
$\widehat{D}$-modules localement convexes.
De plus, le morphisme canonique 
$M \to \widehat{M}$ est un morphisme de $D$-modules localement convexes.
\end{enumerate}
\end{lemm}

\begin{proof}
Contentons-nous de prouver le cas non respectif.
D'après \ref{hatbeta-carrecomm}, l'application $K$-bilinéaire continue structurale canonique
$\mu _{D} \colon D \times D \to D$
induit l'application $K$-bilinéaire continue 
$\mu _{\widehat{D}} := \widehat{\mu_D}\colon \widehat{D} \times  \widehat{D} \to\widehat{D} $
s'inscrivant (de manière unique) dans le diagramme commutatif:
\begin{equation}
\label{hatbeta-carrecommbis}
\xymatrix @R=0,3cm {
{\widehat{D} \times  \widehat{D}} 
\ar[r] ^-{\mu _{\widehat{D}} }
& 
{\widehat{D} } 
\\ 
{D \times D} 
\ar[r] ^-{\mu _D}
\ar[u] ^-{}
& 
{D.} 
\ar[u] ^-{}
}
\end{equation}
Comme les deux applications 
$\mu _{\widehat{D}}  \circ (\mu _{\widehat{D}}  \times id )
,\,
\mu _{\widehat{D}}  \circ (id \times \mu _{\widehat{D}} )
\colon 
\widehat{D} \times \widehat{D}  \times \widehat{D} \to \widehat{D}$
coïncident après composition par le morphisme canonique
$D \times D \times D \to \widehat{D} \times \widehat{D}  \times \widehat{D} $
dont l'image est dense, 
on obtient 
$\mu _{\widehat{D}}  \circ (\mu _{\widehat{D}}  \times id )
=
\mu _{\widehat{D}}  \circ (id \times \mu _{\widehat{D}} )$, 
i.e. la multiplication 
est associative. On vérifie de même les autres propriétés qui font de 
$\widehat{D} $ une $K$-algèbre localement convexe. 
Il est clair que le morphisme continu canonique 
$D \to \widehat{D}$ est alors un morphisme de $K$-algèbres localement convexes.

D'après \ref{constr-hatbeta-carrecomm}, 
les applications $K$-bilinéaires continues structurales canoniques
$\mu _M \colon D \times M \to M$ et
$\mu _N \colon D \times N \to N$ 
induisent les applications $K$-bilinéaires continues
$\mu _{\widehat{M}} := \widehat{\mu_M}\colon \widehat{D} \times  \widehat{M}   \to\widehat{M} $
et
$\mu _{\widehat{N}} := \widehat{\mu_N}\colon \widehat{D} \times  \widehat{N} \to\widehat{N} $.
De même, on vérifie que $\mu _{\widehat{M}} $
et
$\mu _{\widehat{N}} $
induisent respectivement une structure canonique 
de $\widehat{D}$-module localement convexes sur $\widehat{M}$ et $\widehat{N}$. 
Comme le diagramme
\begin{equation}
\notag
\xymatrix @R=0,3cm{
{\widehat{D} \times \widehat{M}} 
\ar[r] ^-{\mu _{\widehat{M}} }
\ar[d] ^-{id \times \widehat{\phi}}
& 
{ \widehat{M}} 
\ar[d] ^-{\widehat{\phi}}
\\ 
{\widehat{D} \times \widehat{N}} 
\ar[r] ^-{\mu _{\widehat{N}} }
& 
{ \widehat{N}} 
}
\end{equation}
est commutatif après composition par 
$D \times M \to \widehat{D} \times \widehat{M}$ dont l'image est dense, 
celui-ci est commutatif. D'où le résultat. 
\end{proof}

\begin{lemm}
\label{Dtopo-hatbeta-carrecomm}
Soient $D$ une $K$-algèbre localement convexe, 
$M$ un $D$-module à droite localement convexe,
$N$ un $D$-module à gauche localement convexe
et
$U$ un $K$-espace localement convexe.
Soit
$\beta \colon M \times N \to U$ une application $D$-balancée et continue.
L'application
$\widehat{\beta}\colon \widehat{M} \times  \widehat{N} \to  \widehat{U}$
(voir \ref{constr-hatbeta-carrecomm})
est alors une
application $\widehat{D}$-balancée et continue. 
\end{lemm}

\begin{proof}
On sait déjà que l'application $\widehat{\beta}$ est continue. 
Considérons le carré :
\begin{equation}
\label{carre-comm-Dbal}
\xymatrix @R=0,3cm{
{ \widehat{M} \times  \widehat{D} \times  \widehat{N}} 
\ar[r] ^-{\mu _{\widehat{M}} \times id }
\ar[d] ^-{id \times \mu _{\widehat{N}}}
& 
{ \widehat{M} \times  \widehat{N}} 
\ar[d] ^-{\widehat{\beta}}
\\ 
{ \widehat{M} \times  \widehat{N}} 
\ar[r] ^-{\widehat{\beta}} 
& 
{\widehat{U} ,} 
}
\end{equation}
où $\mu _{\widehat{M}} \colon \widehat{M} \times \widehat{D} \to \widehat{M}$ et
$\mu _{\widehat{N}} \colon \widehat{D} \times \widehat{N} \to \widehat{N}$ sont les applications $K$-bilinéaires continues structurales canoniques.
Comme l'image de $ M \times  D \times N$
dans 
$ \widehat{M} \times  \widehat{D} \times  \widehat{N}$ est dense, le carré 
\ref{carre-comm-Dbal} est donc commutatif car il l'est sans les chapeaux. 
\end{proof}

\begin{prop}
\label{MwidehatotimesN}
Soient $D$ une $K$-algèbre localement convexe, 
$M$ un $D$-module à droite localement convexe
et 
$N$ un $D$-module à gauche localement convexe.
On dispose alors de l'isomorphisme canonique dans $\mathfrak{D}$ de la forme:
$$M \widehat{\otimes} _{D} N \riso 
\widehat{M} \widehat{\otimes} _{\widehat{D}} \widehat{N}.$$
\end{prop}

\begin{proof}
Par fonctorialité du foncteur séparée complétion, 
on dispose du morphisme dans $\mathfrak{D}$ de la forme:
$M \widehat{\otimes} _{D} N \to 
\widehat{M} \widehat{\otimes} _{\widehat{D}} \widehat{N}.$
Pour construire le morphisme quasi-inverse, par propriété universelle du produit tensoriel, 
il s'agit de définir canoniquement une application continue
de la forme 
$\widehat{M} \times  \widehat{N} \to M \widehat{\otimes} _{D} N $ qui soit $\widehat{D}$-balancée, 
ce qui résulte aussitôt du lemme \ref{Dtopo-hatbeta-carrecomm} appliqué à $\beta$ égal à 
l'application canonique $M \times N \to M \otimes _{D} N $.
\end{proof}

\section{Espaces de type $LB$ en théorie des $D$-modules arithmétiques}

\label{nota-sect2}
Sauf mention explicite du contraire, nous utiliserons les notations et hypothèses suivantes:
soit $\X$ un $\V$-schéma formel affine, lisse, muni de coordonnées locales 
$t _1,\dots, t _d$ et soient $\partial _1, \dots, \partial _d$ les dérivations correspondantes.  
On note $\ZZ := \cap _{i=1} ^{e} V (t _i)$ et 
$u\colon \ZZ  \hookrightarrow \X$ l'immersion fermée induite. 
Soit $\mathfrak{D}= V (t _{e+1} \cdots t _{f})$
un diviseur à croisements normaux strict de $\mathfrak{X}$ tel que 
$u ^{-1} (\mathfrak{D})$ soit un diviseur à croisements normaux strict de
$\mathcal{Z}$. 
On pose $\mathfrak{X} ^{\sharp}:= (\mathfrak{X}, \mathfrak{D})$,
$\mathcal{Z} ^{\sharp}:= (\mathcal{Z}, u ^{-1}\mathfrak{D})$
et $u ^{\sharp}\colon \mathcal{Z} ^{\sharp} \hookrightarrow \mathfrak{X} ^{\sharp}$
l'immersion fermée exacte de schémas formels logarithmiques lisses sur $\V$.
Soient $f \in O _{\X}$ et $f _{0} \in O _X$ sa réduction modulo $\pi$  
telle que $T:= V(f _0)$ soit un diviseur de $X$ tel que $U:= T \cap Z$ soit un diviseur de $Z$.
Soit $\lambda _0 \colon \N\to \N$ une application croissante telle que $\lambda _0 (m) \geq m$, pour tout $m \in \N$.
Pour alléger les notations, 
on pose alors 
$\widetilde{\B} ^{(m)} _{\X} (T):= \widehat{\B} ^{(\lambda _0 (m))} _{\X} (T)$, 
$\widetilde{\B} ^{(m)} _{\ZZ} (U):= \widehat{\B} ^{(\lambda _0 (m))} _{\ZZ} (U)$, 
$\widetilde{\D} ^{(m)} _{\X ^{\sharp}}:= \widetilde{\B} ^{(m)} _{\X} (T) \widehat{\otimes} _{\O _{\X}} \widehat{\D} ^{(m)} _{\X ^{\sharp}}$,
$\widetilde{\D} ^{(m)} _{\ZZ^{\sharp}}:= \widetilde{\B} ^{(m)} _{\ZZ} (U) \widehat{\otimes} _{\O _{\ZZ}} \widehat{\D} ^{(m)} _{\ZZ ^{\sharp}}$,
$\widetilde{\D} ^{(m)}  _{\X ^{\sharp} \leftarrow \ZZ ^{\sharp}}:=
  \widehat{\D} ^{(m)}  _{\X^{\sharp} \leftarrow \ZZ ^{\sharp}}  \widehat{\otimes} _{\O _{\ZZ}}  \widetilde{\B} ^{(m)} _{\ZZ} (U) $,
 $\widetilde{\D} ^{(m)}  _{\ZZ ^{\sharp} \rightarrow \X ^{\sharp}}:=
 \widetilde{\B} ^{(m)} _{\ZZ} (U) \widehat{\otimes} _{\O _{\ZZ}}  \widehat{\D} ^{(m)}  _{\ZZ ^{\sharp} \rightarrow \X^{\sharp}}$,
 $\widetilde{\D} ^\dag _{\ZZ ^{\sharp}\to \X^{\sharp},\Q} = \underrightarrow{\lim} _m \widetilde{\D} ^{(m)} _{\ZZ ^{\sharp}\to \X ^{\sharp},\Q}$
et
$\widetilde{\D} ^\dag _{\X ^{\sharp} \leftarrow \ZZ ^{\sharp},\Q} 
= \underrightarrow{\lim} _m \widetilde{\D} ^{(m)} _{\X ^{\sharp} \leftarrow \ZZ ^{\sharp},\Q}$.

\subsection{Topologies canoniques}

\begin{vide}
[Topologie de $O _{\X,\Q}$-algèbre localement convexe canonique 
(resp.  de $O _{\X,\Q}$-algèbre normée)
sur $O  _{\X} (\hdag T) _{\Q}$]
On définit les topologies suivantes:
\label{topo-cano-Odag}
\begin{itemize}
\item Pour tout entier $m$, la topologie canonique de $O _{\X,\Q}$-algèbre localement convexe sur 
$\widetilde{B} ^{(m)} _{\X} (T) _{\Q}$
est celle dont une base de voisinages de zéro est 
$(p ^{n} \widetilde{B} ^{(m )} _{\X} (T)) _{n \in \N}$.

\item 
On munit $O  _{\X} (\hdag T) _{\Q}$ d'une topologie canonique de $O _{\X,\Q}$-algèbre localement convexe 
de telle manière que 
l'isomorphisme canonique
$O  _{\X} (\hdag T) _{\Q} \riso\underrightarrow{\lim} _m \,
\widetilde{B} ^{(m)} _{\X} (T) _{\Q}$
soit un homéomorphisme (on remarque que d'après \ref{LB-mNm}, cela ne dépend pas de l'application $\lambda _0$).

\item Comme 
$\cap _{n \in \N}p ^{n} O _{\X} [\frac{1}{f}] ^{\dag}= \{0\}$,
comme on dispose de l'isomorphisme canonique
$O  _{\X} (\hdag T) _{\Q} \riso O _{\X} [\frac{1}{f}] ^{\dag} _K$ (voir la preuve de \cite[4.3.2]{Be1}),
on munit alors $O  _{\X} (\hdag T) _{\Q}$ d'une structure de $O _{\X,\Q}$-algèbre normée 
(à ne pas confondre avec la topologie canonique induite par la limite inductive)
dont une base de voisinages de zéro est donnée par 
$(p ^{n} O _{\X} [\frac{1}{f}] ^{\dag} ) _{n\in \N}$.
Muni de cette topologie, $O  _{\X} (\hdag T) _{\Q}$ est normée mais pas de Banach:
son séparé complété est
$O _{\X} \{\frac{1}{f}\} _K$.

\item Comme nous préférons travailler avec des 
$K$-espaces de type $LB$ (voir le lemme \ref{topo-cano-riche}), 
nous prenons par défaut la topologie canonique de $O _{\X,\Q}$-algèbre localement convexe sur 
$O  _{\X} (\hdag T) _{\Q}$ (définie dans le deuxième point).
\end{itemize}
\end{vide}

\begin{vide}
[Faisceaux des opérateurs différentiels]
\label{defi-topo-can}
\begin{itemize}
\item Pour tout entier $m \geq 0$,
on munit $\widetilde{D} ^{(m)}  _{\X ^{\sharp},\Q}$ de sa topologie canonique de $\widetilde{B} ^{(m)} _{\X} (T) _{\Q}$-module de Banach, i.e. 
les $p ^{n}\widetilde{D} ^{(m)}  _{\X ^{\sharp}}$ avec $n\in \N$ forment une base de voisinages de zéro. 
On dispose de même d'une structure canonique de $\widetilde{B} ^{(m)} _{\ZZ} (U) _{\Q}$-module de Banach sur respectivement
$\widetilde{D} ^{(m)}  _{\ZZ ^{\sharp},\Q}$, 
$\widetilde{D} ^{(m)}  _{\ZZ ^{\sharp} \to \X ^{\sharp},\Q}$
et 
$\widetilde{D} ^{(m)}  _{\X ^{\sharp} \leftarrow \ZZ  ^{\sharp},\Q}$
via la base de voisinages d'ouverts de zéro
$(p ^{n}\widetilde{D} ^{(m)}  _{\ZZ ^{\sharp}} ) _{n \in \N}$,
$(p ^{n}\widetilde{D} ^{(m)}  _{\ZZ ^{\sharp} \to \X ^{\sharp}} ) _{n \in \N}$
et
$(p ^{n}\widetilde{D} ^{(m)}  _{\X ^{\sharp} \leftarrow \ZZ ^{\sharp}} ) _{n \in \N}$.

\item La topologie canonique sur la $O _{\X,\Q}$-algèbre $\smash{D} ^\dag _{\X  ^{\sharp}} (\hdag T) _{\Q}$ est la topologie limite inductive dans $\mathfrak{C}$
via $\smash{D} ^\dag _{\X  ^{\sharp}} (\hdag T) _{\Q} = \underrightarrow{\lim} _m \widetilde{D} ^{(m)}  _{\X ^{\sharp},\Q}$, où
les $\widetilde{D} ^{(m)}  _{\X ^{\sharp},\Q}$ sont munis de la topologie définie ci-dessus. 
Une base de voisinages de zéro de $\smash{D} ^\dag _{\X  ^{\sharp}} (\hdag T) _{\Q}$
est donc donnée par la famille
$\sum _{m=0} ^{\infty} p ^{n _m}  \widetilde{D} ^{(m)}  _{\X ^{\sharp}}$, 
où $(n _{m}) _{m \in \N}$ parcourt les suites d'entiers positifs
et où l'on a noté
$\sum _{m=0} ^{\infty} p ^{n _m}  \widetilde{D} ^{(m)}  _{\X ^{\sharp}}
:= 
\underset{n \in \N}{\underrightarrow{\lim}}
\sum _{m=0} ^{n} p ^{n _m}  \widetilde{D} ^{(m)}  _{\X ^{\sharp}}$.

\item De même, on définit
sur 
$\smash{D} ^\dag _{\ZZ ^{\sharp}} (\hdag U) _{\Q} = \underrightarrow{\lim} _m \widetilde{D} ^{(m)}  _{\ZZ ^{\sharp},\Q}$,
$\widetilde{D} ^\dag  _{\ZZ ^{\sharp} \to \X ^{\sharp},\Q} = \underrightarrow{\lim} _m \widetilde{D} ^{(m)}  _{\ZZ ^{\sharp}\to \X ^{\sharp},\Q}$
et
$\widetilde{D} ^\dag _{\X ^{\sharp} \leftarrow \ZZ ^{\sharp},\Q} = \underrightarrow{\lim} _m \widetilde{D} ^{(m)}  _{\X ^{\sharp} \leftarrow \ZZ ^{\sharp},\Q}$
une topologie canonique de $O _{\ZZ, \Q}$-module localement convexe.
On remarque que d'après le lemme \ref{LB-mNm} on peut remplacer l'indice $\N$ par un sous-ensemble cofinal
sans changer la topologie. 
\end{itemize}

\end{vide}

\begin{vide}
[Topologie (de $O _{\X,\Q}$-module localement convexe) canonique d'un $\smash{D} ^\dag _{\X  ^{\sharp}} (\hdag T) _{\Q}$-module cohérent]
\label{topo-can-D-coh}
Soit $E$ est un $\smash{D} ^\dag _{\X  ^{\sharp}} (\hdag T) _{\Q}$-module cohérent. 
D'après \cite[3.6]{Be1}, 
il existe pour $m _0$ assez grand un $\widetilde{D} ^{(m _0)} _{\X ^{\sharp},\Q}$-module de type fini
$E ^{(m_0)}$
et un isomorphisme 
$ \smash{D} ^\dag _{\X  ^{\sharp}} (\hdag T) _{\Q}$-linéaire de la forme
 $\epsilon \colon \smash{D} ^\dag _{\X  ^{\sharp}} (\hdag T) _{\Q} \otimes _{\widetilde{D} ^{(m _0)} _{\X ^{\sharp},\Q} } E ^{(m _0)}
 \riso E$.
Pour tout entier $m \geq m _0$, posons 
$E ^{(m)}:= \widetilde{D} ^{(m)}  _{\X ^{\sharp},\Q}  \otimes _{\widetilde{D} ^{(m_0)} _{\X ^{\sharp},\Q} }E ^{(m _0)}$.
On munit $E ^{(m)}$ de la topologie canonique qui en fait un $\widetilde{D} ^{(m)}  _{\X ^{\sharp},\Q}$-module de type fini de Banach (égale à la topologie
quotient via un épimorphisme de la forme $(\widetilde{D} ^{(m)}  _{\X ^{\sharp},\Q} ) ^{r} \to E ^{(m)}$).
Comme pour tout $m$ les morphismes canoniques de la forme 
$(\widetilde{D} ^{(m)}  _{\X ^{\sharp},\Q} ) ^r \to 
(\widetilde{D} ^{(m+1)} _{\X ^{\sharp},\Q} ) ^r$ sont continus,
on en déduit que le morphisme canonique
$E ^{(m)}\to E ^{(m+1)}$ est continu.
On munit $E$ d'une topologie de $\smash{D} ^\dag _{\X  ^{\sharp}} (\hdag T) _{\Q}$-module cohérent localement convexe qui fait 
de
$\epsilon \colon \underrightarrow{\lim} _m E ^{(m)}  \riso E$
un homéomorphisme. 
Cela ne dépend pas du choix de $(m _0, E ^{(m_0)}, \epsilon) $.
En effet, soit pour $m '_0$ assez grand un $\widetilde{D} ^{(m '_0)} _{\X ^{\sharp},\Q}$-module de type fini
$E ^{\prime (m '_0)}$
et un isomorphisme 
$ \smash{D} ^\dag _{\X  ^{\sharp}} (\hdag T) _{\Q}$-linéaire de la forme
 $\epsilon ' \colon \smash{D} ^\dag _{\X  ^{\sharp}} (\hdag T) _{\Q} \otimes _{\widetilde{D} ^{(m _0)} _{\X ^{\sharp},\Q} } E ^{\prime (m '_0)}
 \riso E$. 
 Pour $m \geq m ' _0$, posons
 $E ^{\prime (m)}:= \widetilde{D} ^{(m)}  _{\X ^{\sharp},\Q}  \otimes _{\widetilde{D} ^{(m' _0)} _{\X ^{\sharp},\Q} }E ^{\prime  (m ' _0)}$. 
D'après \cite[3.6.2]{Be1}, pour $m'' _0$ assez grand, il existe pour pour tout $m \geq m '' _0$ des isomorphismes $\widetilde{D} ^{(m)}  _{\X ^{\sharp},\Q} $-linéaires
$\epsilon _{m}\colon E ^{\prime (m )} \riso E ^{(m)}$
tels que $\epsilon \circ \epsilon ^{\dag} = \epsilon '$, avec 
$\epsilon ^{\dag}:= \underrightarrow{\lim} _m\, \epsilon _m$.
Or, d'après \cite[3.7.3]{bosch} (en effet, on remarque que la commutativité des anneaux est superflue),
les $\epsilon _m$ sont des homéomorphismes. Par passage à la limite,
on en déduit que $\epsilon ^{\dag}$ est un homéomorphisme.
D'où la canonicité de la topologie (on utilise aussi \ref{LB-mNm}).
\end{vide}

\begin{vide}
[Topologie canonique d'un $O  _{\X} (\hdag T) _{\Q}$-module de type fini]
\label{topo-can-O-coh}
Soit $E$ est un $O  _{\X} (\hdag T) _{\Q}$-module de type fini. 
D'après \cite[3.6]{Be1}, 
il existe, pour $m _0$ assez grand, 
un $\widetilde{B} ^{(m _0)} _{\X} (T) _{\Q}$-module de type fini
$E ^{(m_0)}$ 
et un isomorphisme $O  _{\X} (\hdag T) _{\Q}$-linéaire de la forme
 $\epsilon \colon O  _{\X} (\hdag T) _{\Q} \otimes _{\widetilde{B} ^{(m _0)}_{\X} (T) _{\Q} } E ^{(m _0)} \riso   E$.
Pour tout entier $m \geq m _0$, posons 
$E ^{(m)}:= \widetilde{B} ^{(m)} _{\X} (T) _{\Q}  \otimes _{\widetilde{B} ^{(m_0)} _{\X} (T) _{\Q} }E ^{(m _0)}$.
On munit $E ^{(m)}$ de la topologie canonique qui en fait un $\widetilde{B} ^{(m)} _{\X} (T) _{\Q}$-module de type fini de Banach (égale à la topologie
quotient via un épimorphisme de la forme $(\widetilde{B} ^{(m)} _{\X} (T) _{\Q} ) ^{r} \to E ^{(m)}$).
Comme, pour tout $m$, les morphismes canoniques de la forme 
$(\widetilde{B} ^{(m)} _{\X} (T) _{\Q} )^r \to 
(\widetilde{B} ^{(m+1)} _{\X} (T) _{\Q} )^r$ sont continus,
on en déduit que le morphisme canonique
$E ^{(m)}\to E ^{(m+1)}$ est continu.
On munit $E$ d'une topologie de $O  _{\X} (\hdag T) _{\Q}$-module de type fini localement convexe qui fait 
de
$\epsilon \colon \underrightarrow{\lim} _m E ^{(m)}  \riso E$
un homéomorphisme. 
De manière analogue à \ref{topo-can-D-coh}, 
on vérifie que cela ne dépend pas du choix de $(m _0, E ^{(m_0)}, \epsilon) $.
\end{vide}

\subsection{Cas des isocristaux surconvergents}

Rappelons les conventions suivantes:
\begin{defi}
 Un isocristal sur $\X ^{\sharp}$ surconvergent le long de $T$ est un 
$\smash{\D} ^\dag _{\X ^{\sharp}} (\hdag T) _{\Q}$-module cohérent qui soit
aussi un $\O _{\X} (\hdag T) _{\Q}$-module localement projectif de type fini.
\end{defi}

\begin{lemm}
[Topologie canonique d'un log-isocristal surconvergent]
\label{lemm-top-can-isoc}
Soit $E$ est un $D ^{\dag}  _{\X ^{\sharp}} (\hdag T) _{\Q}$-module cohérent 
qui est aussi un
$O  _{\X} (\hdag T) _{\Q}$-module projectif de type fini pour la structure induite. 
La topologie sur $E$ en tant que $D ^{\dag}  _{\X ^{\sharp}} (\hdag T) _{\Q}$-module cohérent 
est la même que celle en tant que 
$O  _{\X} (\hdag T) _{\Q}$-module de type fini.
On l'appellera donc topologie canonique sur l'isocristal surconvergent $E$.
\end{lemm}

\begin{proof}
De manière analogue à \cite[4.4.5]{Be1} (voir aussi \cite[4.4.7]{Be1}), 
quitte à augmenter $\lambda _0$, 
il existe pour $m _0$ assez grand un $\widetilde{D} ^{(m _0)} _{\X ^{\sharp},\Q}$-module de type fini topologiquement nilpotent
$E ^{(m_0)}$ qui est pour la structure induite un 
$\widetilde{B} ^{(m _0)} _{\X} (T) _{\Q}$-module de type fini 
et 
tel que $E \riso  D ^{\dag}  _{\X ^{\sharp}} (\hdag T) _{\Q}\otimes _{\widetilde{D} ^{(m _0)} _{\X ^{\sharp},\Q} } E ^{(m _0)}$.
Or, d'après \cite[4.1.2]{Be1}, 
la topologie canonique de $E ^{(m_0)}$
en tant que $\widetilde{B} ^{(m _0)} _{\X} (T) _{\Q}$-module de type fini
est la même que celle en tant que $\widetilde{D} ^{(m _0)} _{\X ^{\sharp},\Q}$-module de type fini.
De manière analogue à \cite[4.4.11]{Be1}, 
on vérifie que le morphisme canonique
$\widetilde{B} ^{(m)}_{\X} (T) _{\Q} \otimes _{\widetilde{B} ^{(m _0)}_{\X} (T) _{\Q} } E ^{(m _0)}
\to 
\widetilde{D} ^{(m)}  _{\X ^{\sharp},\Q}  \otimes _{\widetilde{D} ^{(m_0)} _{\X ^{\sharp},\Q} }E ^{(m _0)}$
est un isomorphisme.
D'où le résultat.
\end{proof}

\begin{lemm}
\label{cont-Dcoh}
Soit $\phi\colon E \to E'$ un morphisme de $\smash{D} ^\dag _{\X  ^{\sharp}} (\hdag T) _{\Q}$-modules cohérents 
(resp. de $O  _{\X} (\hdag T) _{\Q}$-modules de type fini). 
Avec $E$ et $E'$ munis des topologies canoniques (voir respectivement \ref{topo-can-D-coh} et \ref{topo-can-O-coh}), 
le morphisme $\phi$ est continu. 
\end{lemm}

\begin{proof}
D'après \cite[3.6]{Be1}, 
il existe pour $m _0$ assez grand un morphisme de $\widetilde{D} ^{(m _0)} _{\X ^{\sharp},\Q}$-modules de type fini
$\phi ^{(m _0)}
\colon 
E ^{(m _0)} \to E ^{\prime (m _0)}$
tel que 
$\smash{D} ^\dag _{\X  ^{\sharp}} (\hdag T) _{\Q} \otimes _{\widetilde{D} ^{(m _0)} _{\X ^{\sharp},\Q} } \phi ^{(m _0)} $ 
soit isomorphe à $\phi$. 
Comme $\phi ^{(m _0)}$ est continue pour les topologies canoniques (voir \cite[3.7.3]{bosch}),
il en résulte que $\phi$ est continu. Le cas respectif se traite de la même manière. 
\end{proof}

\begin{lemm}
\label{epi-stric-topo-can}
Soit $\phi\colon E \twoheadrightarrow E'$ un épimorphisme de $\smash{D} ^\dag _{\X  ^{\sharp}} (\hdag T) _{\Q}$-modules cohérents 
(resp. de $O  _{\X} (\hdag T) _{\Q}$-modules de type fini). 
Alors $\phi$ est un morphisme strict.
\end{lemm}

\begin{proof}
D'après \cite[3.6]{Be1}, 
il existe pour $m _0$ assez grand un morphisme de $\widetilde{D} ^{(m _0)} _{\X ^{\sharp},\Q}$-modules de type fini
$\phi ^{(m _0)}
\colon 
E ^{(m _0)} \to E ^{\prime (m _0)}$
tel que 
$\smash{D} ^\dag _{\X  ^{\sharp}} (\hdag T) _{\Q} \otimes _{\widetilde{D} ^{(m _0)} _{\X ^{\sharp},\Q} } \phi ^{(m _0)} $ 
soit isomorphe à $\phi$. 
Comme $\smash{D} ^\dag _{\X  ^{\sharp}} (\hdag T) _{\Q} \otimes _{\widetilde{D} ^{(m _0)} _{\X ^{\sharp},\Q} } \mathrm{coker} (\phi ^{(m _0})= \{0\}$,
quitte à augmenter $m _0$, on peut supposer que $\phi ^{(m _0)}$ est surjectif. 
Il en est alors de même de $\widetilde{D} ^{(m )}_{\X ^{\sharp},\Q} \otimes _{\widetilde{D} ^{(m _0)} _{\X ^{\sharp},\Q} } \phi ^{(m _0)} $, pour tout entier 
$m \geq m _0$, ces derniers étant d'ailleurs continus et stricts pour les topologies de Banach respectives (voir \cite[3.7.3, corollary 5]{bosch}).
Par passage à la limite, il en résulte que $\phi$ est strict (voir \ref{lim-surj-strict}).
Le cas respectif se traite de la même manière. 
\end{proof}

\begin{lemm}
\label{topo-cano-riche}
La topologie de $O _{\X,\Q}$-algèbre localement convexe canonique sur $O  _{\X} (\hdag T) _{\Q}$
est plus riche que celle induite par sa structure de $O _{\X,\Q}$-algèbre normée. 
Pour la topologie canonique,
on en déduit que $O  _{\X} (\hdag T) _{\Q}$
est une $O _{\X,\Q}$-algèbre de type LB. 
\end{lemm}

\begin{proof}
On dispose des isomorphismes canoniques $O  _{\X} (\hdag T) _{\Q} \riso O _{\X} [\frac{1}{f}] ^{\dag} _K$ et 
$\widetilde{B} ^{(m)} _{\X} (T) \riso
O _{\X} \{  \frac{p}{f ^{p ^{m+1}}}\}$ (voir la preuve de \cite[4.3.2]{Be1}).
Il en résulte le morphisme de $\V$-algèbres
$\widetilde{B} ^{(m)} _{\X} (T) \hookrightarrow  O _{\X} [\frac{1}{f}] ^{\dag}$ 
(pour l'injectivité,
voir \cite[4.3.3.(ii)]{Be1}). 
Le morphisme canonique de $O _{\X,\Q}$-algèbres 
$\widetilde{B} ^{(m)} _{\X} (T) _{\Q} \hookrightarrow O  _{\X} (\hdag T) _{\Q}$
est donc continu, avec $O  _{\X} (\hdag T) _{\Q}$ muni de sa topologie de
$O _{\X,\Q}$-algèbre normée. 
En passant à la limite, on en déduit le résultat. 
\end{proof}

\begin{lemm}
\label{I-induit-topcan-ferme}
Soit 
$I \subset (O  _{\X} (\hdag T) _{\Q}) ^{r}$
un monomorphisme de $O  _{\X} (\hdag T) _{\Q}$-modules.
Alors $I$ est fermé dans $(O  _{\X} (\hdag T) _{\Q}) ^{r}$ pour la topologie induite par la topologie canonique de $(O  _{\X} (\hdag T) _{\Q}) ^{r}$. 
\end{lemm}

\begin{proof}
Notons $M := (O  _{\X} (\hdag T) _{\Q}) ^{r} / I$ et 
$M _0$ l'image du morphisme composé 
$(O _{\X} [\frac{1}{f}] ^{\dag}) ^{r}  \subset (O  _{\X} (\hdag T) _{\Q}) ^{r} \twoheadrightarrow M$.
Munissons $M$ de la topologie quotient (pour la topologie canonique de 
$(O  _{\X} (\hdag T) _{\Q}) ^{r}$).
Il s'agit ainsi de vérifier que $M$ est séparé pour cette topologie.
Or, comme pour tout entier $n$ le $\V$-module
$p ^{n}  (O _{\X} [\frac{1}{f}] ^{\dag}) ^{r}$ est un ouvert 
de $(O  _{\X} (\hdag T) _{\Q}) ^{r}$ (voir \ref{topo-cano-riche}), 
le $\V$-module $p ^{n}  M _0$ est alors un ouvert de $M$. 
En outre, comme $O _{\X} [\frac{1}{f}] ^{\dag}$ est noethérien et comme $O _{\X} [\frac{1}{f}] ^{\dag} \to O _{\X} \{ \frac{1}{f} \}$ est fidèlement plat,
comme $M _0$ est un $O _{\X} [\frac{1}{f}] ^{\dag}$-module de type fini, alors $M _0$ est séparé pour la topologie $p$-adique (voir \cite[Théorèmes 8.10 et 8.12]{matsumura}), i.e. $\cap _{n\in \N} p ^{n} M _0 = \{0 \}$. Il en résulte que $M$ est séparé.
\end{proof}

\begin{prop}
\label{O-coh-LB}
Les $O  _{\X} (\hdag T) _{\Q}$-modules de type fini sont des $O _{\X,\Q}$-modules de type LB pour la topologie canonique. 
De plus, les sous-$O  _{\X} (\hdag T) _{\Q}$-modules d'un $O  _{\X} (\hdag T) _{\Q}$-module de type fini $M$
sont fermés dans $M$ pour la topologie canonique de $M$. 
\end{prop}

\begin{proof}
Soit $M$ un $O  _{\X} (\hdag T) _{\Q}$-module de type fini muni de sa topologie canonique.
Il existe un épimorphisme de $O  _{\X} (\hdag T) _{\Q}$-modules de la forme 
$(O  _{\X} (\hdag T) _{\Q}) ^{r}  \twoheadrightarrow M$.
D'après \ref{epi-stric-topo-can}, ce morphisme est strict pour les topologies canoniques respectives.
Il résulte alors de \ref{I-induit-topcan-ferme} que $M$ est séparé et donc $M$ est un $O _{\X,\Q}$-module de type LB. 
Si $M'$ est un sous-$O  _{\X} (\hdag T) _{\Q}$-module de $M$, alors, grâce à \ref{epi-stric-topo-can},
le module $M /M'$ est séparé pour la topologie quotient induite par $M$. D'où le résultat.
\end{proof}

\begin{rema}
Soit 
$I \subset (\smash{D} ^\dag _{\X  ^{\sharp}} (\hdag T) _{\Q}) ^{r}$
un sous-$\smash{D} ^\dag _{\X  ^{\sharp}} (\hdag T) _{\Q}$-module cohérent.
Soit 
$J \subset (O _{\X} (\hdag T) _{\Q}) ^{r}$
un sous-$O _{\X} (\hdag T) _{\Q}$-module.
Il n'est pas clair que ces inclusions soient stricts pour les topologies canoniques respectives
ni que $I$ soit fermé dans $(\smash{D} ^\dag _{\X  ^{\sharp}} (\hdag T) _{\Q}) ^{r}$. 
Cependant, d'après une communication de Tomoyuki Abe,
les $\smash{D} ^\dag _{\X  ^{\sharp}} (\hdag T) _{\Q}$-modules cohérents sont aussi des espaces de type LB.
\end{rema}

\subsection{Cas du faisceau des opérateurs différentiels de niveau fini à singularités surconvergentes}

\begin{prop}
\label{morphDdag->Ddagcontin}
Soient $u\colon \ZZ \hookrightarrow \X$ une immersion fermée de $\V$-schémas formels affines et lisses,
$\mathfrak{D}$ un diviseur à croisements normaux strict de $\mathfrak{X}$ tel que 
$u ^{-1} (\mathfrak{D})$ soit un diviseur à croisements normaux strict de
$\mathcal{Z}$. On pose $\mathfrak{X} ^{\sharp}:= (\mathfrak{X}, \mathfrak{D})$,
$\mathcal{Z} ^{\sharp}:= (\mathcal{Z}, u ^{-1}\mathfrak{D})$
et $u ^{\sharp}\colon \mathcal{Z} ^{\sharp} \hookrightarrow \mathfrak{X} ^{\sharp}$
l'immersion fermée exacte de schémas formels logarithmiques lisses sur $\V$.
Soit $\phi$ un morphisme $\smash{D} ^\dag _{\X  ^{\sharp}} (\hdag T) _{\Q}$-linéaire à gauche de la forme
$\phi \colon (\smash{D} ^\dag _{\X  ^{\sharp}} (\hdag T) _{\Q}) ^{r} \to (\smash{D} ^\dag _{\X  ^{\sharp}} (\hdag T) _{\Q}) ^{s}$.
Soit $\psi$ un morphisme $\smash{D} ^\dag _{\X  ^{\sharp}} (\hdag T) _{\Q}$-linéaire à droite de la forme
$\psi \colon (\smash{D} ^\dag _{\X  ^{\sharp}} (\hdag T) _{\Q}) ^{r} \to (\smash{D} ^\dag _{\X  ^{\sharp}} (\hdag T) _{\Q}) ^{s}$.
Notons $u ^* _g (\phi) \colon (D ^\dag _{\ZZ ^{\sharp} \to \X ^{\sharp},\Q}) ^{r} \to (D ^\dag _{\ZZ ^{\sharp} \to \X ^{\sharp},\Q}) ^{s}$
et
$u ^* _d (\psi) \colon (D ^\dag _{\X ^{\sharp} \leftarrow \ZZ ^{\sharp},\Q}) ^{r} \to (D ^\dag _{\X ^{\sharp} \leftarrow \ZZ ^{\sharp},\Q}) ^{s}$
les morphismes $\smash{D} ^\dag _{\ZZ ^{\sharp}} (\hdag U) _{\Q}$-linéaires induits par fonctorialité. 
Les morphismes $\phi$, $u ^* _g (\phi) $, $u ^* _d (\psi) $ sont continus pour les topologies canoniques respectives (voir
les définitions de \ref{defi-topo-can}).
\end{prop}

\begin{proof}
1) La continuité de $\phi$ résulte du lemme \ref{cont-Dcoh}.

2) Vérifions à présent que $u ^* _g (\phi)$ est continu. 
Notons $\pi$ les morphismes $O _{\X,\Q}$-linéaires surjectifs de la forme
$\pi \colon (\smash{D} ^\dag _{\X  ^{\sharp}} (\hdag T) _{\Q}) ^n \twoheadrightarrow 
(D ^\dag _{\ZZ ^{\sharp} \to \X ^{\sharp},\Q}) ^n$, pour un certain entier positif $n$.
Comme $u ^* _g (\phi) \circ \pi = \pi \circ \phi$, il suffit d'établir que 
$\pi$ est un morphisme continu et strict.
Comme la surjection canonique 
$\widetilde{D} ^{(m)}  _{\X ^{\sharp},\Q}
\to 
\widetilde{D} ^{(m)}  _{\ZZ ^{\sharp} \to \X ^{\sharp} ,\Q}$
envoie $\widetilde{D} ^{(m)}  _{\X ^{\sharp}}$ sur 
$\widetilde{D} ^{(m)}  _{\ZZ ^{\sharp} \to \X ^{\sharp} }$, celle-ci est donc continue. 
D'après le théorème de l'application ouverte de Banach, 
ce morphisme surjectif et continu de $K$-espaces de Banach est donc strict.
Il découle alors du lemme \ref{lim-surj-strict}
qu'il en est de même par passage à la limite sur le niveau de $\pi$.

3) On procède de même que pour l'étape 2) pour valider la continuité de
$ u ^* _d (\psi)$.
\end{proof}

\begin{prop}
\label{lim-separe}
Soit $(N _{m}) _{m\in\N}$ une suite strictement croissante d'entiers positifs. 
Le $K$-espace localement convexe $\underset{m}{\underrightarrow{\lim}} \widetilde{D} ^{(N _m)} _{\X^{\sharp},\Q}$
(muni de la topologie limite inductive dans la catégorie des $K$-espaces localement convexes)
est un espace de type $LB$. 
\end{prop}

\begin{proof}
D'après \ref{LB-mNm}, il suffit de traiter le cas où $N _m =m$. 
Comme les $K$-espaces $ \widetilde{D} ^{(m)}  _{\X ^{\sharp},\Q}$ sont de Banach, il s'agit de vérifier 
la séparation de $\underset{m}{\underrightarrow{\lim}} ~\widetilde{D} ^{(m)}  _{\X ^{\sharp},\Q}$. 
Notons $L$ le sous-$\V$-module de $\smash{D} ^\dag _{\X  ^{\sharp}} (\hdag T) _{\Q}$
des éléments qui s'écrivent sous la forme
$\sum _{\underline{k} \in \N ^d} a _{\underline{k}} \underline{\partial} _{\sharp} ^{[\underline{k}]}$
avec $a _{\underline{k}} \in \O _{\X} [\frac{1}{f}] ^{\dag}$ (on demande bien sûr que la série converge
dans $\smash{D} ^\dag _{\X  ^{\sharp}} (\hdag T) _{\Q}$).
Or, comme $p ^{n} L \supset p ^{n}\smash{D} ^\dag _{\X  ^{\sharp}} (\hdag T)= \sum _{k=0} ^{\infty} p ^{n} \widetilde{D} ^{(m)}  _{\X ^{\sharp}}$,
alors $p ^{n} L $ 
est un ouvert de $ \smash{D} ^\dag _{\X  ^{\sharp}} (\hdag T) _{\Q}$ pour tout entier $n$.
Par unicité de l'écriture sous la forme 
$\sum _{\underline{k} \in \N ^d} a _{\underline{k}} \underline{\partial} _{\sharp} ^{[\underline{k}]}$
des éléments de $L$, 
on vérifie que
$\cap _{n \in \N} p ^{n} L =\{0\}$.
On obtient donc la séparation voulue. 
\end{proof}

\begin{prop}
\label{DotimesD-sep-seqcomp}
Les topologies canoniques sont celles définies en \ref{defi-topo-can}.

\begin{enumerate}
\item Pour tout entier $m\in \N$, soient 
$V ^{(m)}$ un $D ^{(m)} _{\ZZ ^{\sharp}}  (U) _{\Q}$-module à gauche localement convexe
et 
$V ^{(m)} \to V ^{(m+1)} $ 
un morphisme continu. 
Soient $(N _{m}) _{m\in\N}$ et $(N '_{m}) _{m\in\N}$ 
deux suites strictement croissantes d'entiers positifs.  
Avec les notations de \ref{defi-widehat-otimes},
l'isomorphisme canonique 
\begin{equation}
\label{DotimesD-sep-seqcomp-iso1}
\underrightarrow{\lim} _m 
\widetilde{D} ^{(N '_m)} _{\X ^{\sharp} \leftarrow \ZZ ^{\sharp},\Q} 
\widehat{\otimes} _{D _{\ZZ ^{\sharp},\Q}} V ^{( N _m)} 
\riso
\underrightarrow{\lim} _m 
\widetilde{D} ^{(m)}  _{\X ^{\sharp} \leftarrow \ZZ ^{\sharp},\Q} 
\widehat{\otimes} _{D _{\ZZ ^{\sharp},\Q}} V ^{( m)} .
\end{equation}
est alors un homéomorphisme. 

\item Les $K$-espaces localement convexes 
$\underrightarrow{\lim} _m 
\widetilde{D} ^{(N  _m)} _{\ZZ  ^{\sharp} \rightarrow \X ^{\sharp},\Q}$
et
$\underrightarrow{\lim} _m 
\widetilde{D} ^{(m)}  _{\X ^{\sharp} \leftarrow \ZZ ^{\sharp},\Q} 
\widehat{\otimes} _{D _{\ZZ ^{\sharp},\Q}} 
\widetilde{D} ^{(m)}  _{\ZZ ^{\sharp} \rightarrow \X ^{\sharp},\Q}$
sont des $K$-espaces de type LB.
\end{enumerate}

\end{prop}

\begin{proof}
On résout la première assertion de manière analogue à \ref{LB-mNm}.
Pour la seconde assertion, 
on procède de manière analogue à \ref{lim-separe}. 
\end{proof}

 \section{Préservation de la cohérence par foncteur cohomologique local}

Soient $u\colon \ZZ  \hookrightarrow \X$ une immersion fermée de $\V$-schémas formels séparés, quasi-compacts et lisses, 
$T$ un diviseur de $X$ tel que $U:= T \cap Z$ soit un diviseur de $Z$.
Soit $\mathfrak{D}$ un diviseur à croisements normaux strict de $\mathfrak{X}$ tel que 
$u ^{-1} (\mathfrak{D})$ soit un diviseur à croisements normaux strict de
$\mathcal{Z}$. On pose $\mathfrak{X} ^{\sharp}:= (\mathfrak{X}, \mathfrak{D})$,
$\mathcal{Z} ^{\sharp}:= (\mathcal{Z}, u ^{-1}\mathfrak{D})$
et $u ^{\sharp}\colon \mathcal{Z} ^{\sharp} \hookrightarrow \mathfrak{X} ^{\sharp}$
l'immersion fermée exacte de schémas formels logarithmiques lisses sur $\V$.
Soit $\lambda _0 \colon \N\to \N$ une application croissante telle que $\lambda _0 (m) \geq m$, pour tout $m \in \N$.
Pour alléger les notations, 
on pose alors 
$\widetilde{\B} ^{(m)} _{\X} (T):= \widehat{\B} ^{(\lambda _0 (m))} _{\X} (T)$, 
$\widetilde{\B} ^{(m)} _{\ZZ} (U):= \widehat{\B} ^{(\lambda _0 (m))} _{\ZZ} (U)$, 
$\widetilde{\D} ^{(m)} _{\X ^{\sharp}}:= \widetilde{\B} ^{(m)} _{\X} (T) \widehat{\otimes} _{\O _{\X}} \widehat{\D} ^{(m)} _{\X^{\sharp}}$,
$\widetilde{\D} ^{(m)} _{\ZZ^{\sharp}}:= \widetilde{\B} ^{(m)} _{\ZZ} (U) \widehat{\otimes} _{\O _{\ZZ}} \widehat{\D} ^{(m)} _{\ZZ^{\sharp}}$,
$\widetilde{\D} ^{(m)}  _{\X ^{\sharp} \leftarrow \ZZ ^{\sharp}}:=
  \widehat{\D} ^{(m)}  _{\X ^{\sharp} \leftarrow \ZZ ^{\sharp}}  \widehat{\otimes} _{\O _{\ZZ}}  \widetilde{\B} ^{(m)} _{\ZZ} (U) $,
 $\widetilde{\D} ^{(m)}  _{\ZZ ^{\sharp} \rightarrow \X ^{\sharp}}:=
 \widetilde{\B} ^{(m)} _{\ZZ} (U) \widehat{\otimes} _{\O _{\ZZ}}  \widehat{\D} ^{(m)}  _{\ZZ ^{\sharp} \to \X ^{\sharp}}$,
 $\widetilde{\D} ^\dag _{\ZZ ^{\sharp} \to \X ^{\sharp},\Q} = \underrightarrow{\lim} _m \widetilde{\D} ^{(m)} _{\ZZ ^\sharp\to \X ^{\sharp},\Q}$
et
$\widetilde{\D} ^\dag _{\X ^{\sharp} \leftarrow \ZZ ^{\sharp},\Q} = \underrightarrow{\lim} _m \widetilde{\D} ^{(m)} _{\X ^{\sharp} \leftarrow \ZZ ^{\sharp},\Q}$.

\subsection{Images inverses par une immersion fermée de $\V$-schémas formels affines et lisses}

On suppose dans cette section que $\X$ est affine.

\begin{lemm}
[Théorèmes A et B]
\label{thA-B-D->}
Les morphismes canoniques
$O _{\ZZ} \otimes _{O _{\X}}\widetilde{D} ^{(m)}  _{\X ^{\sharp}}
\to 
\widetilde{D} ^{(m)}  _{\ZZ ^{\sharp} \rightarrow \X ^{\sharp}}$,
$O _{\ZZ} \otimes _{O _{\X}}\widetilde{D} ^{(m)}  _{\X ^{\sharp},\Q}
\to 
\widetilde{D} ^{(m)}  _{\ZZ ^{\sharp} \rightarrow \X ^{\sharp},\Q}$
et
$O _{\ZZ} \otimes _{O _{\X}}
\smash{D} ^\dag _{\X  ^{\sharp}} (\hdag T) _{\Q}
\to 
D ^\dag _{\ZZ ^{\sharp} \to \X ^{\sharp},\Q}$
sont des isomorphismes. 

De plus, pour tout entier $q \geq 1$, 
$H ^{q} (\ZZ, \widetilde{\D} ^{(m)} _{\ZZ ^{\sharp} \to \X ^{\sharp}})=0$,
$H ^{q} (\ZZ, \widetilde{\D} ^{(m)} _{\ZZ ^{\sharp} \to \X ^{\sharp},\Q})=0$
et
$H ^{q} (\ZZ, \D ^\dag _{\ZZ ^{\sharp} \to \X ^{\sharp},\Q})=0$.
\end{lemm}

\begin{proof}
Comme le foncteur $\Gamma (\ZZ, -)$ commute aux limites projectives,
on vérifie que le morphisme canonique 
$O _{\ZZ} \otimes _{O _{\X}}\widetilde{D} ^{(m)}  _{\X ^{\sharp}}
\to 
\widetilde{D} ^{(m)}  _{\ZZ ^{\sharp} \rightarrow \X ^{\sharp}}$
est un isomorphisme. 
Comme le foncteur $\Gamma (\ZZ, -)$ et le produit tensoriel commutent à la tensorisation par $\Q$ (car $\ZZ$ est noethérien)
et aux limites inductives filtrantes de faisceaux sur $\ZZ$, 
on en déduit les deux autres isomorphismes.

On procède de même pour les annulations, 
les propriétés satisfaites par le foncteur
$\Gamma (\ZZ, -)$ que l'on a utilisées étant toujours valables
pour les foncteurs dérivés  $H ^{q} (\ZZ, -)$ (voir \cite[3.3.2, 3.4.0.1 et 3.6.5]{Be1}).
\end{proof}

\begin{nota}
[Images inverses dérivées gauches]
\label{nota-u*-u*}
Soit $E ^{(m)} $ un $\widetilde{D} ^{(m)}  _{\X ^{\sharp},\Q}$-module cohérent à gauche. 
On pose $\E ^{(m)}:= \widetilde{\D} ^{(m)} _{\X^{\sharp},\Q} \otimes _{\widetilde{D} ^{(m)}  _{\X ^{\sharp},\Q}} E ^{(m)} $,
$E:= \smash{D} ^\dag _{\X  ^{\sharp}} (\hdag T) _{\Q} \otimes _{\widetilde{D} ^{(m)}  _{\X ^{\sharp},\Q}} E ^{(m)} $
et
$\E := \smash{\D} ^\dag _{\X ^{\sharp}} (\hdag T) _{\Q}\otimes _{\widetilde{D} ^{(m)}  _{\X ^{\sharp},\Q}} E ^{(m)} $. 
On définit les foncteurs images inverses dérivées gauches $\L u ^*$ en posant
$\L u ^* (E ^{(m)}):= \widetilde{D} ^{(m)}  _{\ZZ ^{\sharp} \rightarrow \X ^{\sharp},\Q} \otimes ^{\L} _{\widetilde{D} ^{(m)}  _{\X ^{\sharp},\Q}}  E ^{(m)}$,
$\L u ^* (\E ^{(m)}):= \widetilde{\D} ^{(m)} _{\ZZ ^{\sharp} \to \X ^{\sharp},\Q} \otimes ^{\L} _{\widetilde{\D} ^{(m)} _{\X ^{\sharp},\Q}}  \E  ^{(m)}$,
$\L u ^* (E):= D ^\dag _{\ZZ ^{\sharp} \to \X ^{\sharp},\Q} \otimes ^{\L} _{\smash{D} ^\dag _{\X  ^{\sharp}} (\hdag T) _{\Q}}  E$
et
$\L u ^* (\E):= \D ^\dag _{\ZZ ^{\sharp} \to \X ^{\sharp},\Q} \otimes ^{\L} _{\smash{\D} ^\dag _{\X ^{\sharp}} (\hdag T) _{\Q}}  \E$.
On remarque que les notations sont justifiées par le lemme \ref{thA-B-D->} et correspondent aux foncteurs $\L u^*$
calculés dans la catégorie des $O _{\X, \Q}$-modules ou $\O _{\X, \Q}$-modules. 
Enfin, d'après les théorèmes de type $A$, 
tous les $\widetilde{\D} ^{(m)} _{\X ^{\sharp},\Q}$-modules cohérents et
les $\smash{\D} ^\dag _{\X ^{\sharp}} (\hdag T) _{\Q}$-modules cohérents sont de cette forme  (voir respectivement \cite[3.4 et 3.6.5]{Be1}).

De même, en remplaçant $? _{\ZZ ^{\sharp} \to \X ^{\sharp},\Q}$ par $?  _{\X ^{\sharp} \leftarrow \ZZ ^{\sharp},\Q}$,
on définit le foncteur $\L u ^{*}$ pour les modules à droite. Par exemple, 
si $E ^{(m)} $ un $\widetilde{D} ^{(m)}  _{\X ^{\sharp},\Q}$-module cohérent à droite,
on pose
$\L u ^* (E ^{(m)}):= E \otimes ^{\L} _{\widetilde{D} ^{(m)}  _{\X ^{\sharp},\Q}}  \widetilde{D} ^{(m)}  _{\X ^{\sharp} \leftarrow \ZZ ^{\sharp},\Q}$.
Enfin, si on veut préciser que l'on a affaire à des modules à gauche (resp. à droite), on pourra
noter $\L u ^* _g$ (resp. $\L u ^* _d$) à la place de $\L u ^*$.
\end{nota}

\begin{lemm}
\label{RGammau*=u*}
Avec les notations et hypothèses de \ref{nota-u*-u*}, 
on dispose alors des isomorphismes canoniques
\begin{equation}
\label{RGammau*=u*-iso}
\L u ^{*} (E^{(m)}) \riso 
\R \Gamma ( \ZZ, \L u ^{*} (\E^{(m)})),
\
\
\L u ^{*} (E) \riso 
\R \Gamma ( \ZZ, \L u ^{*} (\E)).
\end{equation}
\end{lemm}

\begin{proof}
Le premier se traitant de manière analogue, contentons-nous de vérifier le dernier isomorphisme.
Soit $P ^{\bullet}$ une résolution gauche de $E$ par des $\smash{D} ^\dag _{\X  ^{\sharp}} (\hdag T) _{\Q}$-modules libres de type fini. 
Le complexe $\PP ^{\bullet} := \smash{\D} ^\dag _{\X ^{\sharp}} (\hdag T) _{\Q} \otimes _{\smash{D} ^\dag _{\X  ^{\sharp}} (\hdag T) _{\Q}}P ^{\bullet}$ est alors une résolution gauche de $\E$
par des $\smash{\D} ^\dag _{\X ^{\sharp}} (\hdag T) _{\Q}$-modules libres de type fini. 
On dispose alors des isomorphismes caoniques:
\small
\begin{gather}
\notag
D ^\dag _{\ZZ ^{\sharp} \to \X ^{\sharp},\Q} \otimes ^{\L} _{\smash{D} ^\dag _{\X  ^{\sharp}} (\hdag T) _{\Q}}  E
\liso
D ^\dag _{\ZZ ^{\sharp} \to \X ^{\sharp},\Q} \otimes  _{\smash{D} ^\dag _{\X  ^{\sharp}} (\hdag T) _{\Q}}  P ^\bullet 
\riso 
\Gamma ( \ZZ, \D ^\dag _{\ZZ ^{\sharp} \to \X ^{\sharp},\Q} \otimes  _{\smash{D} ^\dag _{\X  ^{\sharp}} (\hdag T) _{\Q}}  P ^\bullet )
\riso
\Gamma ( \ZZ, \D ^\dag _{\ZZ ^{\sharp} \to \X ^{\sharp}} \otimes  _{u ^{-1} \smash{\D} ^\dag _{\X ^{\sharp}} (\hdag T) }  u ^{-1} \PP ^\bullet )
\\
\riso 
\R \Gamma ( \ZZ, \D ^\dag _{\ZZ ^{\sharp} \to \X ^{\sharp},\Q} \otimes  _{u ^{-1} \smash{\D} ^\dag _{\X ^{\sharp}} (\hdag T) _{\Q}}  u ^{-1} \PP ^\bullet )
\riso
\R \Gamma ( \ZZ, \D ^\dag _{\ZZ ^{\sharp} \to \X ^{\sharp},\Q} \otimes ^{\L} _{u ^{-1} \smash{\D} ^\dag _{\X ^{\sharp}} (\hdag T) _{\Q}}  u ^{-1} \E),
\end{gather}
\normalsize
l'avant-dernier isomorphisme résultant du fait que 
les $\D ^\dag _{\ZZ ^{\sharp} \to \X ^{\sharp},\Q}$-modules libres sont $\Gamma ( \ZZ, -)$-acycliques (voir le lemme \ref{thA-B-D->}). 
\end{proof}

\begin{rema}
Comme les $\smash{\D} ^\dag _{\X ^{\sharp}} (\hdag T) _{\Q}$-modules cohérents (resp. les $\widetilde{\D} ^{(m)} _{\X ^{\sharp},\Q}$-modules cohérents) 
vérifient le théorème de type $B$ (voir \cite[3.6.4]{Be1}),
on aurait pu utiliser les résolutions de Koszul pour valider le lemme 
\ref{RGammau*=u*} ci-dessus.
\end{rema}

\subsection{Image directe de niveau $m$: exactitude}
\label{nota-coor-locales}
Pour faire des calculs en coordonnées locales, nous utiliserons les notations et hypothèses suivantes:
on suppose que $\X$ est affine, muni de coordonnées locales 
$t _1,\dots, t _d$ telles $\ZZ := \cap _{i=1} ^{e} V (t _i)$, 
$\mathfrak{D}= V (t _{e+1} \cdots t _{f})$ et 
$u \colon \ZZ  \hookrightarrow \X$ soit l'immersion fermée induite. 

\begin{nota}
Pour tout entier $m \geq 0$,
on pose
$D ^{(m)} _{\ZZ ^{\sharp}}  (U):= \widetilde{B} ^{(m)} _{\ZZ} (U) \otimes _{\O _{\ZZ}} D ^{(m)} _{\ZZ ^{\sharp}} $,
$D ^{(m)}  _{\X ^{\sharp} \leftarrow \ZZ ^{\sharp}}  (U) :=
 D ^{(m)}  _{\X ^{\sharp} \leftarrow \ZZ ^{\sharp}}  \otimes _{\O _{\ZZ}}  \widetilde{B} ^{(m)} _{\ZZ} (U) $.
On munit
$D ^{(m)} _{\ZZ ^{\sharp}}  (U) _{\Q}$ (resp. 
$D ^{(m)} _{\X ^{\sharp} \leftarrow \ZZ ^{\sharp}} (U) _{\Q}$) 
d'une structure de $\widetilde{B} ^{(m)} _{\ZZ} (U) _{\Q}$-algèbre normée dont une base de voisinages de zéro est donnée
par la famille $(p ^{n} D ^{(m)} _{\ZZ ^{\sharp}} (U) ) _{n\in\N}$ (resp. $(p ^{n} D ^{(m)} _{\X ^{\sharp} \leftarrow \ZZ ^{\sharp}} (U)) _{n\in\N}$).
En d'autres termes, 
ce sont les topologies provenant des normes induites via les inclusions 
$D ^{(m)} _{\ZZ ^{\sharp}}  (U) _{\Q} \hookrightarrow 
\widetilde{D} ^{(m)}  _{\ZZ ^{\sharp},\Q}$
et
$D ^{(m)} _{\X ^{\sharp} \leftarrow \ZZ ^{\sharp}} (U) _{\Q}\hookrightarrow 
\widetilde{D} ^{(m)}  _{\X ^{\sharp} \leftarrow \ZZ ^{\sharp},\Q}$. 
De plus, on remarque que le séparé complété de 
$D ^{(m)} _{\ZZ ^{\sharp}}  (U) _{\Q}$ 
(resp. $D ^{(m)} _{\X ^{\sharp} \leftarrow \ZZ ^{\sharp}} (U) _{\Q}$)
est 
$\widetilde{D} ^{(m)}  _{\ZZ ^{\sharp},\Q}$ 
(resp. $\widetilde{D} ^{(m)}  _{\X ^{\sharp} \leftarrow \ZZ ^{\sharp},\Q}$).

\end{nota}

\begin{lemm}
\label{DotimesDhat-se}
Soit $V ' \underset{\phi}{\longrightarrow} V \underset{\psi}{\longrightarrow} V ''$ une suite exacte 
de $\widetilde{D} ^{(m)}  _{\ZZ ^{\sharp},\Q}$-modules de Banach (voir la définition \ref{moduleloccvx} et la topologie canonique de \ref{defi-topo-can}). 
On suppose de plus que $\phi$ et $\psi$ sont des morphismes stricts.
Les suites
\begin{gather}
\label{DotimesDhat-se1}
\widetilde{D} ^{(m)}  _{\X ^{\sharp} \leftarrow \ZZ ^{\sharp},\Q} \otimes _{\widetilde{D} ^{(m)}_{\ZZ ^{\sharp},\Q} } V'
\underset{id \otimes \phi}{\longrightarrow}
\widetilde{D} ^{(m)}  _{\X ^{\sharp} \leftarrow \ZZ ^{\sharp},\Q} \otimes _{\widetilde{D} ^{(m)}_{\ZZ ^{\sharp},\Q} }V
\underset{id \otimes \psi}{\longrightarrow}
\widetilde{D} ^{(m)}  _{\X ^{\sharp} \leftarrow \ZZ ^{\sharp},\Q} \otimes _{\widetilde{D} ^{(m)}_{\ZZ ^{\sharp},\Q} }V'',
\\
\widetilde{D} ^{(m)}  _{\X ^{\sharp} \leftarrow \ZZ ^{\sharp},\Q} \widehat{\otimes} _{\widetilde{D} ^{(m)}_{\ZZ ^{\sharp},\Q} } V'
\underset{id \widehat{\otimes} \phi}{\longrightarrow}
\widetilde{D} ^{(m)}  _{\X ^{\sharp} \leftarrow \ZZ ^{\sharp},\Q} \widehat{\otimes} _{\widetilde{D} ^{(m)}_{\ZZ ^{\sharp},\Q} }V
\underset{id \widehat{\otimes} \psi}{\longrightarrow}
\widetilde{D} ^{(m)}  _{\X ^{\sharp} \leftarrow \ZZ ^{\sharp},\Q} \widehat{\otimes} _{\widetilde{D} ^{(m)}_{\ZZ ^{\sharp},\Q} }V''
 \end{gather}
sont alors exactes et leurs morphismes sont stricts. 
\end{lemm}

\begin{proof}
1) Supposons d'abord que l'on dispose en fait de la suite exacte
$0 \to V ' \underset{\phi}{\longrightarrow} V \underset{\psi}{\longrightarrow} V ''\to 0$.
Pour tout $\underline{k} \in \N ^{e}$, en identifiant $\N ^e$ à un sous-ensemble de $\N ^d$ via l'inclusion
$\underline{k}= (k _1,\dots, k _e) \mapsto  (k _1,\dots, k _e, 0,\dots, 0)$, 
notons $\xi _{\underline{k}, (m)}$ l'image de $\underline{\partial} ^{<\underline{k}> _{(m)}}$ (comme 
$\underline{k} \in \N ^{e}$,
on remarque que
$\underline{\partial} ^{<\underline{k}> _{(m)}} = \underline{\partial} ^{<\underline{k}> _{(m)}} _{\sharp}$)
via la surjection canonique
$D ^{(m)} _{\X ^{\sharp}}  (U) _{\Q}
\twoheadrightarrow
D ^{(m)} _{\X ^{\sharp} \leftarrow \ZZ ^{\sharp}} (U) _{\Q}$.
Les éléments de 
$D ^{(m)} _{\X ^{\sharp} \leftarrow \ZZ ^{\sharp}} (U) _{\Q} \otimes _{D ^{(m)} _{\ZZ ^{\sharp}}  (U) _{\Q}}V$ s'écrivent de manière unique de la forme
$\sum _{\underline{k} \in \N ^{e}} \xi _{\underline{k}, (m)} \otimes x _{\underline{k}}$, la somme étant finie et 
$x _{\underline{k}}\in V$. 
D'après \ref{lemm-topo-tens-prod}, si on note 
$V _0$ le sous-$\V$-module de $V$ des éléments de norme inférieure ou égale à $1$
et $U _0$ le sous-$\V$-module de $D ^{(m)} _{\X ^{\sharp} \leftarrow \ZZ ^{\sharp}} (U) _{\Q} \otimes _{D ^{(m)} _{\ZZ ^{\sharp}}  (U) _{\Q}}V$
engendré par l'image canonique de
$D ^{(m)} _{\X ^{\sharp} \leftarrow \ZZ ^{\sharp}}  (U) \times V _0 \to D ^{(m)} _{\X ^{\sharp} \leftarrow \ZZ ^{\sharp}} (U) _{\Q} \otimes _{D ^{(m)} _{\ZZ ^{\sharp}}  (U) _{\Q}}V$, alors
une base de voisinage de 
$D ^{(m)} _{\X ^{\sharp} \leftarrow \ZZ ^{\sharp}} (U) _{\Q} \otimes _{D ^{(m)} _{\ZZ ^{\sharp}}  (U) _{\Q}}V$
est donnée par la famille
$(p ^{n} U _0) _{n\in\N}$.
On en déduit que 
la topologie canonique sur $D ^{(m)} _{\X ^{\sharp} \leftarrow \ZZ ^{\sharp}} (U) _{\Q} \otimes _{D ^{(m)} _{\ZZ ^{\sharp}}  (U) _{\Q}}V$
est induite par la norme
$\parallel \sum _{\underline{k} \in \N ^{e}} \xi _{\underline{k}, (m)} \otimes x _{\underline{k}} \parallel 
=
\max _{\underline{k}} \parallel x _{\underline{k}}  \parallel $. On a la même description pour $V' $ ou $V''$ à la place de $V$.
On obtient alors la suite exacte courte de $K$-espaces normés
\begin{equation}
\label{DotimesDhat-se11}
0 \to D ^{(m)} _{\X ^{\sharp} \leftarrow \ZZ ^{\sharp}} (U) _{\Q} \otimes _{D ^{(m)} _{\ZZ ^{\sharp}}  (U) _{\Q}} V'
\to 
D ^{(m)} _{\X ^{\sharp} \leftarrow \ZZ ^{\sharp}} (U) _{\Q} \otimes _{D ^{(m)} _{\ZZ ^{\sharp}}  (U) _{\Q}}V
\to
D ^{(m)} _{\X ^{\sharp} \leftarrow \ZZ ^{\sharp}} (U) _{\Q} \otimes _{D ^{(m)} _{\ZZ ^{\sharp}}  (U) _{\Q}}V''
\to 
0,
\end{equation}
dont tous les morphismes sont stricts (pour celle de l'injection, c'est évident d'après la description de leur norme ; pour celle de la surjection, on peut par exemple invoquer
\ref{epi-otimes-strict} ou bien faire le calcul).
Or, d'après \ref{MwidehatotimesN}, comme $V$ est un $\widetilde{D} ^{(m)}  _{\ZZ ^{\sharp},\Q}$-module de Banach,
le morphisme canonique
$D ^{(m)} _{\X ^{\sharp} \leftarrow \ZZ ^{\sharp}} (U) _{\Q} \widehat{\otimes} _{D ^{(m)} _{\ZZ ^{\sharp}}  (U) _{\Q}}V
\to 
\widetilde{D} ^{(m)}  _{\X ^{\sharp} \leftarrow \ZZ ^{\sharp},\Q}
\widehat{\otimes} _{\widetilde{D} ^{(m)}_{\ZZ ^{\sharp},\Q} }V$
est un isomorphisme. 
Comme le foncteur de séparée complétion transforme les suites exactes courtes dont les applications
sont des morphismes stricts de $K$-espaces normés en des suites exactes courtes dont les applications
sont des morphismes stricts, on obtient la suite exacte 
\begin{equation}
\label{DotimesDhat-se12}
0
\to 
\widetilde{D} ^{(m)}  _{\X ^{\sharp} \leftarrow \ZZ ^{\sharp},\Q} \widehat{\otimes} _{\widetilde{D} ^{(m)}_{\ZZ ^{\sharp},\Q} } V'
\to 
\widetilde{D} ^{(m)}  _{\X ^{\sharp} \leftarrow \ZZ ^{\sharp},\Q} \widehat{\otimes} _{\widetilde{D} ^{(m)}_{\ZZ ^{\sharp},\Q} }V
\to
\widetilde{D} ^{(m)}  _{\X ^{\sharp} \leftarrow \ZZ ^{\sharp},\Q} \widehat{\otimes} _{\widetilde{D} ^{(m)}_{\ZZ ^{\sharp},\Q} }V''
\to 0
\end{equation}
dont les morphismes sont stricts. 
 Comme $\widetilde{D} ^{(m)}  _{\X ^{\sharp} \leftarrow \ZZ ^{\sharp},\Q}$ est plat à gauche 
sur 
$\widetilde{D} ^{(m)}  _{\ZZ ^{\sharp},\Q}$, on dispose de la même suite exacte que \ref{DotimesDhat-se12} où l'on remplace 
$\widehat{\otimes} $
par 
$\otimes$.
Le morphisme surjectif de cette dernière suite exacte est strict grâce à \ref{epi-otimes-strict}.
Enfin, le caractère strict du morphisme injectif résulte quant à lui du fait qu'en le composant avec le monomorphisme strict
$\widetilde{D} ^{(m)}  _{\X ^{\sharp} \leftarrow \ZZ ^{\sharp},\Q}\otimes  _{\widetilde{D} ^{(m)}_{\ZZ ^{\sharp},\Q} } V
\to
\widetilde{D} ^{(m)}  _{\X ^{\sharp} \leftarrow \ZZ ^{\sharp},\Q} \widehat{\otimes} _{\widetilde{D} ^{(m)}_{\ZZ ^{\sharp},\Q} } V$,
on obtienne encore
un monomorphisme strict. 
 
 2) En décomposant les suites exactes en suites exactes courtes, on obtient le résultat.

\end{proof}

\begin{lemm}
\label{DwidehatD-inj}
Soient $m' \geq m$ deux entiers, 
$V$ un $\widetilde{D} ^{(m)}  _{\ZZ ^{\sharp},\Q}$-module de Banach
et
$V'$ un $\widetilde{D} ^{(m')} _{\ZZ ^{\sharp},\Q} $-module de Banach (voir la définition \ref{moduleloccvx}). 
Soit $\phi \colon V  \hookrightarrow V '$ un mnomorphisme de $\widetilde{D} ^{(m)}  _{\ZZ ^{\sharp},\Q}$-linéaire.
Le morphisme continu 
$\widetilde{D} ^{(m)}  _{\X ^{\sharp} \leftarrow \ZZ ^{\sharp},\Q} \widehat{\otimes} _{\widetilde{D} ^{(m)}_{\ZZ ^{\sharp},\Q} } V
\to
\widetilde{D} ^{(m')} _{\X ^{\sharp} \leftarrow \ZZ ^{\sharp},\Q} \widehat{\otimes} _{\widetilde{D} ^{(m')}_{\ZZ ^{\sharp},\Q} }V'$
canoniquement induit par $\phi$
est alors injectif.

\end{lemm}

\begin{proof}
Notons $\xi _{\underline{k}, (m')}$ l'image de $\underline{\partial} ^{<\underline{k}> _{(m')}}$ via la surjection canonique
$D ^{(m')} _{\X ^{\sharp}}  (U) _{\Q}
\twoheadrightarrow
D ^{(m')} _{\X ^{\sharp} \leftarrow \ZZ ^{\sharp}} (U) _{\Q}$.
D'après la preuve de \ref{DotimesDhat-se}, 
le $K$-espace de Banach
$\widetilde{D} ^{(m')} _{\X ^{\sharp} \leftarrow \ZZ ^{\sharp},\Q} \widehat{\otimes} _{\widetilde{D} ^{(m')}_{\ZZ ^{\sharp},\Q} }V'$
est le $K$-espace des éléments s'écrivant de manière unique de la forme
$\sum _{\underline{k} \in \N ^{e}} \xi _{\underline{k}, (m')} \otimes x _{\underline{k}}$, la somme étant infinie mais la suite des éléments
$x '_{\underline{k}}\in V'$ tendant vers zéro lorsque $| \underline{k}|$ tend vers l'infini.
Comme $\xi _{\underline{k}, (m)} = \lambda _{\underline{k}, (m,m')} \xi _{\underline{k}, (m')}$, pour un certain 
$ \lambda _{\underline{k}, (m,m')}  \in \V \setminus \{ 0\}$, 
le morphisme 
$\widetilde{D} ^{(m)}  _{\X ^{\sharp} \leftarrow \ZZ ^{\sharp}} (U) _{\Q} \widehat{\otimes} _{\widetilde{D} ^{(m)}_{\ZZ ^{\sharp},\Q} } V
\to
\widetilde{D} ^{(m')} _{\X ^{\sharp} \leftarrow \ZZ ^{\sharp}}  (U) _{\Q}\widehat{\otimes} _{\widetilde{D} ^{(m')}_{\ZZ ^{\sharp},\Q} }V'$
envoie 
$\sum _{\underline{k} \in \N ^{e}} \xi _{\underline{k}, (m)} \otimes x _{\underline{k}} $
sur 
$\sum _{\underline{k} \in \N ^{e}} \xi _{\underline{k}, (m')} \otimes  \lambda _{\underline{k}, (m,m')} \phi (x _{\underline{k}})$.  
D'où le résultat.
\end{proof}

Nous n'aurons pas besoin de la proposition \ref{prop-DotimesDhat-se} mais 
elle résulte immédiatement des lemmes \ref{DotimesDhat-se} et \ref{DwidehatD-inj} que nous utiliserons:
\begin{prop}
\label{prop-DotimesDhat-se}
Soit $0\to V ' \underset{\phi}{\longrightarrow} V \underset{\psi}{\longrightarrow} V ''$ une suite exacte 
de $\widetilde{D} ^{(m)}  _{\ZZ ^{\sharp},\Q}$-modules de Banach (voir la définition \ref{moduleloccvx}). 
On suppose de plus que $\phi$ est un morphisme strict.
La suite 
\begin{gather}
\label{prop-DotimesDhat-se1}
0 
\to 
\widetilde{D} ^{(m)}  _{\X ^{\sharp} \leftarrow \ZZ ^{\sharp},\Q} \widehat{\otimes} _{\widetilde{D} ^{(m)}_{\ZZ ^{\sharp},\Q} } V'
\underset{id \widehat{\otimes} \phi}{\longrightarrow}
\widetilde{D} ^{(m)}  _{\X ^{\sharp} \leftarrow \ZZ ^{\sharp},\Q} \widehat{\otimes} _{\widetilde{D} ^{(m)}_{\ZZ ^{\sharp},\Q} }V
\underset{id \widehat{\otimes} \psi}{\longrightarrow}
\widetilde{D} ^{(m)}  _{\X ^{\sharp} \leftarrow \ZZ ^{\sharp},\Q} \widehat{\otimes} _{\widetilde{D} ^{(m)}_{\ZZ ^{\sharp},\Q} }V''
 \end{gather}
est alors exacte et $id \widehat{\otimes} \phi$ est strict. 
\end{prop}

\begin{proof}
En munissant $W := V / V'$ de la topologie quotient, on obtient le morphisme $\widetilde{D} ^{(m)}  _{\ZZ ^{\sharp},\Q}$-linéaire, injectif et continue
$W \hookrightarrow V''$. On en déduit que $W$ est un $\widetilde{D} ^{(m)}  _{\ZZ ^{\sharp},\Q}$-module de Banach, 
car il est un quotient séparé de $V$. On applique alors respectivement les lemmes \ref{DotimesDhat-se} et \ref{DwidehatD-inj}
à la suite exacte
$0 \to V' \to V \to W \to 0$ et au monomorphisme 
$W \to V''$.
\end{proof}

\begin{vide}
Comme $\widetilde{D} ^{(m)}  _{\X ^{\sharp} \leftarrow \ZZ ^{\sharp}} \otimes _{\widetilde{D} ^{(m)}  _{\ZZ ^{\sharp}}} \widetilde{D} ^{(m)}  _{\ZZ ^{\sharp} \to \X ^{\sharp}}$
est séparé pour la topologie $p$-adique, 
le morphisme canonique
$\widetilde{D} ^{(m)}  _{\X ^{\sharp} \leftarrow \ZZ ^{\sharp}} \otimes _{\widetilde{D} ^{(m)}  _{\ZZ ^{\sharp}}} \widetilde{D} ^{(m)}  _{\ZZ ^{\sharp} \to \X ^{\sharp}}
\to
\widetilde{D} ^{(m)}  _{\X ^{\sharp} \leftarrow \ZZ ^{\sharp}} \widehat{\otimes} _{\widetilde{D} ^{(m)}  _{\ZZ ^{\sharp}}} \widetilde{D} ^{(m)}  _{\ZZ ^{\sharp} \to \X ^{\sharp}}$
est injectif. 
On munit 
$\widetilde{D} ^{(m)}  _{\X ^{\sharp} \leftarrow \ZZ ^{\sharp},\Q} \otimes _{\widetilde{D} ^{(m)}  _{\ZZ ^{\sharp},\Q}} \widetilde{D} ^{(m)}  _{\ZZ ^{\sharp} \to \X ^{\sharp},\Q}$
de la topologie produit tensoriel 
définie dans \ref{def-prod-tens}, i.e. 
d'après le lemme \ref{lemm-topo-tens-prod}, c'est la topologie dont une base de voisinages de zéro est donnée par la famille
$p ^n \widetilde{D} ^{(m)}  _{\X ^{\sharp} \leftarrow \ZZ ^{\sharp}} \otimes _{\widetilde{D} ^{(m)}  _{\ZZ ^{\sharp}}} \widetilde{D} ^{(m)}  _{\ZZ ^{\sharp} \to \X ^{\sharp}} $ avec $n $ parcourant $\N$. 
On munit naturellement
$(\widetilde{D} ^{(m)}  _{\X ^{\sharp} \leftarrow \ZZ ^{\sharp}} \widehat{\otimes} _{\widetilde{D} ^{(m)}  _{\ZZ ^{\sharp}}} \widetilde{D} ^{(m)}  _{\ZZ ^{\sharp} \to \X ^{\sharp}} )_\Q$
de la topologie dont une base de voisinages de zéro est donnée par la famille
$p ^n \widetilde{D} ^{(m)}  _{\X ^{\sharp} \leftarrow \ZZ ^{\sharp}} \widehat{\otimes} _{\widetilde{D} ^{(m)}  _{\ZZ ^{\sharp}}} \widetilde{D} ^{(m)}  _{\ZZ ^{\sharp} \to \X ^{\sharp}}$ avec $n $ parcourant $\N$, ce qui en fait un $K$-espace de Banach. 
On obtient alors le monomorphisme strict de $K$-espaces normés
\begin{equation}
\label{alpha-compl}
\widetilde{D} ^{(m)}  _{\X ^{\sharp} \leftarrow \ZZ ^{\sharp},\Q} \otimes _{\widetilde{D} ^{(m)}  _{\ZZ ^{\sharp},\Q}} \widetilde{D} ^{(m)}  _{\ZZ ^{\sharp} \to \X ^{\sharp},\Q}
\hookrightarrow
(\widetilde{D} ^{(m)}  _{\X ^{\sharp} \leftarrow \ZZ ^{\sharp}} \widehat{\otimes} _{\widetilde{D} ^{(m)}  _{\ZZ ^{\sharp}}} \widetilde{D} ^{(m)}  _{\ZZ ^{\sharp} \to \X ^{\sharp}} )_\Q.
\end{equation}
Par construction des séparées complétions de $K$-espaces localement connexes (e.g. voir la preuve de \cite[7.5]{Schneider-NonarchFuncAn}), 
on vérifie que le morphisme \ref{alpha-compl}
se factorise en l'isomorphisme canonique  (indépendant des coordonnées locales) de $K$-espaces de Banach: 
\begin{equation}
\label{alpha-hat}
\widetilde{D} ^{(m)}  _{\X ^{\sharp} \leftarrow \ZZ ^{\sharp},\Q} \widehat{\otimes} _{\widetilde{D} ^{(m)}  _{\ZZ ^{\sharp},\Q}} \widetilde{D} ^{(m)}  _{\ZZ ^{\sharp} \to \X ^{\sharp},\Q}
\riso 
(\widetilde{D} ^{(m)}  _{\X ^{\sharp} \leftarrow \ZZ ^{\sharp}} \widehat{\otimes} _{\widetilde{D} ^{(m)}  _{\ZZ ^{\sharp}}} \widetilde{D} ^{(m)}  _{\ZZ ^{\sharp} \to \X ^{\sharp}} )_\Q.
\end{equation}

\end{vide}

\subsection{Stabilité de la cohérence par foncteur cohomologique local en degré zéro}

\begin{vide}
[Rappels]
Les foncteurs 
$\R \underline{\Gamma} ^{\dag} _{Z}$
et 
$u _+ ^{\sharp (\bullet)}  \circ u ^{\sharp (\bullet) !}
\colon 
\smash{\underrightarrow{LD}} ^{\mathrm{b}} _{\Q ,\mathrm{qc}}
(\widetilde{\D} _{\X ^\sharp } ^{(\bullet)})
\to
\smash{\underrightarrow{LD}} ^{\mathrm{b}} _{\Q ,\mathrm{qc}}
(\widetilde{\D} _{\X ^\sharp } ^{(\bullet)})
$
sont isomorphes (voir \cite[5.3.8]{caro-stab-sys-ind-surcoh}).
On dispose aussi du foncteur 
$\R \underline{\Gamma} ^{\dag} _{Z}\colon D ^\mathrm{b} _\mathrm{coh} ( \smash{\D} ^\dag _{\X ^{\sharp}} (\hdag T) _{\Q} )
\to 
D ^\mathrm{b}  ( \smash{\D} ^\dag _{\X ^{\sharp}} (\hdag T) _{\Q} )$
défini de telle sorte que l'on ait l'isomorphisme de foncteurs 
$\R \underline{\Gamma} ^{\dag} _{Z} \circ \underrightarrow{\lim} 
\riso \underrightarrow{\lim} \circ \R \underline{\Gamma} ^{\dag} _{Z}
\colon 
\smash{\underrightarrow{LD}} ^{\mathrm{b}} _{\Q ,\mathrm{coh}}
(\widetilde{\D} _{\X ^\sharp } ^{(\bullet)})
\to
D ^\mathrm{b}  ( \smash{\D} ^\dag _{\X ^{\sharp}} (\hdag T) _{\Q} )$,
où
$\underrightarrow{\lim}$
désigne  
désigne le foncteur canonique 
$\underrightarrow{\lim}  \colon
\smash{\underrightarrow{LD}} ^{\mathrm{b}} _{\Q ,\mathrm{qc}}
(\widetilde{\D} _{\X ^\sharp } ^{(\bullet)})
\to
D ^\mathrm{b}  ( \smash{\D} ^\dag _{\X ^{\sharp}} (\hdag T) _{\Q} )$
qui induit l'équivalence de catégories
$\underrightarrow{\lim}  \colon
\smash{\underrightarrow{LD}} ^{\mathrm{b}} _{\Q ,\mathrm{coh}}
(\widetilde{\D} _{\X ^\sharp } ^{(\bullet)})
\cong
D ^\mathrm{b} _\mathrm{coh} ( \smash{\D} ^\dag _{\X ^{\sharp}} (\hdag T) _{\Q} )$
(voir \cite[2]{caro-stab-sys-ind-surcoh} ou \cite[4.2.4]{Beintro2} énoncé sans structure logarithmique).
\end{vide}

Pour établir la proposition qui suit, nous aurons besoin des deux lemmes suivants:
\begin{lemm}
[Berthelot-Kashiwara]
\label{BK-cohloc}
Soit $\E $ un $\smash{\D} ^\dag _{\X ^{\sharp}} (\hdag T) _{\Q}$-module cohérent à support dans $\ZZ$.
\begin{enumerate}
\item Le complexe $u ^{!\sharp} (\E)$ est isomorphe à $\mathcal{H} ^{0}u ^{!\sharp} (\E)$, ce dernier étant un $\smash{\D} ^\dag _{\ZZ ^\sharp} (\hdag U) _{\Q}$-module cohérent.
De plus, 
le morphisme canonique
$u _+ ^\sharp  \circ u ^{!\sharp} (\E) \to \E$ est un isomorphisme. 

\item Le morphisme canonique
$\R \underline{\Gamma} ^{\dag} _{Z} (\E) \to \E$
est un isomorphisme.
\end{enumerate}

\end{lemm}

\begin{proof}
La première assertion est exactement la version arithmétique de Berthelot du théorème de Kashiwara. 
La seconde assertion étant locale, on peut supposer que nous sommes dans la situation de \ref{nota-coor-locales} dont nous reprenons les notations.
Pour $i = 1, \dots, e$, notons $\ZZ _i:= V (t _i)$.
Comme $\E (\hdag Z _i)$ est un $\D ^\dag _{\X ^{\sharp}} (\hdag Z _i \cup T) _\Q$-module cohérent nul en dehors de $\ZZ _i$, 
ce dernier est nul. De même, pour $1\leq i,j\leq e$, 
comme $\E (\hdag Z _i \cup Z _j )$ est un $\D ^\dag _{\X ^{\sharp}} (\hdag Z _i \cup Z _j  \cup T) _\Q$-module cohérent nul en dehors de $Z _i \cup Z _j $, 
ce dernier est nul. Or, via le triangle de localisation de Mayer-Vietoris (voir \cite[4.4.5]{caro-stab-sys-ind-surcoh}),
on en déduit que 
$\E (\hdag Z _i \cap Z _j )=0$. En réitérant le procédé, on obtient que 
$\E (\hdag Z  )=0$. On conclut grâce au triangle distingué de localisation en $Z$ de $\E$.

\end{proof}

\begin{lemm}
\label{H0Zinj}
Soient $Y$ un sous-schéma fermé de $X$, 
$\E $ un $\smash{\D} ^\dag _{\X ^{\sharp}} (\hdag T) _{\Q}$-module cohérent.
Le morphisme canonique
$\mathcal{H} ^{\dag 0} _{Y} (\E)
\to 
\E$
est injectif.
\end{lemm}

\begin{proof}
Via le triangle de localisation de $\E$ par rapport à $Y$, il s'agit d'établir 
$\mathcal{H} ^{-1} (\E (\hdag Y) ) =0$. Lorsque $Y$ est un diviseur, 
cela résulte de l'exactitude du foncteur $(\hdag Y)$ sur la catégorie des
$\smash{\D} ^\dag _{\X ^{\sharp}} (\hdag T) _{\Q}$-modules cohérents.
On en déduit le cas général via les triangles distingués de Mayer-Vietoris (voir \cite[4.4.5]{caro-stab-sys-ind-surcoh})
en procédant par récurrence
sur le nombre minimal de diviseurs dont l'intersection donne $Y$.
\end{proof}

\begin{prop}
\label{stab-coh-loc-H0}
Soit $\E $ un $\smash{\D} ^\dag _{\X ^{\sharp}} (\hdag T) _{\Q}$-module cohérent.
Alors, $\mathcal{H} ^{\dag 0} _{Z} (\E)$ est un $\smash{\D} ^\dag _{\X ^{\sharp}} (\hdag T) _{\Q}$-module cohérent
si et seulement si 
$\mathcal{H} ^{0}u ^{!\sharp} (\E)$ est un $\smash{\D} ^\dag _{\ZZ ^\sharp} (\hdag U) _{\Q}$-module cohérent.
Si l'une de ces conditions est satisfaite, on a alors 
$u ^{\sharp} _{+}\mathcal{H} ^{0} u ^{!\sharp} (\E) \riso \mathcal{H} ^{\dag 0} _{Z} (\E)$.
\end{prop}

\begin{proof}
Supposons que $\mathcal{H} ^{0}u ^{!\sharp} (\E)$ est un $\smash{\D} ^\dag _{\ZZ ^\sharp} (\hdag U) _{\Q}$-module cohérent.
On dispose alors du morphisme canonique de $\smash{\D} ^\dag _{\X ^{\sharp}} (\hdag T) _{\Q}$-modules cohérents de la forme
$\phi \colon u ^{\sharp} _{+}\mathcal{H} ^{0} u ^{!\sharp} (\E) \to \E$. 
Comme le noyau de $\phi$ est un $\smash{\D} ^\dag _{\X ^{\sharp}} (\hdag T) _{\Q}$-module cohérent à support dans $Z$, comme $\mathcal{H} ^{0} u ^{!\sharp}$ est exact à gauche (sur la catégorie des $\smash{\D} ^\dag _{\X ^{\sharp}} (\hdag T) _{\Q}$-modules cohérents) 
et comme le noyau de 
$\mathcal{H} ^{0} u ^{!\sharp}(\phi ) $ est nul, on déduit alors du théorème de Berthelot-Kashiwara 
(voir la première partie de \ref{BK-cohloc}) 
que $\phi$ est injectif. 

Comme le morphisme canonique  
$\mathcal{H} ^{\dag 0} _{Z} (u ^{\sharp} _{+}\mathcal{H} ^{0} u ^{!\sharp} (\E)) 
\to u ^{\sharp} _{+}\mathcal{H} ^{0} u ^{!\sharp} (\E)$ est un isomorphisme
(voir la seconde partie de \ref{BK-cohloc}) , 
on en déduit que l'injection canonique
$u ^{\sharp} _{+}\mathcal{H} ^{0} u ^{!\sharp} (\E) \hookrightarrow \E$
se factorise par 
les inclusions
$u ^{\sharp} _{+}\mathcal{H} ^{0} u ^{!\sharp} (\E) 
\hookrightarrow 
\mathcal{H} ^{\dag 0} _{Z} (\E)
\hookrightarrow  \E$
(la seconde flèche est bien injective grâce à \ref{H0Zinj}).
Par l'absurde, l'inclusion
$u ^{\sharp} _{+}\mathcal{H} ^{0} u ^{!\sharp} (\E) 
\hookrightarrow 
\mathcal{H} ^{\dag 0} _{Z} (\E)$
n'est pas un isomorphisme. 
Dans ce cas, il existe un ouvert affine $\U$ de $\X$, une section $s$ sur $\U$
de $\mathcal{H} ^{\dag 0} _{Z} (\E)$ qui n'est pas une section sur $\U$ de 
$u ^{\sharp} _{+}\mathcal{H} ^{0} u ^{!\sharp} (\E) $. 
Soit $\G$ le sous-$\smash{\D} ^\dag _{\U^{\sharp}} (\hdag T\cap U) _{\Q}$-module de 
$\mathcal{H} ^{\dag 0} _{Z} (\E) |\U$ engendré par 
$u ^{\sharp} _{+}\mathcal{H} ^{0} u ^{!\sharp} (\E) |\U$ et par la section $s$.
Comme le faisceau associé à un sous-préfaisceau d'un faisceau contient le sous-préfaisceau,
$\G$ contient strictement $u ^{\sharp} _{+}\mathcal{H} ^{0} u ^{!\sharp} (\E) |\U$. De plus, $\G$ est à support dans $\U \cap Z$ car
$\mathcal{H} ^{\dag 0} _{Z} (\E) |\U$ l'est. 
Comme $\G$ est un sous-$\smash{\D} ^\dag _{\U^{\sharp}} (\hdag T\cap U) _{\Q}$-module de type fini du $\smash{\D} ^\dag _{\U^{\sharp}} (\hdag T\cap U) _{\Q}$-module 
cohérent $\E |\U$, $\G$ est alors un sous-$\smash{\D} ^\dag _{\U^{\sharp}} (\hdag T\cap U) _{\Q}$-module cohérent. 
Or, en appliquant le foncteur exact à gauche $\mathcal{H} ^{0} u ^{!\sharp}$ aux inclusions
$u ^{\sharp} _{+}\mathcal{H} ^{0} u ^{!\sharp} (\E) |\U \hookrightarrow \G \hookrightarrow \E |\U$, 
on obtient 
les isomorphismes
$\mathcal{H} ^{0} u ^{!\sharp} (\E) |\U 
\riso \mathcal{H} ^{0} u ^{!\sharp} (\G) |\U
\riso 
\mathcal{H} ^{0} u ^{!\sharp} (\E) |\U$.
Via le théorème de Berthelot-Kashiwara, le premier de ces deux isomorphismes entraîne alors que
$u ^{\sharp} _{+}\mathcal{H} ^{0} u ^{!\sharp} (\E) |\U \riso \G$, 
ce qui est une contradiction. 
\end{proof}

\subsection{Stabilité de la cohérence par foncteur cohomologique local en degré maximal}

\begin{nota}
\label{faisceautisation-hat}
Soit $\mathfrak{B}$ la base de voisinages de $\X$ des ouverts affines et munis de coordonnées locales. 
Pour tout $\U \in \mathfrak{B}$, on note $\U ^{\sharp}:= (\U, \U \cap \mathfrak{D})$. 
\begin{itemize}
\item On note 
$\widetilde{\D} ^{(m)} _{\X ^{\sharp} \leftarrow \ZZ ^{\sharp},\Q} \widehat{\otimes} _{\widetilde{\D} ^{(m)} _{\ZZ^{\sharp},\Q}} \widetilde{\D} ^{(m)} _{\ZZ ^{\sharp} \to \X ^{\sharp},\Q}$
le faisceau (d'ensemble) sur $\X$ associé
au préfaisceau sur $\mathfrak{B}$ définie par
$\U \in \mathfrak{B} \mapsto 
\widetilde{D} ^{(m)} _{\U ^{\sharp}\leftarrow \ZZ ^{\sharp}\cap \U^{\sharp},\Q} \widehat{\otimes} _{\widetilde{D} ^{(m)}  _{\ZZ ^{\sharp}\cap \U^{\sharp},\Q}} \widetilde{D} ^{(m)}  _{\ZZ ^{\sharp}\cap \U^{\sharp} \to \U^{\sharp},\Q}$, 
les morphismes de restriction étant les morphismes canoniques (la séparée complétion est un foncteur).

\item Les morphismes canoniques 
$\widetilde{D} ^{(m)} _{\U^{\sharp}  \leftarrow \ZZ ^{\sharp} \cap \U^{\sharp} ,\Q} \otimes _{\widetilde{D} ^{(m)}  _{\ZZ ^{\sharp}\cap \U^{\sharp} ,\Q}} \widetilde{D} ^{(m)}  _{\ZZ ^{\sharp}\cap \U ^{\sharp} \to \U^{\sharp} ,\Q}
\to
\widetilde{D} ^{(m)} _{\U ^{\sharp} \leftarrow \ZZ^{\sharp}  \cap \U^{\sharp} ,\Q} \widehat{\otimes} _{\widetilde{D} ^{(m)}  _{\ZZ ^{\sharp}\cap \U^{\sharp} ,\Q}} \widetilde{D} ^{(m)}  _{\ZZ ^{\sharp}\cap \U ^{\sharp} \to \U^{\sharp} ,\Q}$
sont fonctoriels en $\U$,
on obtient le morphisme canonique de faisceaux:
\begin{equation}
\label{faisceautisation-hat-adj}
\widetilde{\D} ^{(m)} _{\X ^{\sharp} \leftarrow \ZZ ^{\sharp},\Q}\otimes _{\widetilde{\D} ^{(m)} _{\ZZ ^\sharp,\Q}} \widetilde{\D} ^{(m)} _{\ZZ ^{\sharp} \to \X ^{\sharp},\Q}
\to
\widetilde{\D} ^{(m)} _{\X ^{\sharp} \leftarrow \ZZ ^{\sharp},\Q} \widehat{\otimes} _{\widetilde{\D} ^{(m)} _{\ZZ ^\sharp,\Q}} \widetilde{\D} ^{(m)} _{\ZZ ^{\sharp} \to \X ^{\sharp},\Q}.
\end{equation}

\item Soit $\alpha ^{(m)}\colon (\widetilde{\D} ^{(m )} _{\X^{\sharp},\Q}) ^{r} 
\to 
(\widetilde{\D} ^{(m)} _{\X ^{\sharp},\Q}) ^{s}$ un morphisme $\widetilde{\D} ^{(m)} _{\X ^{\sharp},\Q}$-linéaire à gauche.
Les morphismes 
$$id \widehat{\otimes} \Gamma (\U \cap \ZZ, u ^{*} (\alpha ^{(m)})) \colon 
\widetilde{D} ^{(m)} _{\U^{\sharp}  \leftarrow \ZZ ^{\sharp} \cap \U^{\sharp} ,\Q}\widehat{\otimes}  _{\widetilde{D} ^{(m)}  _{\ZZ ^{\sharp}\cap \U^{\sharp} ,\Q}}
(\widetilde{D} ^{(m)}  _{\ZZ ^{\sharp}\cap \U^{\sharp}  \to \U^{\sharp} ,\Q} ) ^{r}
\to 
\widetilde{D} ^{(m)} _{\U^{\sharp}  \leftarrow \ZZ ^{\sharp} \cap \U^{\sharp} ,\Q} \widehat{\otimes} _{\widetilde{D} ^{(m)}  _{\ZZ ^{\sharp}\cap \U^{\sharp} ,\Q}}
(\widetilde{D} ^{(m)}  _{\ZZ ^{\sharp}\cap \U ^{\sharp} \to \U^{\sharp} ,\Q} ) ^{s}$$
\end{itemize}
sont fonctoriels en $\U$
et induisent donc le morphisme
de faisceaux: 
\begin{equation}
\label{id-otimes-u*alpha}
id \widehat{\otimes} u ^{*} (\alpha ^{(m)})
\colon 
\widetilde{\D} ^{(m)} _{\X ^{\sharp} \leftarrow \ZZ ^{\sharp},\Q} \widehat{\otimes} _{\widetilde{\D} ^{(m)} _{\ZZ ^\sharp,\Q}} (\widetilde{\D} ^{(m)} _{\ZZ ^{\sharp} \to \X ^{\sharp},\Q} ) ^{r}
\to 
\widetilde{\D} ^{(m)} _{\X ^{\sharp} \leftarrow \ZZ ^{\sharp},\Q} \widehat{\otimes} _{\widetilde{\D} ^{(m)} _{\ZZ ^\sharp,\Q}} (\widetilde{\D} ^{(m)} _{\ZZ ^{\sharp} \to \X ^{\sharp},\Q}) ^{s}.
\end{equation}

\end{nota}

\begin{lemm}
\begin{enumerate}
\item Avec les notations de \ref{faisceautisation-hat}, on dispose du diagramme canonique commutatif de faisceaux:
\begin{equation}
\label{alpha-hat-faisc}
\xymatrix @R=0,3cm {
{\widetilde{\D} ^{(m)} _{\X ^{\sharp} \leftarrow \ZZ ^{\sharp},\Q} \widehat{\otimes} _{\widetilde{\D} ^{(m)} _{\ZZ ^\sharp,\Q}} \widetilde{\D} ^{(m)} _{\ZZ ^{\sharp} \to \X ^{\sharp},\Q}} 
\ar@{.>}[rr] ^-{\sim}
&
& 
{(\widetilde{\D} ^{(m)} _{\X ^{\sharp} \leftarrow \ZZ ^{\sharp}} \widehat{\otimes} _{\widetilde{\D} ^{(m)} _{\ZZ ^\sharp}} \widetilde{\D} ^{(m)} _{\ZZ ^{\sharp} \to \X ^{\sharp}} )_\Q } 
\\ 
&
{\widetilde{\D} ^{(m)} _{\X ^{\sharp} \leftarrow \ZZ ^{\sharp},\Q} \otimes _{\widetilde{\D} ^{(m)} _{\ZZ ^\sharp,\Q}} \widetilde{\D} ^{(m)} _{\ZZ ^{\sharp} \to \X ^{\sharp},\Q}} 
\ar[ul] ^-{\ref{faisceautisation-hat-adj}}
\ar[ur] ^-{}
& { }
 }
\end{equation}
dont la flèche du haut est un isomorphisme.

\item Soient $\epsilon ^{(m)}\colon (\widetilde{\D} ^{(m)} _{\X ^{\sharp}}) ^{r} 
\to 
(\widetilde{\D} ^{(m)} _{\X ^{\sharp}}) ^{s}$ un morphisme 
$\widetilde{\D} ^{(m)} _{\X ^{\sharp}}$-linéaire à gauche
et
$\alpha ^{(m)}\colon (\widetilde{\D} ^{(m )} _{\X^{\sharp},\Q}) ^{r} 
\to 
(\widetilde{\D} ^{(m)} _{\X ^{\sharp},\Q}) ^{s}$ le morphisme induit par tensorisation par $\Q$.
On dispose alors du carré commutatif
\begin{equation}
\label{pre-diag-comm-u+-c=HnZ}
\xymatrix @ C=2cm @R=0,3cm{
{(\widetilde{\D} ^{(m)} _{\X ^{\sharp} \leftarrow \ZZ ^{\sharp}} \widehat{\otimes} _{\widetilde{\D} ^{(m)} _{\ZZ ^\sharp}} (\widetilde{\D} ^{(m)} _{\ZZ ^{\sharp} \to \X ^{\sharp}} ) ^r )_\Q } 
\ar[r] ^-{(id \widehat{\otimes} u ^*\epsilon ^{(m)}) _\Q}
& 
{(\widetilde{\D} ^{(m)} _{\X ^{\sharp} \leftarrow \ZZ ^{\sharp}} \widehat{\otimes} _{\widetilde{\D} ^{(m)} _{\ZZ ^\sharp}} (\widetilde{\D} ^{(m)} _{\ZZ ^{\sharp} \to \X ^{\sharp}} ) ^s )_\Q } 
\\ 
{ \widetilde{\D} ^{(m)} _{\X ^{\sharp} \leftarrow \ZZ ^{\sharp},\Q} \widehat{\otimes} _{\widetilde{\D} ^{(m)}_{\ZZ ^{\sharp},\Q} }
(u ^* \widetilde{\D} ^{(m)} _{\X ^{\sharp},\Q}) ^{r} }
\ar[r] ^-{ id \widehat{\otimes} u ^*\alpha ^{(m)}} 
\ar[u] ^-{\sim}
& 
{\widetilde{\D} ^{(m)} _{\X ^{\sharp} \leftarrow \ZZ ^{\sharp},\Q} \widehat{\otimes} _{\widetilde{\D} ^{(m)}_{\ZZ ^{\sharp},\Q} }
(u ^* \widetilde{\D} ^{(m)} _{\X ^{\sharp},\Q}) ^{s} ,} 
\ar[u] ^-{\sim}
} 
\end{equation}
dont les isomorphismes verticaux sont induites par la factorisation de \ref{alpha-hat-faisc}.
\end{enumerate}

\end{lemm}

\begin{proof}
Soit $\U \in \mathfrak{B}$.
Pour tout entier positif $i$, on note $X _i$, $X ^{\sharp}_i$, $Z _i$, $Z ^{\sharp}_i$,
$U _i$, $U ^{\sharp}_i$ les réductions modulo
$\pi ^{i+1}$ de respectivement $\X$, $\X ^{\sharp}$, $\ZZ$, $\ZZ ^{\sharp}$,
$\U$, $\U ^{\sharp}$. 
 On dispose des $\B ^{(m)} _{Z _i\cap U _i} (T \cap Z _i\cap U _i)$-modules quasi-cohérents
$\widetilde{\D} ^{(m)} _{U ^{\sharp}_i \leftarrow Z ^{\sharp}_i\cap U ^{\sharp}_i}:= 
\widetilde{\D} ^{(m)} _{\U ^{\sharp}\leftarrow \ZZ^{\sharp} \cap \U^{\sharp}} \otimes _{\V} \V / \pi ^{i+1} $,
$\widetilde{\D} ^{(m)} _{Z^{\sharp} _i\cap U ^{\sharp}_i}:=
\widetilde{\D} ^{(m)} _{\ZZ ^{\sharp}\cap \U^{\sharp}}\otimes _{\V} \V / \pi ^{i+1} $, 
$\widetilde{\D} ^{(m)} _{Z ^{\sharp}_i\cap U ^{\sharp}_i \to U ^{\sharp}_i}:=
\widetilde{\D} ^{(m)} _{\ZZ^{\sharp}\cap \U ^{\sharp}\to \U ^{\sharp}} \otimes _{\V} \V / \pi ^{i+1} $.
Comme le foncteur sections globales commute aux limites projectives et 
par quasi-cohérence de nos faisceaux sur $X _i$, on obtient les isomorphismes canoniques
\begin{gather}
\notag
\Gamma (\U, \widetilde{\D} ^{(m)} _{\X ^{\sharp} \leftarrow \ZZ ^{\sharp}} \widehat{\otimes} _{\widetilde{\D} ^{(m)} _{\ZZ ^\sharp}} \widetilde{\D} ^{(m)} _{\ZZ ^{\sharp} \to \X ^{\sharp}} )
=
\Gamma 
(\U, \widetilde{\D} ^{(m)} _{\U ^{\sharp}\leftarrow \ZZ ^{\sharp} \cap \U ^{\sharp}} \widehat{\otimes} _{\widetilde{\D} ^{(m)} _{\ZZ^{\sharp}\cap \U^{\sharp}}} \widetilde{\D} ^{(m)} _{\ZZ^{\sharp}\cap \U ^{\sharp} \to \U^{\sharp}})
\riso
\underleftarrow{\lim} _i 
\Gamma 
(U _i, \widetilde{\D} ^{(m)} _{U ^{\sharp}_i \leftarrow Z ^{\sharp}_i\cap U ^{\sharp}_i} \otimes _{\widetilde{\D} ^{(m)} _{Z ^{\sharp}_i\cap U ^{\sharp}_i}} \widetilde{\D} ^{(m)} _{Z ^{\sharp}_i\cap U ^{\sharp}_i \to U ^{\sharp}_i})
\\
\riso
\underleftarrow{\lim} _i \widetilde{D} ^{(m)} _{U^{\sharp} _i \leftarrow Z ^{\sharp}_i\cap U ^{\sharp}_i} 
\otimes _{\widetilde{D} ^{(m)} _{Z ^{\sharp}_i\cap U ^{\sharp}_i}} \widetilde{D} ^{(m)} _{Z ^{\sharp}_i\cap U ^{\sharp}_i \to U ^{\sharp}_i}
\riso
\widetilde{D} ^{(m)} _{\U ^{\sharp}\leftarrow \ZZ^{\sharp} \cap \U^{\sharp}} \widehat{\otimes} _{\widetilde{D} ^{(m)}  _{\ZZ ^{\sharp}\cap \U^{\sharp}}} \widetilde{D} ^{(m)}  _{\ZZ ^{\sharp}\cap \U ^{\sharp}\to \U^{\sharp}}.
\end{gather}
Or, pour tout $\U \in \mathfrak{B}$, 
le foncteur $\Gamma (\U, -)$ commute au produit tensoriel par $\Q$. 
On en déduit alors que  le faisceau 
$(\widetilde{\D} ^{(m)} _{\X ^{\sharp} \leftarrow \ZZ ^{\sharp}} \widehat{\otimes} _{\widetilde{\D} ^{(m)} _{\ZZ^{\sharp} }} \widetilde{\D} ^{(m)} _{\ZZ ^{\sharp} \to \X ^{\sharp}} )_\Q$
sur $\X$
est associé au faisceau sur $\mathfrak{B}$ défini par 
$\U \in \mathfrak{B} \mapsto 
(\widetilde{D} ^{(m)} _{\U^{\sharp} \leftarrow \ZZ^{\sharp} \cap \U^{\sharp}} \widehat{\otimes} _{\widetilde{D} ^{(m)}  _{\ZZ ^{\sharp}\cap \U^{\sharp}}} \widetilde{D} ^{(m)}  _{\ZZ ^{\sharp}\cap \U ^{\sharp}\to \U^{\sharp}} ) _{\Q}$.
Les isomorphismes canoniques \ref{alpha-hat} nous permettent de conclure le premier point.
La seconde assertion se vérifie de même facilement.
\end{proof}

\begin{rema}
Avec les notations de \ref{faisceautisation-hat}, 
le faisceau sur $\mathfrak{B}$ en $K$-espaces vectoriels
$\widetilde{\D} ^{(m)} _{\X ^{\sharp} \leftarrow \ZZ ^{\sharp},\Q} \widehat{\otimes} _{\widetilde{\D} ^{(m)} _{\ZZ ^\sharp,\Q}} \widetilde{\D} ^{(m)} _{\ZZ ^{\sharp} \to \X ^{\sharp},\Q}$
est aussi un faisceau sur $\mathfrak{B}$ à valeur dans la catégorie des 
$K$-espaces topologiques (on vérifie que les morphismes de restriction sont injectifs et stricts).
\end{rema}

\begin{lemm}
\label{lemm-u^*coh}
On suppose que $u\colon \ZZ  \hookrightarrow \X$ est 
de codimension pure $e$.
Soit $\alpha$ un morphisme de $\smash{\D} ^\dag _{\X ^{\sharp}} (\hdag T) _{\Q}$-modules à gauche de la forme
$\alpha \colon (\smash{\D} ^\dag _{\X ^{\sharp}} (\hdag T) _{\Q}) ^{r} \to (\smash{\D} ^\dag _{\X ^{\sharp}} (\hdag T) _{\Q}) ^{s}$.
Soit $m _0$ assez grand tel qu'il existe  
un morphisme $\widetilde{\D} ^{(m _0)} _{\X ^{\sharp},\Q}$-linéaire à gauche 
de la forme
$\alpha ^{(m _0)}\colon (\widetilde{\D} ^{(m _0)} _{\X ^{\sharp},\Q}) ^{r} 
\to 
(\widetilde{\D} ^{(m _0)} _{\X ^{\sharp},\Q}) ^{s}$
induisant $\alpha$ par extension
via 
$\widetilde{\D} ^{(m _0)} _{\X ^{\sharp},\Q} \to \smash{\D} ^\dag _{\X ^{\sharp}} (\hdag T) _{\Q} $.
Pour $m \geq m _0$, 
on note $\alpha ^{(m)}\colon (\widetilde{\D} ^{(m )} _{\X^{\sharp},\Q}) ^{r} 
\to 
(\widetilde{\D} ^{(m)} _{\X^{\sharp},\Q}) ^{s}$, 
le morphisme induit par extension de $\alpha ^{(m _0)}$.

On dispose alors du diagramme commutatif 
\begin{equation}
\label{diag-comm-u+-c=HnZ}
\xymatrix @ C=2cm @R=0,3cm {
{\mathcal{H} ^{\dag e} _{Z}
((\smash{\D} ^\dag _{\X ^{\sharp}} (\hdag T) _{\Q}) ^{r})} 
\ar[r] ^-{\mathcal{H} ^{\dag e} _{Z} (\alpha)}
& 
{\mathcal{H} ^{\dag e} _{Z}
((\smash{\D} ^\dag _{\X ^{\sharp}} (\hdag T) _{\Q}) ^{s})}
\\ 
{ \underrightarrow{\lim} _m  
 \widetilde{\D} ^{(m)} _{\X ^{\sharp} \leftarrow \ZZ ^{\sharp},\Q} \widehat{\otimes} _{\widetilde{\D} ^{(m)}_{\ZZ ^{\sharp},\Q} }
(u ^* \widetilde{\D} ^{(m)} _{\X^{\sharp},\Q}) ^{r} }
\ar[r] ^-{ \underrightarrow{\lim} _m   (id \widehat{\otimes} u ^*\alpha ^{(m)})} 
\ar[u] ^-{\sim}
& 
{ \underrightarrow{\lim} _m  
 \widetilde{\D} ^{(m)} _{\X ^{\sharp} \leftarrow \ZZ ^{\sharp},\Q} \widehat{\otimes} _{\widetilde{\D} ^{(m)}_{\ZZ ^{\sharp},\Q} }
(u ^* \widetilde{\D} ^{(m)} _{\X^{\sharp},\Q}) ^{s} } 
\ar[u] ^-{\sim}
} 
\end{equation}
dont les flèches verticales sont des isomorphismes. 
\end{lemm}

\begin{proof}
Quitte à multiplier $\alpha ^{(m _0)}$ par une puissance de $p$ assez grande, on peut supposer que
$\alpha ^{(m _0)}$ provient par extension d'un morphisme
$\widetilde{\D} ^{(m _0)} _{\X ^{\sharp}}$-linéaire à gauche 
de la forme
$\epsilon ^{(m _0)}\colon (\widetilde{\D} ^{(m _0)} _{\X ^{\sharp}}) ^{r} 
\to 
(\widetilde{\D} ^{(m _0)} _{\X ^{\sharp}}) ^{s}$.
Notons $\epsilon ^{(m)}\colon (\widetilde{\D} ^{(m)} _{\X ^{\sharp}}) ^{r} 
\to 
(\widetilde{\D} ^{(m)} _{\X ^{\sharp}}) ^{s}$
les morphismes induits par extension. 
On obtient ainsi le foncteur dans 
$\smash{\underrightarrow{LD}} ^{\mathrm{b}} _{\Q ,\mathrm{coh}}
(\widetilde{\D} _{\X ^\sharp } ^{(\bullet)})
$
de la forme
$$\epsilon ^{(\bullet+m_0)}
\colon 
(\widetilde{\D} ^{(\bullet+m _0)} _{\X ^{\sharp}}) ^{r} 
\to 
(\widetilde{\D} ^{(\bullet+m _0)} _{\X ^{\sharp}}) ^{s}.$$
Or,
$$\underrightarrow{\lim}   \circ u _+ ^{\sharp (\bullet)}  \circ u ^{*} (\epsilon ^{(\bullet+m_0)})
\colon 
\underrightarrow{\lim} _m
\,
\left (\widetilde{\D} ^{(m)} _{\X ^{\sharp} \leftarrow \ZZ ^{\sharp}} \widehat{\otimes} _{\widetilde{\D} ^{(m)} _{\ZZ ^\sharp}}( \widetilde{\D} ^{(m)} _{\ZZ ^{\sharp} \to \X ^{\sharp}} ) ^{r} \right )_\Q  
\to 
\underrightarrow{\lim} _m
\,
\left (\widetilde{\D} ^{(m)} _{\X ^{\sharp} \leftarrow \ZZ ^{\sharp}} \widehat{\otimes} _{\widetilde{\D} ^{(m)} _{\ZZ ^\sharp}}( \widetilde{\D} ^{(m)} _{\ZZ ^{\sharp} \to \X ^{\sharp}} )^{s} \right )_\Q  
.$$
Or, comme par définition
$\underrightarrow{\lim} = \underrightarrow{\lim} _m \circ (\Q \otimes _{\Z} -)$, 
en appliquant le foncteur $ \underrightarrow{\lim} _m $
au carré commutatif \ref{pre-diag-comm-u+-c=HnZ}, on obtient que 
$\underrightarrow{\lim}   \circ u _+ ^{\sharp (\bullet)}  \circ u ^{*} (\epsilon ^{(\bullet+m_0)})$
et 
$ \underrightarrow{\lim} _m   (id \widehat{\otimes} u ^*\alpha ^{(m)})$ sont canoniquement isomorphes.
Comme par définition
$\R \underline{\Gamma} ^{\dag} _{Z} (\alpha) =\underrightarrow{\lim} \circ \R \underline{\Gamma} ^{\dag} _{Z} 
(\epsilon ^{(\bullet+m_0)})$,
comme les deux foncteurs 
$\R \underline{\Gamma} ^{\dag} _{Z}$
et 
$u _+ ^{\sharp (\bullet)}  \circ u ^{\sharp (\bullet) !}$ sont isomorphes,
comme $u ^{*}= u ^{\sharp (\bullet) !} [e]$, 
on en déduit alors le résultat. 
\end{proof}

\begin{rema}
Avec les notations de \ref{lemm-u^*coh}, supposons $\X$ affine. 
Soient $E ^{(m _0)}$ un $\widetilde{D} ^{(m _0)} _{\X ^{\sharp}}$-module cohérent sans $p$-torsion,
$\E ^{(\bullet)}:= \widetilde{\D} ^{(\bullet+m _0)} _{\X ^{\sharp}} \otimes _{\widetilde{D} ^{(m _0)} _{\X ^{\sharp}}} E ^{(m _0)}$ 
le $\widetilde{\D} ^{(\bullet+m _0)} _{\X ^{\sharp}}$-module localement de présentation finie associé, 
$\E := \underrightarrow{\lim} (\E ^{\bullet}) $ le $\smash{\D} ^\dag _{\X ^{\sharp}} (\hdag T) _{\Q}$-module cohérent associé. 
En prenant une résolution de $E ^{(m _0)}$ par des 
$\widetilde{D} ^{(m _0)} _{\X ^{\sharp}}$-module libre de type fini qui induit, après application du foncteur 
$\widetilde{\D} ^{(\bullet+m _0)} _{\X ^{\sharp}} \otimes _{\widetilde{D} ^{(m _0)} _{\X ^{\sharp}}}-$,
une résolution de $\E ^{(\bullet)}$ par des 
$\widetilde{\D} ^{(\bullet+m _0)} _{\X ^{\sharp}}$-modules libres de type fini, 
on vérifie que
le complexe $\R \underline{\Gamma} ^{\dag} _Z (\E)$ 
est isomorphe à un complexe dont les termes sont de la forme $\mathcal{H} ^{\dag e} _{Z}
((\smash{\D} ^\dag _{\X ^{\sharp}} (\hdag T) _{\Q}) ^{N})$ pour certains entiers $N$.
\end{rema}

\begin{theo}
\label{theo-u^*coh}
On suppose que $u\colon \ZZ  \hookrightarrow \X$ est 
de codimension pure $e$.
Soit $\E $ un $\smash{\D} ^\dag _{\X ^{\sharp}} (\hdag T) _{\Q}$-module cohérent vérifiant les deux propriétés suivantes:
\begin{enumerate}
\item \label{theo-u^*coh-(1)} 
pour $i = 0, \dots, e -1$, localement en $\ZZ$, les $\smash{\D} ^\dag _{\ZZ ^\sharp} (\hdag U) _{\Q}$-modules $\mathcal{H} ^{i} u ^{!\sharp} (\E)$ sont $\Gamma (\ZZ,-)$-acycliques ; 

\item \label{theo-u^*coh-(2)}
le module $u ^{*} (\E)$ est un isocristal sur $\ZZ ^\sharp$ surconvergent le long de $U$.
\end{enumerate}
Le $\smash{\D} ^\dag _{\X ^{\sharp}} (\hdag T) _{\Q}$-module $\mathcal{H} ^{\dag e} _{Z} (\E)$ est alors cohérent.
\end{theo}

\begin{proof}
0) {\it Hypothèses \ref{theo-u^*coh-(1)}: passage aux sections globales, quelques notations}. 

Comme la $\smash{\D} ^\dag _{\X ^{\sharp}} (\hdag T) _{\Q}$-cohérence de $\mathcal{H} ^{\dag e} _{Z} (\E)$ est locale, 
on peut reprendre les notations et hypothèses du chapitre \ref{nota-sect2}.
Dans ce cas, posons $E:= \Gamma (\X, \E)$.
Avec les notations de \ref{nota-u*-u*}, 
les hypothèses d'acyclicité sur $\E$ implique 
$\Gamma ( \ZZ, u ^*\E)
\riso
\mathcal{H} ^{0}(\R \Gamma ( \ZZ, \L u ^* \E))
$ (on utilise la suite spectrale d'hypercohomologie du foncteur dérivé de
$\Gamma ( \ZZ, -)$).
On déduit alors du lemme \ref{RGammau*=u*} 
l'isomorphisme
$\Gamma ( \ZZ, u ^{*} \E) \riso 
\mathcal{H} ^{0} (\L u ^* (E))=u ^* (E)$.

Choisissons une présentation finie
$(\smash{D} ^\dag _{\X  ^{\sharp}} (\hdag T) _{\Q}) ^{r} \underset{a}{\longrightarrow}  (\smash{D} ^\dag _{\X  ^{\sharp}} (\hdag T) _{\Q}) ^{s} \underset{b}{\longrightarrow} E \to 0$
de $E$ 
(voir le théorème de type $A$ de Berthelot de \cite[3.6.5]{Be1}
qui est valable pour les $\smash{\D} ^\dag _{\X ^{\sharp}} (\hdag T) _{\Q}$-modules cohérents).
En appliquant à cette présentation finie le foncteur $\smash{\D} ^\dag _{\X ^{\sharp}} (\hdag T) _{\Q} \otimes _{\smash{D} ^\dag _{\X  ^{\sharp}} (\hdag T) _{\Q}}-$, 
on en déduit la suite exacte de $\smash{\D} ^\dag _{\X ^{\sharp}} (\hdag T) _{\Q}$-modules à gauche 
$(\smash{\D} ^\dag _{\X ^{\sharp}} (\hdag T) _{\Q}) ^{r} \underset{\alpha}{\longrightarrow}  (\smash{\D} ^\dag _{\X ^{\sharp}} (\hdag T) _{\Q}) ^{s} \underset{\beta}{\longrightarrow} \E \to 0$.
Par exactitude à droite des deux foncteurs de la forme $u ^*$ (voir les définitions de \ref{nota-u*-u*}), on obtient les suites exactes
$ (u ^* \smash{\D} ^\dag _{\X ^{\sharp}} (\hdag T) _{\Q}) ^{r} \underset{u ^{*} (\alpha)}{\longrightarrow} (u ^* \smash{\D} ^\dag _{\X ^{\sharp}} (\hdag T) _{\Q}) ^{s}
 \underset{u ^{*} (\beta)}{\longrightarrow} u ^{*} (\E)\to 0$
 et 
 $(u ^* \smash{D} ^\dag _{\X  ^{\sharp}} (\hdag T) _{\Q}) ^{r} \underset{u ^{*} (a)}{\longrightarrow} (u ^* \smash{D} ^\dag _{\X  ^{\sharp}} (\hdag T) _{\Q}) ^{s}
 \underset{u ^{*} (b)}{\longrightarrow} u ^{*} (E)\to 0$. 
Comme $\Gamma (\ZZ,u ^{*} (\alpha))=u ^{*} (a)$,
comme le morphisme canonique
$u ^* (E) \to \Gamma ( \ZZ, u ^{*} \E)$ 
est un isomorphisme, en appliquant le foncteur $\Gamma (\ZZ, -)$ à l'avant dernière suite exacte, 
on obtient encore une suite exacte (canoniquement isomorphe à la dernière).
En posant $\M:= u ^{*}( \E)$,
$M:= \Gamma ( \ZZ, \M)$, 
$\phi:= \Gamma ( \ZZ, u ^{*} (\alpha))$
et
$\psi:= \Gamma ( \ZZ, u ^{*} (\beta))$, on obtient donc la suite exacte
 $(u ^* \smash{D} ^\dag _{\X  ^{\sharp}} (\hdag T) _{\Q}) ^{r} \underset{\phi}{\longrightarrow} (u ^* \smash{D} ^\dag _{\X  ^{\sharp}} (\hdag T) _{\Q}) ^{s}
 \underset{\psi}{\longrightarrow} M\to 0$.

\medskip
0 bis) {\it Hypothèses \ref{theo-u^*coh-(2)}: niveaux $m \geq m _0$, quelques notations.}

Soit $m _0 \geq 0$ assez grand tel que, pour tout $m \geq m _0$, il existe  
un morphisme $\widetilde{\D} ^{(m)} _{\X^{\sharp},\Q}$-linéaire à gauche 
de la forme
$\alpha ^{(m)}\colon (\widetilde{\D} ^{(m)} _{\X^{\sharp},\Q}) ^{r} 
\to 
(\widetilde{\D} ^{(m)} _{\X^{\sharp},\Q}) ^{s}$
et induisant $\alpha$ par extension via 
$\widetilde{\D} ^{(m)} _{\X^{\sharp},\Q} 
\to 
\smash{\D} ^\dag _{\X ^{\sharp}} (\hdag T) _{\Q} $.
D'après la seconde partie de nos hypothèses, $\M$ est associé à un isocristal sur $Z ^{\sharp} \setminus U$ surconvergent le long de $U$ 
(voir \cite{caro-Tsuzuki}).
Donc, comme $M=\Gamma (\ZZ, \M)$, 
comme $\ZZ$ est affine et lisse, 
le $O _{\ZZ} (\hdag U) _{\Q}$-module $M$ est alors projectif et de type fini.
De plus, 
de manière analogue à \cite[4.4.5]{Be1} (et aussi \cite[4.4.7]{Be1}), 
quitte à augmenter $\lambda _0$ et $m _0$, 
il existe un $\widetilde{\D} ^{(m _0)} _{\ZZ ^{\sharp},\Q}$-module de type fini et topologiquement nilpotent
$\M ^{(m_0)}$ qui est, pour la structure induite de $\widetilde{\B} ^{(m _0)} _{\ZZ} (U) _{\Q}$-module,
un $\widetilde{\B} ^{(m _0)} _{\ZZ} (U) _{\Q}$-module (localement projectif) de type fini 
et un isomorphisme $\D ^{\dag}  _{\ZZ^{\sharp}} (\hdag U) _{\Q}$-linéaire de la forme
$\M \riso  \D ^{\dag}  _{\ZZ^{\sharp}} (\hdag U) _{\Q}\otimes _{\widetilde{\D} ^{(m _0)}_{\ZZ ^{\sharp},\Q} } \M ^{(m _0)}$.
Pour tout $m\geq m _0$, on pose alors
$\M ^{(m)} := \widetilde{\D} ^{(m )}_{\ZZ ^{\sharp},\Q}   \otimes _{\widetilde{\D} ^{(m _0)}_{\ZZ ^{\sharp},\Q} } \M ^{(m _0)}$.

Notons $\beta ^{(m)}\colon 
( \widetilde{\D} ^{(m)} _{\X^{\sharp},\Q}) ^{s} 
\to 
\E$
le composé de l'inclusion canonique
$(\widetilde{\D} ^{(m)} _{\X^{\sharp},\Q}) ^{s} 
\hookrightarrow 
( \smash{\D} ^\dag _{\X ^{\sharp}} (\hdag T) _{\Q}) ^{s}$
avec le morphisme $\beta$. 
Posons enfin
$ \phi ^{(m)}:= \Gamma (\ZZ, u ^{*}(\alpha ^{(m)}))$
et
$\psi ^{(m)}:= \Gamma (\ZZ, u ^{*}(\beta ^{(m)}))$.
Avec les définitions et notations de \ref{nota-u*-u*}, on obtient ainsi la suite 
 $(u ^* \widetilde{D} ^{(m)}  _{\X ^{\sharp},\Q}) ^{r} \underset{\phi^{(m)}}{\longrightarrow} (u ^* \widetilde{D} ^{(m)}  _{\X ^{\sharp},\Q}) ^{s}
 \underset{\psi^{(m)}}{\longrightarrow} M \to 0$. On remarque que cette suite devient exacte seulement après passage à la limite sur le niveau.

\medskip
1) {\it Hypothèses \ref{theo-u^*coh-(2)}: topologie de type LB sur $M $, niveaux $m \geq m _1$, sections continues $\theta$, $\theta ^{(m)}$.}

On munit 
$M$ de la topologie canonique de $O _{\ZZ} (\hdag U) _{\Q}$-module de type fini qui en 
fait un $O _{\ZZ} (\hdag U) _{\Q}$-module de type LB (voir \ref{lemm-top-can-isoc} et \ref{O-coh-LB}).
Comme $M$ est un $O _{\ZZ} (\hdag U) _{\Q}$-module projectif,
il existe alors un morphisme $O _{\ZZ} (\hdag U) _{\Q}$-linéaire  
$\theta \colon M \to (u ^* \smash{D} ^\dag _{\X  ^{\sharp}} (\hdag T) _{\Q}) ^{s}$ tel que
$\psi \circ \theta = id$.
Comme $M ^{(m _0)}$ est 
un $\widetilde{B} ^{(m _0)} _{\ZZ} (U) _{\Q}$-module de type fini,
il existe $m _1 \geq m _0$ tel que, pour tout $m \geq m _1$, 
on ait 
$\theta (M ^{(m _0)}) \subset (u ^* \widetilde{D} ^{(m)}  _{\X ^{\sharp},\Q}) ^{s}$.
Ainsi, cette section $\theta$ se factorise (de manière unique) en un morphisme $\widetilde{B} ^{(m _0)} _{\ZZ} (U) _{\Q}$-linéaire
de la forme
$M^{(m _0)} \to (u ^* \widetilde{D} ^{(m)}  _{\X ^{\sharp},\Q}) ^{s}$. 
D'où la factorisation $\widetilde{\B} ^{(m )} _{\ZZ} (U) _{\Q}$-linéaire unique 
$\theta ^{(m)}\colon M ^{(m )} \to (u ^* \widetilde{D} ^{(m)}  _{\X ^{\sharp},\Q}) ^{s}$ pour tout $m \geq m _1$.
Or,
une base de voisinages de zéro sur $(u ^* \widetilde{D} ^{(m)}  _{\X ^{\sharp},\Q}) ^{s}$ est donnée
par la famille de $\widetilde{B} ^{(m )} _{\ZZ} (U)$-modules $( p ^{n}  (u ^* \widetilde{D} ^{(m)}  _{\X ^{\sharp}}) ^{s}) _{n\in\N}$.
On en déduit que le morphisme
$\theta ^{(m)}$ de $\widetilde{B} ^{(m )} _{\ZZ} (U) _{\Q}$-modules de Banach est alors
continue. Par composition de morphismes continus,
il en est donc de même $M ^{(m)} \to (u ^* \smash{D} ^\dag _{\X  ^{\sharp}} (\hdag T) _{\Q}) ^{s}$.
Par passage à la limite sur le niveau $ m \geq m _1$, 
la section $\theta \colon M\to (u ^* \smash{D} ^\dag _{\X  ^{\sharp}} (\hdag T) _{\Q}) ^{s}$
est alors continue.

\medskip
2) {\it Le $K$-espace $G$ de type $LB$, les $K$-espaces de Banach $G ^{(m)}$.}
Avec les topologies canoniques respectives (voir les définitions de \ref{defi-topo-can}), 
d'après la proposition \ref{morphDdag->Ddagcontin}, 
l'application
$\phi$ est un morphisme continu de $K$-espaces de type $LB$ (voir \ref{DotimesD-sep-seqcomp}).
Notons 
$G := (u ^* \smash{D} ^\dag _{\X  ^{\sharp}} (\hdag T) _{\Q}) ^{r}/ \ker (\phi)$ le $K$-espace localement convexe dont
la topologie est celle qui 
fait de la projection canonique
$\pi \colon (u ^* \smash{D} ^\dag _{\X  ^{\sharp}} (\hdag T) _{\Q}) ^{r} 
\twoheadrightarrow 
G$
un morphisme strict.
Notons 
$\iota
\colon G
\hookrightarrow 
(u ^* \smash{D} ^\dag _{\X  ^{\sharp}} (\hdag T) _{\Q}) ^{s}$ le monomorphisme tel 
que $\iota \circ \pi = \phi$.
Comme $\phi$ est continue, alors $\iota$ est continu.
Comme $(u ^* \smash{D} ^\dag _{\X  ^{\sharp}} (\hdag T) _{\Q}) ^{s}$ est séparé (voir \ref{DotimesD-sep-seqcomp}), il en est alors de même de $G$.
Ainsi $G$ est un quotient séparé d'un espace de type $LB$.
D'après \ref{qsepLBestLB}, cela entraîne que $G$ est aussi un espace de type $LB$.
Plus précisément, d'après sa preuve, en notant
 $G ^{(m)}:=  (u ^* \widetilde{D} ^{(m)}  _{\X ^{\sharp},\Q}) ^{r} / ( \ker (\phi) \cap (u ^* \widetilde{D} ^{(m)}  _{\X ^{\sharp},\Q}) ^{r}  )$
muni de la topologie quotient, 
 $G ^{(m)}$ est un $K$-espace de Banach et 
l'isomorphisme canonique
$\underrightarrow{\lim} _m G ^{(m)}
\riso G$
 est aussi un homéomorphisme.
Comme par définition le morphisme 
$ (u ^* \widetilde{D} ^{(m)}  _{\X ^{\sharp},\Q}) ^{r} 
\to 
G ^{(m)}$
est strict, 
il découle de \ref{DotimesDhat-se}, 
que l'on dispose de l'épimorphisme strict
\begin{equation}
 \label{theo-u^*coh-arrow3} 
 \widetilde{D} ^{(m)}  _{\X ^{\sharp} \leftarrow \ZZ ^{\sharp},\Q} \widehat{\otimes} _{\widetilde{D} ^{(m)}_{\ZZ ^{\sharp},\Q} }
(u ^* \widetilde{D} ^{(m)}  _{\X ^{\sharp},\Q}) ^{r} 
\twoheadrightarrow  
\widetilde{D} ^{(m)}  _{\X ^{\sharp} \leftarrow \ZZ ^{\sharp},\Q} \widehat{\otimes} _{\widetilde{D} ^{(m)}_{\ZZ ^{\sharp},\Q} } G ^{(m)}.
\end{equation}
Il résulte alors de \ref{lim-surj-strict} que l'on dispose de l'épimorphisme strict:
\begin{equation}
 \label{theo-u^*coh-arrow3lim} 
 \underrightarrow{\lim} _m  
 \widetilde{D} ^{(m)}  _{\X ^{\sharp} \leftarrow \ZZ ^{\sharp},\Q} \widehat{\otimes} _{\widetilde{D} ^{(m)}_{\ZZ ^{\sharp},\Q} }
(u ^* \widetilde{D} ^{(m)}  _{\X ^{\sharp},\Q}) ^{r} 
\twoheadrightarrow
\underrightarrow{\lim} _m  
\widetilde{D} ^{(m)}  _{\X ^{\sharp} \leftarrow \ZZ ^{\sharp},\Q} \widehat{\otimes} _{\widetilde{D} ^{(m)}_{\ZZ ^{\sharp},\Q} } G ^{(m)}.
\end{equation}

\medskip
3) {\it Les applications $\phi$, $\iota$ et $\psi$ sont des morphismes stricts.}

D'après ce qui précède, on dispose 
du morphisme $O _{\ZZ, \Q}$-linéaire 
$(\iota, \theta) \colon G \oplus M
\to 
(u ^* \smash{D} ^\dag _{\X  ^{\sharp}} (\hdag T) _{\Q}) ^{s}$ 
qui est aussi une application continue et bijective entre deux $K$-espaces de type $LB$.
D'après le théorème de l'application ouverte de Banach (voir \cite[8.8]{Schneider-NonarchFuncAn}), 
le morphisme $(\iota, \theta)$ est donc un homéomorphisme. 
Comme $\phi $ est le morphisme composé
\begin{equation}
 \label{theo-u^*coh-arrow1}
\phi \colon 
(u ^* \smash{D} ^\dag _{\X  ^{\sharp}} (\hdag T) _{\Q}) ^{r}
\underset{\pi}{\twoheadrightarrow} 
G 
\subset 
G \oplus M
\underset{(\iota, \theta)}{\riso} 
(u ^* \smash{D} ^\dag _{\X  ^{\sharp}} (\hdag T) _{\Q}) ^{s},
\end{equation}
le morphisme $\phi$ est alors strict 
(ce qui d'ailleurs est équivalent à la propriété que le morphisme $\iota$ est strict). 
On remarque que, comme $M$ est un $K$-espace séparé, 
alors $G$ est un fermé de $G \oplus M$ (car homéomorphe au fermé
$G \oplus \{ 0 \}$).
Enfin, comme $\iota (G) =\mathrm{Im} (\phi) =\ker (\psi)$, on obtient le carré commutatif canonique 
\begin{equation}
\label{carre-comm-u*psi}
\xymatrix @ R=0,3cm{
{G \oplus M} 
\ar[d] ^-{(0, id)}
\ar[r] ^-{(\iota, \theta)} _-{\sim}
& 
{(u ^* \smash{D} ^\dag _{\X  ^{\sharp}} (\hdag T) _{\Q}) ^{s}} 
\ar[d] ^-{\psi}
\\ 
{M} 
\ar@{=}[r] ^-{}
& 
{M,} 
}
\end{equation}
dont on sait déjà que toutes les flèches autres que $\psi$ sont continues et dont la flèche du haut est un homéomorphisme.
On en déduit la continuité de $\psi$. Comme $\psi$ est un morphisme continu et surjectif entre deux 
espaces de type $LB$, on en déduit que $\psi$ est un morphisme strict.

\medskip
4) {\it Constructions et propriétés de $H ^{(m)}$ et $N ^{(m)}$.}

Dans toute la suite de la preuve, $m$ sera toujours par défaut un entier plus grand que $m_1 \geq m _0$.
Notons 
$H ^{(m)}:= \iota  ^{-1} ((u ^* \widetilde{D} ^{(m)}  _{\X ^{\sharp},\Q}) ^{s})$
et
$\iota ^{(m)} \colon 
H ^{(m)}
\hookrightarrow
 (u ^* \widetilde{D} ^{(m)}  _{\X ^{\sharp},\Q}) ^{s}$ le morphisme $\widetilde{D} ^{(m)}  _{\ZZ ^{\sharp},\Q}$-linéaire induit par $\iota$.
 On munit $H ^{(m)}$ de l'unique topologie telle que 
$\iota ^{(m)}$ soit un morphisme strict (avec $(u ^* \widetilde{D} ^{(m)}  _{\X ^{\sharp},\Q}) ^{s}$ muni de
sa topologie canonique de $K$-espace de Banach).
On a vu au cours de l'étape $3$ que $\iota (G)$ est fermé. 
Il en résulte que $\iota ^{(m)} ( H ^{(m)}) $ est un fermé
de $(u ^* \widetilde{D} ^{(m)}  _{\X ^{\sharp},\Q}) ^{s} $
qui est un $\widetilde{B} ^{(m)} _{\ZZ}(U) _{\Q}$-module de Banach.
Ainsi, $H ^{(m)}$ est aussi un $\widetilde{B} ^{(m)} _{\ZZ}(U) _{\Q}$-module de Banach.
Notons alors $N ^{(m)}:= \mathrm{Im} (\psi ^{(m)})$
et
$\psi ^{\prime (m)}\colon 
(u ^* \widetilde{D} ^{(m)}  _{\X ^{\sharp},\Q}) ^{s} 
\twoheadrightarrow 
N ^{(m)}$
l'épimorphisme canonique factorisant $\psi ^{(m)}$.
On obtient alors sur $N ^{(m)}$ une structure de 
$\widetilde{B} ^{(m)} _{\ZZ}(U) _{\Q}$-module de Banach 
qui fait de l'épimorphisme $\widetilde{B} ^{(m)} _{\ZZ}(U) _{\Q}$-linéaire
$\psi ^{\prime (m)}\colon 
(u ^* \widetilde{D} ^{(m)}  _{\X ^{\sharp},\Q}) ^{s} 
\twoheadrightarrow 
N ^{(m)}$
une application stricte.
On dispose ainsi du diagramme commutatif
\begin{equation}
\label{theo-u^*coh-s.e.1}
\xymatrix @C=2cm @R=0,3cm {
{0} \ar[r] ^-{} & {G } \ar[r] ^-{\iota } & { (u ^* \smash{D} ^\dag _{\X  ^{\sharp}} (\hdag T) _{\Q}) ^{s}} \ar[r] ^-{\psi} & { M} \ar[r] ^-{}& { 0} 
\\ 
{0} 
\ar[r] ^-{}
& 
{ H ^{(m)}  }
\ar[r] ^-{\iota ^{(m)} } 
\ar@{^{(}->}[u] ^-{}
& 
{(u ^* \widetilde{D} ^{(m)}  _{\X ^{\sharp},\Q}) ^{s}  } 
\ar[r] ^-{\psi ^{\prime (m)}} 
\ar[ru] ^-{\psi ^{(m)}} 
\ar@{^{(}->}[u] ^-{}
& 
{N ^{(m)}}\ar[r] ^-{}
\ar@{^{(}->}[u] ^-{}
& { 0} 
}
\end{equation}
dont les morphismes horizontaux sont stricts et forment deux suites exactes courtes 
(la première est d'ailleurs scindée via $\theta$). 
Comme le composé de $H ^{(m)} \subset G$ avec $\iota$ est continu,
comme $\iota $ est un monomorphisme strict, 
on remarque que l'inclusion $H ^{(m)} \subset G$ est alors continue.
Enfin, comme $\psi ^{(m)}$ est continu, l'inclusion
$N ^{(m)} \subset M$ l'est aussi.

Grâce à \ref{DotimesDhat-se}, 
la suite exacte courte du bas de \ref{theo-u^*coh-s.e.1} induit la suite exacte 
avec morphismes stricts:
\begin{equation}
\label{theo-u^*coh-s.e.1-u+}
0
\to
\widetilde{D} ^{(m)}  _{\X ^{\sharp} \leftarrow \ZZ ^{\sharp},\Q} \widehat{\otimes} _{\widetilde{D} ^{(m)}_{\ZZ ^{\sharp},\Q} }
 H ^{(m)}
\underset{id \widehat{\otimes} \iota ^{(m)} }{\longrightarrow}  
\widetilde{D} ^{(m)}  _{\X ^{\sharp} \leftarrow \ZZ ^{\sharp},\Q} \widehat{\otimes} _{\widetilde{D} ^{(m)}_{\ZZ ^{\sharp},\Q} }
(u ^* \widetilde{D} ^{(m)}  _{\X ^{\sharp},\Q}) ^{s}
\underset{id \widehat{\otimes} \psi ^{\prime (m)}}{\longrightarrow}  
\widetilde{D} ^{(m)}  _{\X ^{\sharp} \leftarrow \ZZ ^{\sharp},\Q} \widehat{\otimes} _{\widetilde{D} ^{(m)}_{\ZZ ^{\sharp},\Q} }
N ^{(m)}
\to 
 0.
\end{equation}
En passant à la limite sur le niveau, il en résulte la suite exacte 
avec morphismes continues:
\begin{equation}
\label{theo-u^*coh-s.e.1-u+-lim}
0
\to
 \underset{m}{\underrightarrow{\lim}}\, \widetilde{D} ^{(m)}  _{\X ^{\sharp} \leftarrow \ZZ ^{\sharp},\Q} \widehat{\otimes} _{\widetilde{D} ^{(m)}_{\ZZ ^{\sharp},\Q} }
 H ^{(m)}
\to  
 \underset{m}{\underrightarrow{\lim}}\, \widetilde{D} ^{(m)}  _{\X ^{\sharp} \leftarrow \ZZ ^{\sharp},\Q} \widehat{\otimes} _{\widetilde{D} ^{(m)}_{\ZZ ^{\sharp},\Q} }
(u ^* \widetilde{D} ^{(m)}  _{\X ^{\sharp},\Q}) ^{s}
\to  
 \underset{m}{\underrightarrow{\lim}}\, \widetilde{D} ^{(m)}  _{\X ^{\sharp} \leftarrow \ZZ ^{\sharp},\Q} \widehat{\otimes} _{\widetilde{D} ^{(m)}_{\ZZ ^{\sharp},\Q} }
N ^{(m)}
\to 
 0.
\end{equation}

\medskip
5) {\it Passage à la limite sur le niveau pour $G ^{(m)}$ et $H ^{(m)}$: comparaison.}

Notons
$j _m \colon G ^{(m)} \to G$ le monomorphisme canonique continue. 
Pour tout $m \geq m _1$, comme
le morphisme $\phi$ se factorise par le morphisme continu
$\phi  ^{(m)} \colon (u ^* \widetilde{D} ^{(m)}  _{\X ^{\sharp},\Q}) ^{r}  \to (u ^* \widetilde{D} ^{(m)}  _{\X ^{\sharp},\Q}) ^{s} $,
on obtient alors l'inclusion
$\iota \circ j _{m}(G ^{(m)} ) \subset  (u ^{*}\widetilde{D} ^{(m)}  _{\X ^{\sharp},\Q}) ^{s} $
et donc
$j _{m}(G ^{(m)} ) \subset  H ^{(m)}$.
Le monomorphisme continu $j _m$ se factorise donc (de manière unique) par
un monomorphisme continu de la forme
$G ^{(m)} \to H ^{(m)}$.

Réciproquement, 
comme $H ^{(m)}$ est un $K$-espace de Banach, 
comme 
$G= \cup _{m \in \N} j _{m} (G ^{(m)})$,
d'après \cite[8.9]{Schneider-NonarchFuncAn}, 
il existe $N _{m}\geq m$ assez grand tel que 
$H ^{(m)} \subset G$ est le composé d'un morphisme (unique) continu de la forme
$H ^{(m)} \to G ^{(N _{m})}$ suivi de $j _{N _m}$.
Il ne coûte rien de supposer que la suite $(N _{m} ) _{m \geq m _1}$ est strictement croissante.

D'après \ref{DwidehatD-inj}, 
on dispose alors des monomorphismes continus:
\begin{equation}
\label{theo-u^*coh-inclusions}
\widetilde{D} ^{(m)}  _{\X ^{\sharp} \leftarrow \ZZ ^{\sharp},\Q} \widehat{\otimes} _{\widetilde{D} ^{(m)}_{\ZZ ^{\sharp},\Q} } H ^{(m)}
\hookrightarrow 
\widetilde{D} ^{(m)}  _{\X ^{\sharp} \leftarrow \ZZ ^{\sharp},\Q} \widehat{\otimes} _{\widetilde{D} ^{(m)}_{\ZZ ^{\sharp},\Q} } G ^{(N _m)}
\hookrightarrow
\widetilde{D} ^{(m)}  _{\X ^{\sharp} \leftarrow \ZZ ^{\sharp},\Q} \widehat{\otimes} _{\widetilde{D} ^{(m)}  _{\ZZ ^{\sharp},\Q}} H ^{(N _m)}.
\end{equation}
En passant la suite \ref{theo-u^*coh-inclusions} à la limite inductive sur le niveau, grâce \ref{DotimesD-sep-seqcomp-iso1}, 
son composé est un homéomorphisme. 
Comme le foncteur de passage à la limite inductive sur le niveau préserve la continuité et l'injectivité, 
on en déduit que le morphisme canonique horizontal du haut du carré
\begin{equation}
\label{theo-u^*coh-inclusions-lim}
\xymatrix @R=0,3cm {
{\underrightarrow{\lim} _m  
\widetilde{D} ^{(m)}  _{\X ^{\sharp} \leftarrow \ZZ ^{\sharp},\Q} \widehat{\otimes} _{\widetilde{D} ^{(m)}_{\ZZ ^{\sharp},\Q} } G ^{(N _m)}} 
\ar[r] ^-{}
& 
{\underrightarrow{\lim} _m  
\widetilde{D} ^{(m)}  _{\X ^{\sharp} \leftarrow \ZZ ^{\sharp},\Q} \widehat{\otimes} _{\widetilde{D} ^{(m)}_{\ZZ ^{\sharp},\Q} } H ^{(N _m)}
} 
\\ 
{\underrightarrow{\lim} _m  
\widetilde{D} ^{(m)}  _{\X ^{\sharp} \leftarrow \ZZ ^{\sharp},\Q} \widehat{\otimes} _{\widetilde{D} ^{(m)}_{\ZZ ^{\sharp},\Q} } G ^{(m)}} 
\ar[u] ^-{\sim}
\ar[r] ^-{}
& 
{\underrightarrow{\lim} _m  
\widetilde{D} ^{(m)}  _{\X ^{\sharp} \leftarrow \ZZ ^{\sharp},\Q} \widehat{\otimes} _{\widetilde{D} ^{(m)}_{\ZZ ^{\sharp},\Q} } H ^{(m)},} 
\ar[u] ^-{\sim}
}
\end{equation}
est un homéomorphisme.
Comme les morphismes verticaux sont des homéomorphismes (voir \ref{DotimesD-sep-seqcomp-iso1}), il en est de même de la flèche du bas.

\medskip
6) {\it Passage à la limite sur le niveau pour $M ^{(m)}$ et $N ^{(m)}$: comparaison.}

Comme $\psi ^{(m)}$ est continu, 
pour tout $m \geq m _1$,
d'après \cite[8.9]{Schneider-NonarchFuncAn}, 
il existe $N _{m}\geq m$ assez grand tel que 
 $\psi ^{(m)}\colon 
(u ^* \widetilde{D} ^{(m)}  _{\X ^{\sharp},\Q}) ^{s} 
\to 
M$ est la composition d'un morphisme continu (forcément unique) de la forme
$(u ^* \widetilde{D} ^{(m)}  _{\X ^{\sharp},\Q}) ^{s}  
\to M ^{(N _{m})}$ suivi du monomorphisme 
$M ^{(N _{m})}\hookrightarrow M$.
Comme $N ^{(m)} = \mathrm{Im} (\psi ^{(m)})= \mathrm{Im} (\psi ^{\prime (m)})$, 
ce morphisme $(u ^* \widetilde{D} ^{(m)}  _{\X ^{\sharp},\Q}) ^{s}  
\to M ^{(N _{m})}$ se décompose 
de manière unique en
$(u ^* \widetilde{D} ^{(m)}  _{\X ^{\sharp},\Q}) ^{s}  
\underset{\psi ^{\prime (m)}}{\twoheadrightarrow}
N ^{(m)}
\hookrightarrow M ^{(N _{m})}$.
Le morphisme $\psi ^{(m)}$ induit donc le morphisme composé ci-dessous noté abusivement
$$ u _+ ^{\sharp (m)} (\psi ^{(m)})
\colon 
\widetilde{D} ^{(m)}  _{\X ^{\sharp} \leftarrow \ZZ ^{\sharp},\Q} \widehat{\otimes} _{\widetilde{D} ^{(m)}_{\ZZ ^{\sharp},\Q} }
(u ^* \widetilde{D} ^{(m)}  _{\X ^{\sharp},\Q}) ^{s}
\underset{id \widehat{\otimes} \psi ^{\prime (m)}}{\twoheadrightarrow}
\widetilde{D} ^{(m)}  _{\X ^{\sharp} \leftarrow \ZZ ^{\sharp},\Q} \widehat{\otimes} _{\widetilde{D} ^{(m)}_{\ZZ ^{\sharp},\Q} }
N ^{(m)}
\to  
 \underset{m}{\underrightarrow{\lim}}\, \widetilde{D} ^{(m)}  _{\X ^{\sharp} \leftarrow \ZZ ^{\sharp},\Q} \widehat{\otimes} _{\widetilde{D} ^{(m)}_{\ZZ ^{\sharp},\Q} }
M ^{(m)}.$$
La notation $u _+ ^{\sharp (m)} (\psi ^{(m)})$ est justifiée par le fait que ce morphisme composé 
 ne dépend pas du choix de $N _m$ mais seulement du niveau $m$ fixé, de $\psi ^{(m)}$ et
de l'immersion fermée $u$.
En passant à la limite sur le niveau, on obtient les morphismes continues:
\begin{equation}
\label{limmu+psim}
\underset{m}{\underrightarrow{\lim}}\,  u _+ ^{\sharp  (m)} (\psi ^{(m)})
\colon 
 \underset{m}{\underrightarrow{\lim}}\, \widetilde{D} ^{(m)}  _{\X ^{\sharp} \leftarrow \ZZ ^{\sharp},\Q} \widehat{\otimes} _{\widetilde{D} ^{(m)}_{\ZZ ^{\sharp},\Q} }
(u ^* \widetilde{D} ^{(m)}  _{\X ^{\sharp},\Q}) ^{s}
\underset{id \widehat{\otimes} \psi ^{\prime (m)}}{\twoheadrightarrow}
 \underset{m}{\underrightarrow{\lim}}\, \widetilde{D} ^{(m)}  _{\X ^{\sharp} \leftarrow \ZZ ^{\sharp},\Q} \widehat{\otimes} _{\widetilde{D} ^{(m)}_{\ZZ ^{\sharp},\Q} }
N ^{(m)}
\to  
 \underset{m}{\underrightarrow{\lim}}\, \widetilde{D} ^{(m)}  _{\X ^{\sharp} \leftarrow \ZZ ^{\sharp},\Q} \widehat{\otimes} _{\widetilde{D} ^{(m)}_{\ZZ ^{\sharp},\Q} }
M ^{(m)}.
\end{equation}
Comme, pour tout $m \geq m _1$, le morphisme $\psi ^{(m)}\circ \theta ^{(m)}$
est l'inclusion canonique
$M ^{(m)} \hookrightarrow M$, 
ce dernier se factorise alors canoniquement en 
$M ^{(m)} \hookrightarrow N ^{(m)} \subset M$. 
On en déduit comme pour l'étape $5$ que la flèche de droite de \ref{limmu+psim}
est un homéomorphisme.

\medskip
7) {\it Première conclusion.} 

En composant \ref{theo-u^*coh-arrow3lim} avec l'isomorphisme du bas de \ref{theo-u^*coh-inclusions-lim}, on obtient l'épimorphisme strict de la suite:
\begin{equation}
 \label{theo-u^*coh-epim-monom} 
 \underrightarrow{\lim} _m  
 \widetilde{D} ^{(m)}  _{\X ^{\sharp} \leftarrow \ZZ ^{\sharp},\Q} \widehat{\otimes} _{\widetilde{D} ^{(m)}_{\ZZ ^{\sharp},\Q} }
(u ^* \widetilde{D} ^{(m)}  _{\X ^{\sharp},\Q}) ^{r} 
\twoheadrightarrow
\underrightarrow{\lim} _m  
\widetilde{D} ^{(m)}  _{\X ^{\sharp} \leftarrow \ZZ ^{\sharp},\Q} \widehat{\otimes} _{\widetilde{D} ^{(m)}_{\ZZ ^{\sharp},\Q} } H ^{(m)}
\hookrightarrow  
 \underset{m}{\underrightarrow{\lim}}\, \widetilde{D} ^{(m)}  _{\X ^{\sharp} \leftarrow \ZZ ^{\sharp},\Q} \widehat{\otimes} _{\widetilde{D} ^{(m)}_{\ZZ ^{\sharp},\Q} }
(u ^* \widetilde{D} ^{(m)}  _{\X ^{\sharp},\Q}) ^{s} ,
\end{equation}
la seconde flèche étant le monomorphisme de  \ref{theo-u^*coh-s.e.1-u+-lim}.
On déduit alors de la suite exacte \ref{theo-u^*coh-s.e.1-u+-lim} et du fait que la flèche de droite de \ref{limmu+psim} est homéomorphisme, la suite exacte:
\begin{equation}
\label{theo-u^*coh-s.e.2}
 \underset{m}{\underrightarrow{\lim}}\, \widetilde{D} ^{(m)}  _{\X ^{\sharp} \leftarrow \ZZ ^{\sharp},\Q} \widehat{\otimes} _{\widetilde{D} ^{(m)}_{\ZZ ^{\sharp},\Q} }
(u ^* \widetilde{D} ^{(m)}  _{\X ^{\sharp},\Q}) ^{r}
\underset{\alpha ^{\dag} _Z}{\longrightarrow}  
 \underset{m}{\underrightarrow{\lim}}\, \widetilde{D} ^{(m)}  _{\X ^{\sharp} \leftarrow \ZZ ^{\sharp},\Q} \widehat{\otimes} _{\widetilde{D} ^{(m)}_{\ZZ ^{\sharp},\Q} }
(u ^* \widetilde{D} ^{(m)}  _{\X ^{\sharp},\Q}) ^{s}
\underset{\beta ^{\dag} _Z}{\longrightarrow}  
 \underset{m}{\underrightarrow{\lim}}\, \widetilde{D} ^{(m)}  _{\X ^{\sharp} \leftarrow \ZZ ^{\sharp},\Q} \widehat{\otimes} _{\widetilde{D} ^{(m)}_{\ZZ ^{\sharp},\Q} }
M ^{(m)}
\to 
 0,
\end{equation}
où $\alpha ^{\dag} _Z:=  
id \widehat{\otimes} _{\widetilde{D} ^{(m)}_{\ZZ ^{\sharp},\Q} } \Gamma (\ZZ, u ^{*} \alpha ^{(m)} )$, qui est égale
au composé des deux morphismes de \ref{theo-u^*coh-epim-monom}, 
et où, avec les notations de l'étape $6$,
$\beta ^{\dag} _Z:= 
\underset{m}{\underrightarrow{\lim}}\,  u _+ ^{\sharp (m)} (\Gamma (\ZZ, u ^{*} \beta ^{(m)} ))$.

\medskip
8) {\it Faisceautisation.} 

i) Soit $\mathfrak{B}$ la base de voisinages de $\X$ des ouverts affines. 
Par définition de $\M ^{(m)}$, on vérifie que
le préfaisceau sur $\mathfrak{B}$ défini par
$\U \in \mathfrak{B} \mapsto 
\underset{m}{\underrightarrow{\lim}}\, 
\widetilde{D} ^{(m)} _{\U^{\sharp} \leftarrow \ZZ^{\sharp} \cap \U^{\sharp},\Q} \widehat{\otimes} _{\widetilde{D} ^{(m)}_{\ZZ ^{\sharp}\cap \U^{\sharp},\Q}}
\Gamma (\U \cap \ZZ, \M ^{(m)})$
est en fait un faisceau dont le faisceau sur $\X$
est 
$u _+ ^\sharp ( \M)$.

ii) Avec les notations de \ref{faisceautisation-hat},
comme le foncteur faisceau associé à un préfaisceau (d'ensemble) commute 
aux limites inductives filtrantes, 
le faisceau associé au préfaisceau 
$\U \in \mathfrak{B} \mapsto 
\underset{m}{\underrightarrow{\lim}}\, \widetilde{D} ^{(m)} _{\U ^{\sharp} \leftarrow \ZZ^{\sharp} \cap \U^{\sharp},\Q} \widehat{\otimes} _{\widetilde{D} ^{(m)}_{\ZZ^{\sharp} \cap \U^{\sharp},\Q}}
\Gamma (\U \cap \ZZ, u ^{*} (( \widetilde{\D} ^{(m)} _{\X^{\sharp},\Q}) ^{r}) )$
est
$\underset{m}{\underrightarrow{\lim}}\, \widetilde{\D} ^{(m)} _{\X ^{\sharp} \leftarrow \ZZ ^{\sharp},\Q} \widehat{\otimes} _{\widetilde{\D} ^{(m)}_{\ZZ ^{\sharp},\Q} }
(u ^* \widetilde{\D} ^{(m)} _{\X^{\sharp},\Q}) ^{r}
$, de même pour un autre entier que $r$.

iii) Pour tout $\U \in \mathfrak{B}$, 
notons $u^{\sharp} |\U  \colon \ZZ ^{\sharp}\cap \U^{\sharp} \hookrightarrow \U^{\sharp}$ l'immersion fermée exacte induite par $u^{\sharp}$.
Pour tout $m \geq m _1$, on dispose de la suite (par forcément exacte):
$ (\widetilde{\D} ^{(m)}_{\U^{\sharp},\Q}) ^{r} \underset{\alpha ^{(m)} |\U}{\longrightarrow} (\widetilde{\D} ^{(m)}_{\U^{\sharp},\Q}) ^{s}
 \underset{\beta ^{(m)} |\U}{\longrightarrow} 
 \E |\U \to 0$.
Comme pour 
\ref{theo-u^*coh-s.e.2}, (il suffit de remplacer $u^{\sharp}$ par $u ^{\sharp}|\U$),
on obtient alors la suite exacte
\small
$$
\underset{m}{\underrightarrow{\lim}}\, \widetilde{D} ^{(m)} _{\U \leftarrow \ZZ \cap \U,\Q} \widehat{\otimes} 
\Gamma (\U \cap \ZZ, u ^{*} (( \widetilde{\D} ^{(m)} _{\X^{\sharp},\Q}) ^{r}) )
\underset{\alpha ^{\dag} _\U}{\longrightarrow}  
\underset{m}{\underrightarrow{\lim}}\, \widetilde{D} ^{(m)} _{\U \leftarrow \ZZ \cap \U,\Q} \widehat{\otimes} 
\Gamma (\U \cap \ZZ, u ^{*} (( \widetilde{\D} ^{(m)} _{\X^{\sharp},\Q}) ^{s}) )
\underset{\beta ^{\dag} _\U}{\longrightarrow}  
\underset{m}{\underrightarrow{\lim}}\, \widetilde{D} ^{(m)} _{\U \leftarrow \ZZ \cap \U,\Q} \widehat{\otimes} 
\Gamma (\U \cap \ZZ, u ^{*} ( \M ^{(m)})
\to 0.
$$
\normalsize
où
$\widehat{\otimes}$
désigne 
$\widehat{\otimes} _{\widetilde{D} ^{(m)}_{\ZZ ^{\sharp}\cap \U^{\sharp},\Q}}$
et où
$\alpha ^{\dag} _\U:=
id
\widehat{\otimes} 
\Gamma (\U \cap \ZZ, u ^{*} \alpha ^{(m)} )$
et 
$\beta ^{\dag} _\U:= 
\underset{m}{\underrightarrow{\lim}}\,  (u ^{\sharp} |\U )_{+} ^{(m)} (\Gamma (\U \cap \ZZ, u ^{*} \beta ^{(m)} ))$.
Comme ces suites exactes sont fonctoriels en $\U$,
comme le foncteur faisceau associé à un préfaisceau est exact, 
on obtient alors la suite exacte:
\begin{equation}
\label{theo-u^*coh-s.e.2faisc}
\underset{m}{\underrightarrow{\lim}}\, \widetilde{\D} ^{(m)} _{\X ^{\sharp} \leftarrow \ZZ ^{\sharp},\Q} \widehat{\otimes} _{\widetilde{\D} ^{(m)}_{\ZZ ^{\sharp},\Q} }
(u ^* \widetilde{\D} ^{(m)} _{\X^{\sharp},\Q}) ^{r}
\underset{\underset{m}{\underrightarrow{\lim}}\, 
id \widehat{\otimes} 
u ^{*} \alpha ^{(m)} }{\longrightarrow}  
 \underset{m}{\underrightarrow{\lim}}\, \widetilde{\D} ^{(m)} _{\X ^{\sharp} \leftarrow \ZZ ^{\sharp},\Q} \widehat{\otimes} _{\widetilde{\D} ^{(m)}_{\ZZ ^{\sharp},\Q} }
(u ^* \widetilde{\D} ^{(m)} _{\X^{\sharp},\Q}) ^{s}
\to  
u _+ ^\sharp ( \M)
\to 
 0.
\end{equation}

\bigskip
9) {\it Fin de la preuve.} 
Comme le foncteur $\mathcal{H} ^{\dag e} _{Z}$ est exact à droite (dans la catégorie des 
$ \smash{\D} ^\dag _{\X ^{\sharp}} (\hdag T) _{\Q}$-modules cohérents), 
alors 
le conoyau de $\mathcal{H} ^{\dag e} _{Z} (\phi)$ est isomorphe à $\mathcal{H} ^{\dag e} _{Z} (\E)$.
Or,  
il découle du lemme \ref{lemm-u^*coh} et de la suite exacte \ref{theo-u^*coh-s.e.2faisc},
que le conoyau 
de $\mathcal{H} ^{\dag e} _{Z} (\phi)$ est isomorphe à
$u _+ ^\sharp  \M$.
Comme $u ^{*} (\E) = \M$, on a ainsi démontré
$\mathcal{H} ^{\dag e} _{Z} (\E) \riso  u _+ ^{\sharp} u ^{*} (\E)$.
Enfin, comme $u$ est propre et
$u ^{*} (\E)$ est un $\smash{\D} ^\dag _{\ZZ ^\sharp} (\hdag U) _{\Q}$-module cohérent, 
alors 
$ u _+ ^{\sharp} u ^{*} (\E)$ est un $\smash{\D} ^\dag _{\X ^{\sharp}} (\hdag T) _{\Q}$-module cohérent.
\end{proof}

 \begin{coro}
\label{coro-theo-u^*coh}
On suppose que $u\colon \ZZ  \hookrightarrow \X$ est 
de codimension pure $1$.
Soit 
$\E $ un $\smash{\D} ^\dag _{\X ^{\sharp}} (\hdag T) _{\Q}$-module cohérent
tel que $\mathcal{H} ^{0} u ^{!\sharp} (\E)$ soit un $\smash{\D} ^\dag _{\ZZ ^\sharp} (\hdag U) _{\Q}$-module cohérent
et
que $\mathcal{H} ^{1} u ^{!\sharp} (\E)$ soit un isocristal sur $\ZZ ^\sharp$ surconvergent le long de $U$.
Le complexe
$\R \underline{\Gamma} ^{\dag} _Z (\E)$ est alors $\smash{\D} ^\dag _{\X ^{\sharp}} (\hdag T) _{\Q}$-cohérent.

\end{coro}

\begin{proof}
La $\smash{\D} ^\dag _{\X ^{\sharp}} (\hdag T) _{\Q}$-cohérence du complexe 
$\R \underline{\Gamma} ^{\dag} _Z (\E)$ résulte de \ref{stab-coh-loc-H0} et de \ref{theo-u^*coh}.
\end{proof}

\begin{rema}
\label{rema-fin}
Avec leurs notations, 
le corollaire \ref{coro-theo-u^*coh} ainsi que le théorème \ref{5assert-eq}
ci-dessous donnent une condition suffisante sur $\E$ pour valider l'implication
$ u ^{!\sharp} (\E) \in 
D ^{\mathrm{b}} _{\mathrm{coh}}( \smash{\D} ^\dag _{\ZZ ^\sharp} (\hdag U) _{\Q} )
\Rightarrow
u ^{\sharp (\bullet)!} (\E ^{(\bullet)}) \in 
\smash{\underrightarrow{LD}} ^{\mathrm{b}} _{\Q, \mathrm{coh}} ( \smash{\widetilde{\D}} _{\ZZ ^\sharp} ^{(\bullet)})$.

\end{rema}

\begin{theo}
\label{5assert-eq}
On suppose que $u\colon \ZZ  \hookrightarrow \X$ est 
de codimension pure $1$.
Soit
$\E ^{(\bullet)}$ est un objet 
de $\smash{\underrightarrow{LD}} ^{\mathrm{b}} _{\Q, \mathrm{coh}} ( \smash{\widetilde{\D}}  _{\X ^{\sharp}} ^{(\bullet)})$
et 
$\E := \underrightarrow{\lim}  ~ (\E ^{(\bullet)}) $
l'objet de $D ^{\mathrm{b}} _{\mathrm{coh}}( \smash{\D} ^\dag _{\X ^{\sharp}} (\hdag T) _{\Q} )$ correspondant. 
Les assertions suivantes sont équivalentes:
\begin{enumerate}
\item $u ^{\sharp (\bullet)!} (\E ^{(\bullet)}) \in 
\smash{\underrightarrow{LD}} ^{\mathrm{b}} _{\Q, \mathrm{coh}} ( \smash{\widetilde{\D}} _{\ZZ ^{\sharp}} ^{(\bullet)})$.
\item $\R \underline{\Gamma} ^{\dag} _Z (\E ^{(\bullet)}) \in 
\smash{\underrightarrow{LD}} ^{\mathrm{b}} _{\Q, \mathrm{coh}} ( \smash{\widetilde{\D}}  _{\X ^{\sharp}} ^{(\bullet)})$.
\item $(\hdag Z) ( \E ^{(\bullet)}) 
\in \smash{\underrightarrow{LD}} ^{\mathrm{b}} _{\Q, \mathrm{coh}} ( \smash{\widetilde{\D}}  _{\X ^{\sharp}} ^{(\bullet)} )$.
\item $(\hdag Z) ( \E) \in D ^{\mathrm{b}} _{\mathrm{coh}}( \smash{\D} ^\dag _{\X ^{\sharp}} ( \hdag T) _{\Q} )$.
\item $\R \underline{\Gamma} ^{\dag} _Z (\E) \in D ^{\mathrm{b}} _{\mathrm{coh}}( \smash{\D} ^\dag _{\X ^{\sharp}} ( \hdag T) _{\Q} )$.
\end{enumerate}
\end{theo}

\begin{proof}
Les équivalences $2 \Leftrightarrow 3$ et $4 \Leftrightarrow 5$ résultent du triangle distingué de localisation
$\R \underline{\Gamma} ^{\dag} _Z (\E ^{(\bullet)}) \to  \E ^{(\bullet)} \to (\hdag Z) ( \E ^{(\bullet)})  \to +1$.
L'équivalence $3 \Leftrightarrow 4$ est exactement le corollaire \cite[3.5.2]{caro-stab-sys-ind-surcoh}. 
L'équivalence $1 \Leftrightarrow 2$ découle de 
l'isomorphisme canonique 
$u ^{\sharp (\bullet)} _+ \circ u ^{\sharp (\bullet)!} (\E ^{(\bullet)}) 
\riso 
\R \underline{\Gamma} ^{\dag} _Z (\E ^{(\bullet)})$
(voir \cite[5.3.8.1]{caro-stab-sys-ind-surcoh}), 
de l'isomorphisme canonique
$ u ^{\sharp (\bullet)!}  \circ \R \underline{\Gamma} ^{\dag} _Z (\E ^{(\bullet)}) 
\riso
u ^{\sharp (\bullet)!} (\E ^{(\bullet)})$,
ainsi que du théorème de Berthelot-Kashiwara toujours valable dans
le contexte des catégories de la forme 
$\smash{\underrightarrow{LD}} ^{\mathrm{b}} _{\Q, \mathrm{coh}} ( \smash{\widetilde{\D}}  _{\X ^{\sharp}} ^{(\bullet)})$ 
(voir \cite[5.3.7]{caro-stab-sys-ind-surcoh}).
\end{proof}

\begin{rema}
\label{rema-3.4.10}
On garde les notations et hypothèses du théorème \ref{5assert-eq}. 

$\bullet$ 
Si l'une des conditions équivalentes du théorème \ref{5assert-eq} 
est satisfaite, alors on dispose de l'isomorphisme
$u ^{\sharp} _{+} \circ u ^{!\sharp}( \E) \riso  \R \underline{\Gamma} ^{\dag} _Z (\E)$
(dans le terme de gauche, les foncteurs $u ^{\sharp} _{+} $ et $u ^{!\sharp}$ 
sont calculés sur les catégories respectives de $\D ^{\dag}$-modules cohérents).

$\bullet$ Par contre, si on suppose seulement que
$ u ^{!\sharp}( \E) \in D ^{\mathrm{b}} _{\mathrm{coh}}( \smash{\D} ^\dag _{\ZZ ^\sharp} ( \hdag U) _{\Q} )$, 
alors il n'est ni clair que
$u ^{\sharp} _{+} \circ u ^{!\sharp}( \E) $ (les foncteurs $u ^{\sharp} _{+} $ et $u ^{!\sharp}$ 
sont calculés sur les catégories respectives de $\D ^{\dag}$-modules cohérents)
soit isomorphe à 
$\underrightarrow{\lim}   
  \circ u ^{\sharp (\bullet)}_{+} \circ u ^{\sharp (\bullet)!}(  \E ^{(\bullet)})
=
  \R \underline{\Gamma} ^{\dag} _Z (\E)$.
En effet, 
il n'est alors pas évident que la cohérence de la cohérence de $u ^{!\sharp}( \E) $ implique que 
$u ^{\sharp (\bullet)!} ( \E ^{(\bullet)}) 
\in \smash{\underrightarrow{LD}} ^{\mathrm{b}} _{\Q, \mathrm{coh}} ( \smash{\widetilde{\D}} _{\ZZ ^{\sharp}} ^{(\bullet)})$.

\end{rema}

\subsection{Applications aux log-isocristaux surconvergents}

Rappelons les conditions de non Liouvillité apparaissant dans \cite[1.1.1]{caro-Tsuzuki}:
\begin{defi}
\label{conditions*}
Soit $\E$ un isocristal sur $\X ^{\sharp}$ surconvergent le long de $T$.
On dit que $\E$ satisfait respectivement les conditions (a), (b), (p) si
le long de chacune des composantes irréductibles de 
 $D$ non incluse dans $T$ on ait
\begin{itemize}
\item [(a)] aucune différence des exposants de $\E$ ne soit un nombre $p$-adic de Liouville ;
\item [(b)] aucun des exposants de $\E$ ne soit un nombre $p$-adic de Liouville ;
\item [(p)] aucun des exposants de $\E$ ne soit un entier strictement positif.
\end{itemize}
\end{defi}

Nous aurons besoin de la définition suivante:
\begin{defi}
\label{devi-log}
Soit $\E \in D ^{\mathrm{b}} _{\mathrm{coh}}( \smash{\D} ^\dag _{\X ^{\sharp}} (\hdag T) _{\Q} )$.
\begin{itemize}
\item On dit que $\E$ est strictement $0$-dévissable en log-isocristaux surconvergents sur $\X ^{\sharp}$, 
s'il existe un diviseur $T' $ contenant $T$, une immersion fermée de $\V$-schémas formels lisses 
$i\colon \ZZ ' \hookrightarrow \X$ tels que $i ^{-1}(\mathfrak{D})$ soit un diviseur à croisements normaux strict de 
$\ZZ'$, $T' \cap Z'$ soit un diviseur de $Z'$, $\E \in D ^{\mathrm{b}} _{\mathrm{coh}}( \smash{\D} ^\dag _{\X ^{\sharp}} (\hdag T') _{\Q} )$,
$\E$ soit à support dans $\ZZ'$ et, si on note $\ZZ ^{\prime \sharp} := (\ZZ ', i ^{-1} (\mathcal{D}))$ et $i ^{\sharp} \colon \ZZ ^{\prime \sharp} \hookrightarrow \X ^{ \sharp}$ l'immersion fermée exacte induite, alors 
les espaces de cohomologie de 
$i ^{\sharp !} (\E)$ sont des 
isocristaux sur $\ZZ ^{\prime \sharp} $
surconvergents le long de $T' \cap Z'$.  

\item Par récurrence sur $n \in \N$, on dit que 
$\E$ est strictement $n+1$-dévissable en log-isocristaux surconvergents sur $\X ^{\sharp}$, 
s'il existe un triangle distingué dans 
$D ^{\mathrm{b}} _{\mathrm{coh}}( \smash{\D} ^\dag _{\X ^{\sharp}} (\hdag T) _{\Q} )$
de la forme
$$\E' \to \E \to \E'' \to \E' [1],$$
où $\E', \E''$ sont strictement $n$-dévissables en log-isocristaux surconvergents sur $\X ^{\sharp}$.
\end{itemize}

\end{defi}

\begin{nota}
\label{nota-preAbe}
On suppose que 
$\ZZ \cup \mathfrak{D}$ est un diviseur à croisements normaux strict de $\X$.
On note 
$\X ^{\flat}:= (\X , \ZZ \cup \mathfrak{D})$
le log-schéma formel logarithmique lisse,
$\alpha \colon \X ^{\flat} \to \X ^{\sharp}$ 
et
$\beta \colon \X ^{\sharp} \to \X $
les morphismes canoniques.

\end{nota}

\begin{rema}
\label{AC}
On garde les notations et hypothèses de \ref{nota-preAbe}.
D'après une remarque de \cite[4]{AC-weil2}, 
il existe un diagramme de la forme du carré de gauche: 
\begin{equation}
\xymatrix @R=0,4cm{
{\ZZ} 
\ar[r] _-{u '}
\ar@{=}[d] ^-{}
\ar@{}[dr]|\square
& 
{\X '} 
\ar@/_0.5pc/[l] _-{g '}
\ar[d] ^-{f}
\ar@{}[dr]|\square
&
{\X ^{\prime \sharp}} 
\ar[l] ^-{\beta '}
\ar[d] ^-{f ^{\sharp}}
\ar@{}[dr]|\square
\ar@/^0.5pc/[r] ^-{g ^{\prime \sharp}}
&
{\ZZ ^{\sharp}} 
\ar[l] ^-{u ^{\prime \sharp}}
\ar@{=}[d] 
\\ 
{\ZZ} 
\ar[r] _-{u}
& 
{\X} 
&
{\X ^{\sharp}} 
\ar[l] ^-{\beta }
&
{\ZZ ^{\sharp},} 
\ar[l] ^-{u ^{\sharp}}
 }
 \xymatrix @R=0,4cm{
{\X ^{\prime \sharp}} 
\ar[d] ^-{f ^{\sharp}}
\ar@{}[dr]|\square
&
{\X ^{\prime \flat}} 
\ar[l] ^-{\alpha '}
\ar[d] ^-{f ^{\flat}}
\\ 
{\X ^{\sharp}} 
&
{\X ^{\flat}} 
\ar[l] ^-{\alpha }
 }
\end{equation}
où $f$ est un morphisme étale, $g '$ est un morphisme lisse qui est une rétraction de $u'$ et tel que, 
en posant $\mathfrak{D}':= f ^{-1}\mathfrak{D}$, on ait  
$\mathfrak{D}'=g ^{\prime -1} \circ u ^{\prime -1} (\mathfrak{D}')$, 
où les trois autres carrés sont cartésiens par définition
(on remarque alors que tous les morphismes sont exacts).
\end{rema}

Le corollaire qui suit étend le théorème \cite[1.3.13]{caro-Tsuzuki}:
\begin{coro}
\label{1.3.13AM}
On garde les notations et hypothèses de \ref{nota-preAbe}.
Soit $\E$ un isocristal sur $\X ^{\flat}$ surconvergent le long de $T$ et
satisfaisant aux conditions 
(a) et (b) de \ref{conditions*}. 
Les propriétés suivantes sont alors satisfaites:  
\begin{enumerate}
\item Le complexe $ u ^{\sharp !} \circ \alpha _{+} (\E)$ est un complexe d'isocristaux sur $\ZZ ^{\sharp}$ surconvergents le long de $U$
 et satisfaisant aux conditions
(a) et (b) de \ref{conditions*}.
\item On dispose du triangle distingué dans $D ^{\mathrm{b}} _{\mathrm{coh}}( \smash{\D} ^\dag _{\X ^{\sharp}} ( T) _{\Q} )$ de la forme:
\begin{equation}
\notag
u ^{\sharp} _{+} \circ u ^{\sharp !} \circ \alpha _{+} (\E)
\to 
\alpha _{+} (\E)
\to 
\E (\hdag Z)
\to 
u ^{\sharp} _{+} \circ u ^{\sharp !} \circ \alpha _{+} (\E) [1].
\end{equation}
En particulier, 
$\alpha _{+} (\E)$ est strictement
$1$-dévissable en log-isocristaux surconvergents sur $\X ^{\sharp}$.
\item Si $\E$ vérifie en outre la condition (p) de \ref{conditions*}, 
alors le morphisme canonique 
 $\alpha _{+} (\E)
\to 
\E (\hdag Z)$ 
est un isomorphisme.
 
\end{enumerate}

\end{coro}

\begin{proof}
Grâce à la remarque \ref{AC}
et grâce à l'isomorphisme de changement de base par un morphisme lisse (voir \cite[5.4.6]{caro-stab-sys-ind-surcoh})
qui donne l'isomorphisme 
$f ^{\sharp!} \circ \alpha _{+} (\E) \riso
\alpha ' _+ \circ f ^{\flat!} (\E)$, 
on se ramène au cas où 
il existe un morphisme lisse 
$g \colon \X \to \ZZ$ tel que $g \circ u = id$ et 
$\mathfrak{D}=g ^{-1} \circ u ^{-1} (\mathfrak{D})$.
On note alors $g ^{\sharp}\colon \X ^{\sharp} \to \ZZ ^{\sharp}$
le morphisme induit par $g$
et
$g ^{\flat}:= g ^{\sharp}\circ u$.
Soit $\E ^{(\bullet)} \in  \smash{\underrightarrow{LM}} ^{\mathrm{b}} _{\mathbb{Q}, \mathrm{coh}} 
( \smash{\widehat{\mathcal{D}}} _{\mathfrak{X} ^{\flat}} ^{(\bullet)} (T))$ 
un objet tel que $\underrightarrow{\lim}~\E ^{(\bullet)} \riso \E$.
Comme 
$(\hdag Z) \circ \alpha ^{(\bullet)}  _{+} (\E^{(\bullet)} )
=
\E ^{(\bullet)}  (\hdag Z)$,
on dispose alors du triangle distingué dans 
$\smash{\underrightarrow{LD}} ^{\mathrm{b}} _{\mathbb{Q}, \mathrm{qc}} 
( \smash{\widehat{\mathcal{D}}} _{\mathfrak{X} ^{\sharp}} ^{(\bullet)} (T))$:
\begin{equation}
\label{1.3.13AM-tri1}
u ^{\sharp (\bullet)} _{+} \circ u ^{\sharp (\bullet)!} \circ \alpha ^{ (\bullet)} _{+} (\E^{(\bullet)} )
\to 
\alpha ^{(\bullet)}  _{+} (\E^{(\bullet)} )
\to 
\E ^{(\bullet)}  (\hdag Z)
\to 
u ^{\sharp (\bullet)} _{+} \circ u ^{\sharp (\bullet) !} \circ \alpha ^{ (\bullet)} _{+} (\E^{(\bullet)} ) [1].
\end{equation}
En appliquant le foncteur $g ^{\sharp (\bullet)} _{+}$ au triangle 
\ref{1.3.13AM-tri1}, 
comme 
$g ^{\sharp (\bullet)} _{+} \circ u ^{\sharp (\bullet)} _{+} =id$, 
comme 
$g ^{\sharp (\bullet)} _{+} \circ \alpha ^{ (\bullet)} _{+} =
g ^{\flat (\bullet)} _{+}$,
on obtient alors le triangle distingué dans 
$\smash{\underrightarrow{LD}} ^{\mathrm{b}} _{\mathbb{Q}, \mathrm{qc}} 
( \smash{\widehat{\mathcal{D}}} _{\ZZ ^{\sharp}} ^{(\bullet)} (U))$:
\begin{equation}
\label{1.3.13AM-tri2}
u ^{\sharp (\bullet)!} \circ \alpha ^{ (\bullet)} _{+} (\E^{(\bullet)} ) 
\to 
g ^{\flat (\bullet)} _{+} (\E^{(\bullet)} )
\to 
g ^{\flat (\bullet)} _{+} (\E ^{(\bullet)}  (\hdag Z))
\to 
u ^{\sharp (\bullet) !} \circ \alpha ^{ (\bullet)} _{+} (\E^{(\bullet)} ) [1].
\end{equation}
En appliquant le foncteur $ \underrightarrow{\lim}  $ à 
\ref{1.3.13AM-tri2}, 
comme 
$\alpha ^{ (\bullet)} _{+} (\E^{(\bullet)} ) 
\in  \smash{\underrightarrow{LM}} ^{\mathrm{b}} _{\mathbb{Q}, \mathrm{coh}} 
( \smash{\widehat{\mathcal{D}}} _{\mathfrak{X} ^{\sharp}} ^{(\bullet)} (T))$,
avec \cite[5.1.7]{caro-stab-sys-ind-surcoh},
on obtient: 
\begin{equation}
\label{1.3.13AM-tri3}
u ^{\sharp!} \circ \alpha  _{+} (\E ) 
\to 
g ^{\flat} _{+} (\E)
\to 
g ^{\flat} _{+} (\E   (\hdag Z))
\to 
u ^{\sharp!} \circ \alpha  _{+} (\E )  [1].
\end{equation}
Grâce au théorème \cite[1.3.13]{caro-Tsuzuki},
on déduit de \ref{1.3.13AM-tri3} que 
$u ^{\sharp!} \circ \alpha  _{+} (\E )$
est un complexe d'isocristaux sur $\ZZ ^{\sharp}$ surconvergents le long de $U$
 et satisfaisant aux conditions
(a) et (b) de \ref{conditions*}.
Il résulte alors de \ref{coro-theo-u^*coh}, de \ref{5assert-eq} 
que $u ^{\sharp (\bullet)!} \circ \alpha ^{ (\bullet)} _{+} (\E^{(\bullet)} ) \in
\smash{\underrightarrow{LD}} ^{\mathrm{b}} _{\mathbb{Q}, \mathrm{coh}} 
( \smash{\widehat{\mathcal{D}}} _{\ZZ ^{\sharp}} ^{(\bullet)} (U))$. 
Le triangle \ref{1.3.13AM-tri1} est donc un triangle  distingué de 
$\smash{\underrightarrow{LD}} ^{\mathrm{b}} _{\mathbb{Q}, \mathrm{coh}} 
( \smash{\widehat{\mathcal{D}}} _{\mathfrak{X} ^{\sharp}} ^{(\bullet)} (T))$.
On en déduit la partie 2) du corollaire.
Lorsque 
$\E$ satisfait aux conditions 
(a),  (b) et (p) de \ref{conditions*}
 il découle de  \cite[1.1.22.1]{caro-Tsuzuki} que 
$ g ^{\flat} _{+} (\E)
\to 
g ^{\flat} _{+} (\E   (\hdag Z))$
est un isomorphisme (cela se vérifie comme \cite[1.3.13]{caro-Tsuzuki}
découle de \cite[1.1.22.1]{caro-Tsuzuki}).
Via \ref{1.3.13AM-tri3}, cela entraîne que 
$u ^{\sharp!} \circ \alpha  _{+} (\E ) =0$ et donc 
$u ^{\sharp} _{+} \circ u ^{\sharp !} \circ \alpha _{+} (\E)=0$.
On a ainsi validé la partie 3) du corollaire.

\end{proof}

Le corollaire qui suit étend \cite[2.2.9]{caro-Tsuzuki}
et 
\cite[2.3.4]{caro-Tsuzuki}:
\begin{coro}
Supposons pour simplifier les énoncés que $\X$ soit intègre. 
Soient $\mathfrak{D} _1, \dots, \mathfrak{D} _s$ les composantes irréductibles de $\mathfrak{D}$. 
Soit $1\leq r\leq s$, $\mathfrak{D}'':= \cup _{1 \leq i \leq r} ~\mathfrak{D} _i$,
$\mathfrak{D}':= \cup _{r +1 \leq i \leq s} ~\mathfrak{D} _i$,
$\X ^{\flat}:= (\X, \mathfrak{D}')$, 
$\alpha \colon \X ^{\sharp} \to \X ^{\flat}$ le morphisme canonique.
Soit $\E$ un isocristal sur $\X ^{\sharp}$ surconvergent le long de $T$.
\begin{enumerate}
\item Si $\E$ vérifie les conditions (a) et (b) de \ref{conditions*}, alors
$\alpha _{+} (\E)$ est 
strictement
$r$-dévissable en log-isocristaux surconvergents sur $\X ^{\flat}$ 
satisfaisant aux condition (a) et (b) de \ref{conditions*}.
\item Si $\E$ vérifie les condition (a), (b) et (p) de \ref{conditions*}, le morphisme canonique
$\alpha _{+} (\E) \to \E (\hdag D'')$ est un isomorphisme. 

\end{enumerate}

\end{coro}

\begin{proof}
Vérifions d'abord 1). 
On procède par récurrence sur $r$. Le cas où $r =1$ est le corollaire
\ref{1.3.13AM}. Supposons donc $r \geq 2$. 
Posons $\ZZ := \mathfrak{D}  _1$, 
$\mathfrak{D}  _{\geq 2}:= \cup _{2 \leq i \leq s} ~\mathfrak{D} _i$,
$\X ^{\sharp \flat}:= (\X, \mathfrak{D} _{\geq 2})$,
$\ZZ ^{\sharp \flat} := (\ZZ, \ZZ \cap \mathfrak{D} _{\geq 2})$,
$\ZZ ^{\flat} := (\ZZ, \ZZ \cap \mathfrak{D} ')$.
On note 
$\alpha _1\colon \X ^{\sharp} \to \X ^{\sharp \flat}$, 
$\alpha ' \colon \X ^{\sharp \flat} \to \X ^{\flat}$, 
$\beta '\colon \ZZ ^{\sharp \flat} \to \ZZ ^{\sharp}$, 
$u ^{\sharp \flat} \colon \ZZ ^{\sharp \flat} \to \X ^{\sharp \flat}$
et
$u ^{ \flat} \colon \ZZ ^{\flat} \to \X ^{\flat}$
les morphismes canoniques. 
Lorsque $T \supset Z$, on calcule facilement (voir par exemple \cite[2.2.1]{caro-Tsuzuki})
que $ \alpha _{1+} (\E)= \E$.
Supposons alors que $T \cap Z$ soit un diviseur de $Z$.
D'après \ref{1.3.13AM}, on dispose alors du triangle distingué de localisation
dans $D ^{\mathrm{b}} _{\mathrm{coh}}( \smash{\D} ^\dag _{\X ^{\sharp \flat }} ( T) _{\Q} )$
\begin{equation}
\notag
u ^{\sharp \flat} _{+} \circ u ^{\sharp \flat !} \circ \alpha _{1+} (\E)
\to 
\alpha _{1+} (\E)
\to 
\E (\hdag Z)
\to 
u ^{\sharp \flat } _{+} \circ u ^{\sharp \flat  !} \circ \alpha _{1+} (\E) [1],
\end{equation}
tel que 
 $\FF := u ^{\sharp \flat !} \circ \alpha _{1+} (\E)$
 soit un complexe d'isocristaux sur $ \ZZ ^{\sharp \flat} $
surconvergent le long de $T \cap Z$
satisfaisant aux conditions (a) et (b) de \ref{conditions*}.
Or, par hypothèse de récurrence (et aussi parce que le foncteur $\beta ' _+$ est
exact sur la catégorie des isocristaux sur 
$ \ZZ ^{\sharp \flat} $
surconvergent le long de $T \cap Z$), 
$\alpha ' _+ (\E (\hdag Z _1))$ 
(resp. 
$\beta ' _+(\FF)$)
est 
strictement
$r-1$-dévissable en log-isocristaux surconvergents sur $\X ^{\flat}$
(resp. $\ZZ ^{\flat}$).
L'isomorphisme canonique 
$\alpha ' _+ \circ u ^{\sharp \flat} _{+} ( \FF )
\riso 
u ^{\flat} _{+} \circ \beta ' _+
(\FF)$
nous permet de conclure 1).
On établit 2) de manière identique.
\end{proof}

\bibliographystyle{smfalpha}

\begin{thebibliography}{BGR84}

\bibitem[AC12]{AC-weil2}
{\scshape T.~Abe {\normalfont \smfandname} D.~Caro} -- {\og {$p$-adic Weil II
  for arithmetic $\mathcal{D}$-modules}\fg},  (2012).

\bibitem[Ber96]{Be1}
{\scshape P.~Berthelot} -- {\og ${\mathcal{d}}$-modules arithm\'etiques. {I}.
  {O}p\'erateurs diff\'erentiels de niveau fini\fg}, \emph{Ann. Sci. \'Ecole
  Norm. Sup. (4)} \textbf{29} (1996), no.~2, p.~185--272.

\bibitem[Ber02]{Beintro2}
\bysame , {\og {Introduction \`a la th\'eorie arithm\'etique des
  {$\mathcal{D}$}-modules}\fg}, \emph{Ast\'erisque} (2002), no.~279, p.~1--80,
  Cohomologies {$p$}-adiques et applications arithm\'etiques, {II}.

\bibitem[BGR84]{bosch}
{\scshape S.~Bosch, U.~G{\"u}ntzer {\normalfont \smfandname} R.~Remmert} --
  \emph{Non-{A}rchimedean analysis}, Springer-Verlag, Berlin, 1984, A
  systematic approach to rigid analytic geometry.

\bibitem[Car12]{caro-stab-sys-ind-surcoh}
{\scshape D.~Caro} -- {\og {Syst{\`e}mes inductifs surcoh{\'e}rents de
  $\mathcal{D}$-modules arithm{\'e}tiques}\fg}, \emph{ArXiv Mathematics
  e-prints} (2012).

\bibitem[CT12]{caro-Tsuzuki}
{\scshape D.~Caro {\normalfont \smfandname} N.~Tsuzuki} -- {\og
  {{Overholonomicity of overconvergent $F$-isocrystals over smooth
  varieties}}\fg}, \emph{Annals of Math.} (2012).

\bibitem[Mat89]{matsumura}
{\scshape H.~Matsumura} -- \emph{Commutative ring theory}, second \smfedname,
  Cambridge Studies in Advanced Mathematics, vol.~8, Cambridge University
  Press, Cambridge, 1989, Translated from the Japanese by M. Reid.

\bibitem[Mon02]{these_montagnon}
{\scshape C.~Montagnon} -- {\og {G{\'e}n{\'e}ralisation de la th{\'e}orie
  arithm{\'e}tique des $\mathcal{D}$-modules {\`a} la g{\'e}om{\'e}trie
  logarithmique}\fg}, \smfphdthesisname, Universit{\'e} de {R}ennes {I}, 2002.

\bibitem[Sch02]{Schneider-NonarchFuncAn}
{\scshape P.~Schneider} -- \emph{Nonarchimedean functional analysis}, Springer
  Monographs in Mathematics, Springer-Verlag, Berlin, 2002.

\end{thebibliography}
\providecommand{\bysame}{\leavevmode ---\ }
\providecommand{\og}{``}
\providecommand{\fg}{''}
\providecommand{\smfandname}{et}
\providecommand{\smfedsname}{\'eds.}
\providecommand{\smfedname}{\'ed.}
\providecommand{\smfmastersthesisname}{M\'emoire}
\providecommand{\smfphdthesisname}{Th\`ese}

\bigskip
\noindent Daniel Caro\\
Laboratoire de Mathématiques Nicolas Oresme\\
Université de Caen
Campus 2\\
14032 Caen Cedex\\
France.\\
email: daniel.caro@unicaen.fr

\end{document}